\documentclass[11pt]{article}
\usepackage{amsmath,mathtools}
\usepackage{graphicx}
\usepackage{epsfig,epsf,psfrag}
\usepackage{amsfonts}
\usepackage{setspace}
\usepackage{color}
\usepackage[font=small]{caption}
\usepackage[font=footnotesize]{subcaption}
\usepackage[top=1.1in, bottom=1.1in, left=1.2in, right=1.2in]{geometry}
\usepackage{afterpage}
\usepackage{url}
\usepackage{lmodern}
\usepackage[T1]{fontenc}

\usepackage[bf,sf,medium]{titlesec}

\usepackage[sort,compress]{cite}

\usepackage{bm}
\usepackage{isomath}

\newcommand{\vct}{\vectorsym}
\newcommand{\mtx}{\matrixsym}

\newcommand{\er}{\vectorsym{e}_r}
\newcommand{\et}{\vectorsym{e}_{\theta}}
\newcommand{\ez}{\vectorsym{e}_z}

\newcommand{\bn}{\vectorsym{n}}
\newcommand{\bt}{\vectorsym{t}}
\newcommand{\bx}{\vectorsym{x}}
\newcommand{\by}{\vectorsym{y}}

\newcommand{\be}{\vectorsym{E}^{\text{tot}}}
\newcommand{\beinc}{\vectorsym{E}^{\text{inc}}}
\newcommand{\besc}{\vectorsym{E}}
\newcommand{\besce}{\vectorsym{E}_0}
\newcommand{\besci}{\vectorsym{E}_1}

\newcommand{\bh}{\vectorsym{H}^{\text{tot}}}
\newcommand{\bhinc}{\vectorsym{H}^{\text{inc}}}
\newcommand{\bhsc}{\vectorsym{H}}
\newcommand{\bhsce}{\vectorsym{H}_0}
\newcommand{\bhsci}{\vectorsym{H}_1}

\newcommand{\bJ}{\vectorsym{J}} 
\newcommand{\bM}{\vectorsym{M}}

\newcommand{\cK}{\mathcal{K}}
\newcommand{\cS}{\mathcal{S}}
\newcommand{\cN}{\mathcal{N}}

\newcommand{\lp}{\left(}
\newcommand{\rp}{\right)}

\renewcommand{\phi}{\varphi}

\newtheorem{theorem}{\sffamily Theorem}[section]
\newtheorem{lemma}[theorem]{\sffamily Lemma}

\newtheorem{remark}{\sffamily Remark}

\newenvironment{proof}[1][Proof]
{\begin{trivlist} \item[\hskip \labelsep {\bfseries #1}]}{\end{trivlist}}

\title{\bfseries \sffamily An FFT-accelerated direct solver for
  electromagnetic scattering from 
  penetrable axisymmetric objects} 

\numberwithin{equation}{section}

\author{ Jun Lai\footnote{School of Mathematical Sciences, Zhejiang
    University, Hangzhou, Zhejiang. Research was supported in part by
    the Funds for Creative Research Groups of NSFC (No. 11621101), the
    Major Research Plan of NSFC (No. 91630309), NSFC grant
    No. 11871427 and The Fundamental Research Funds for the Central
    Universities., Email: {\tt laijun6@zju.edu.cn}} \, and Michael
  O'Neil\footnote{Courant Institute, New York University, New York,
    NY.  Research was supported in part by the Office of Naval
    Research under award numbers \#N00014-17-1-2059 and
    \#N00014-17-1-2451. Email: {\tt oneil@cims.nyu.edu}. } }

\date{\today}

\begin{document}

\maketitle 
\begin{abstract}
  Fast, high-accuracy algorithms for electromagnetic scattering from
  axisymmetric objects are of great importance when modeling physical
  phenomena in optics, materials science (e.g. meta-materials), and
  many other fields of applied science.  In this paper, we develop an
  FFT-accelerated separation of variables solver that can be used to
  efficiently invert integral equation formulations of Maxwell's
  equations for scattering from axisymmetric penetrable (dielectric)
  bodies. Using a standard variant of M\"uller's integral
  representation of the fields, our numerical solver rapidly and
  directly inverts the resulting second-kind integral equation.  In
  particular, the algorithm of this work (1)~rapidly evaluates the
  modal Green's functions, and their derivatives, via kernel splitting
  and the use of novel recursion formulas, (2) discretizes the
  underlying integral equation using generalized Gaussian quadratures
  on adaptive meshes, and (3) is applicable to geometries containing
  edges and points.  Several numerical examples are provided to
  demonstrate the efficiency and accuracy of the aforementioned
  algorithm in various geometries.

  \vspace{\baselineskip}
\end{abstract}

%\onehalfspacing

\section{Introduction}
\label{sec_intro}

While many scattering problems in computational electromagnetics
require the solution to Maxwell's equations in arbitrary complex
geometries (e.g. radar scattering from aircraft, capacitance
extraction, etc.), it is often useful to study the same scattering
problems in somewhat simpler geometries, namely axisymmetric
ones. This problem, of computing scattered waves in axisymmetric
geometries, has a very rich history in the electrical engineering
community~\cite{andreasen1965,mautz1969bor,mautz1977,MH1978,gedney_1990,Geng1999Wide,Yu2008Closed},
and recently several groups have built specialized high-order solvers
for particular applications in plasma physics~\cite{oneil2018},
resonance calculations~\cite{HK15,Helsing2016,helsing2015ieee}, and
utilizing novel integral representations~\cite{epstein_2018}.  Decades
ago, the computation of radar cross sections and scattering phenomena
in axisymmetric geometries was popular, in part, due to the very
limited computational resources available at the time: the separation
of variables procedure reduced the dimension of the problem,
effectively reducing the number of unknowns needed for boundary
integral equation discretizations by a square-root-factor. The schemes
were mostly low-order Galerkin-based (i.e. method-of-moments), and
often only obtained a modest engineering precision (if accuracies were
reported at all).  In this work, we address the problem of scattering
from homogeneous penetrable axisymmetric bodies which may contain
edges, and therefore high-order adaptive discretizations of the
geometry are required. Our algorithm is based on solving M\"uller's
integral equation using a Nystr\"om-like discretization scheme.  Using
modern developments in numerical analysis and quadrature, our solver
is, in most cases, easily able to obtain 10-digits of accuracy using a
small number of unknowns located along the boundary of a
two-dimensional cross section of the scatterer. We now briefly
introduce the time harmonic transmission problem for Maxwell's
equations, and discuss earlier work in the field.

In regions free of electric charge and current, with electric
permittivity~$\epsilon$ and magnetic permeability~$\mu$, the time harmonic
Maxwell's equations can be reduced to:
\begin{equation}\label{eq:maxwell}
  \begin{aligned}
    \nabla \times \besc &= i\omega \mu \bhsc, &\qquad 
    \nabla \times \bhsc &= - i\omega \epsilon \besc, \\
    \nabla \cdot \besc &= 0, &\qquad 
    \nabla \cdot \bhsc &= 0,
  \end{aligned}
\end{equation}
where~$\omega$ denotes the angular frequency; the time dependence on
the fields of~$e^{-i\omega t}$ has been suppressed. It will also be
useful to define the wavenumber~$k = \omega \sqrt{\epsilon \mu}$.  In
the most general case, the material parameters are allowed to be
spatially dependent tensors~\cite{imbert2018}.

There are two canonical boundary value problems in classical
electromagnetics, that of scattering from perfect electric conductors
(PECs) and scattering in non-conducting (dielectric, or penetrable)
materials with piecewise constant material properties. In this work,
we will focus on the latter scattering problem. This problem, as well
as the first, offers a surprising regime in which the mathematical and
physics model (i.e. Maxwell's equations) is very well understood
\emph{and} agrees very closely with experimental phenomena.  With this
in mind, it is useful to construct numerical methods which accurately
solve the underlying equations in order to complement, or partially
replace, experimental design methods.

To be more precise regarding the formulation of the PDE we are
focusing on, denote by~$\Omega$ a closed bounded object in
$\mathbb R^3$ with boundary~$\Gamma$ and constant material
parameters~$\epsilon_1$,~$\mu_1$.  Let the background be denoted by
$\Omega_0 = \overline{\mathbb R^3 \setminus \Omega}$, also a closed
region with the same boundary~$\Gamma$. The background material
parameters will be given by~$\epsilon_0$, $\mu_0$.  See
Figure~\ref{figure1}.  Considering Maxwell's
equations~\eqref{eq:maxwell} in this two-component geometry,
we will assume that $\omega\geq 0$,
$\epsilon_0, \mu_0 > 0$, and that
\begin{equation}\label{eq:materials}
  \begin{aligned}
    \Re{\epsilon_1} > 0 & \text{ and }  \Im{\epsilon_1} \geq 0 , \\ 
    \Re{\mu_1} > 0 & \text{ and } \Im{\mu_1} \geq 0 .
  \end{aligned}
\end{equation}
Existence and uniqueness results for a slightly more general set of
parameters, namely for those with $\Re{\epsilon_1} <0$ and $\Re{\mu_1}
<0$, can be found in~\cite{vico2018cpde} and~\cite[\S 20]{MULLER}.
This parameter regime covers all classical materials, as well as
metamaterials with non-zero dissipation~\cite{vico2018cpde}. Our
selection of parameters is a subset of this more general case, and
therefore all existence and uniqueness results still hold.

Furthermore, many problems of considerable interest in real-world
phenomena (e.g. computing radar cross sections) take the form of
\emph{scattering problems}. In this setup, the total electromagnetic
field~$(\be,\bh)$ is the
sum of two pieces: an incident field~$(\beinc, \bhinc)$ and a
scattered field~$(\besc,\bhsc)$.  Both the incoming and scattered
field are assumed to satisfy Maxwell's equations, and therefore so
does the total field.
It is also further assumed that the scattered field~$(\besc,\bhsc)$ must satisfy
the Silver-M\"uller radiation condition at infinity~\cite{Cot2}:
\begin{equation}\label{radbdy}
  \lim_{|\bx|\rightarrow \infty} \frac{1}{|\bx|}
  \lp \frac{\bx}{|\bx|}\times \sqrt{\varepsilon_0} \, \besc(\bx) -
    \sqrt{\mu_0} \, \bhsc(\bx)
    \rp = 0.  
\end{equation}  
The physical boundary
conditions along an interface between two pieces of penetrable media
state that the tangential components of the total field are
continuous across~$\Gamma$:
\begin{equation}
  \label{bdycond}
  [\bn\times \be] = 0, \qquad [\bn\times
  \bh] = 0, \qquad \text{on } \Gamma.
\end{equation}
Here,~$[\cdot]$ denotes the jump across~$\Gamma$ and~$\bn$ is the unit
normal along~$\Gamma$ that points into the unbounded
region~$\Omega_0$.
A proof of uniqueness to this boundary value
problem~(i.e. equations~\eqref{eq:maxwell} and~\eqref{radbdy} with
boundary condition~\eqref{bdycond}) with the
aforementioned material parameter constraints~\eqref{eq:materials} is given
in~\cite{vico2018cpde}, which is simpler to parse than the original
treatment by M\"uller~\cite[\S 21, Thm 61]{MULLER}.
We will refer to this as the Transmission Boundary Value
Problem (TBVP).

\begin{figure}[tbp]
  \begin{center}
    \includegraphics[width=.35\linewidth]{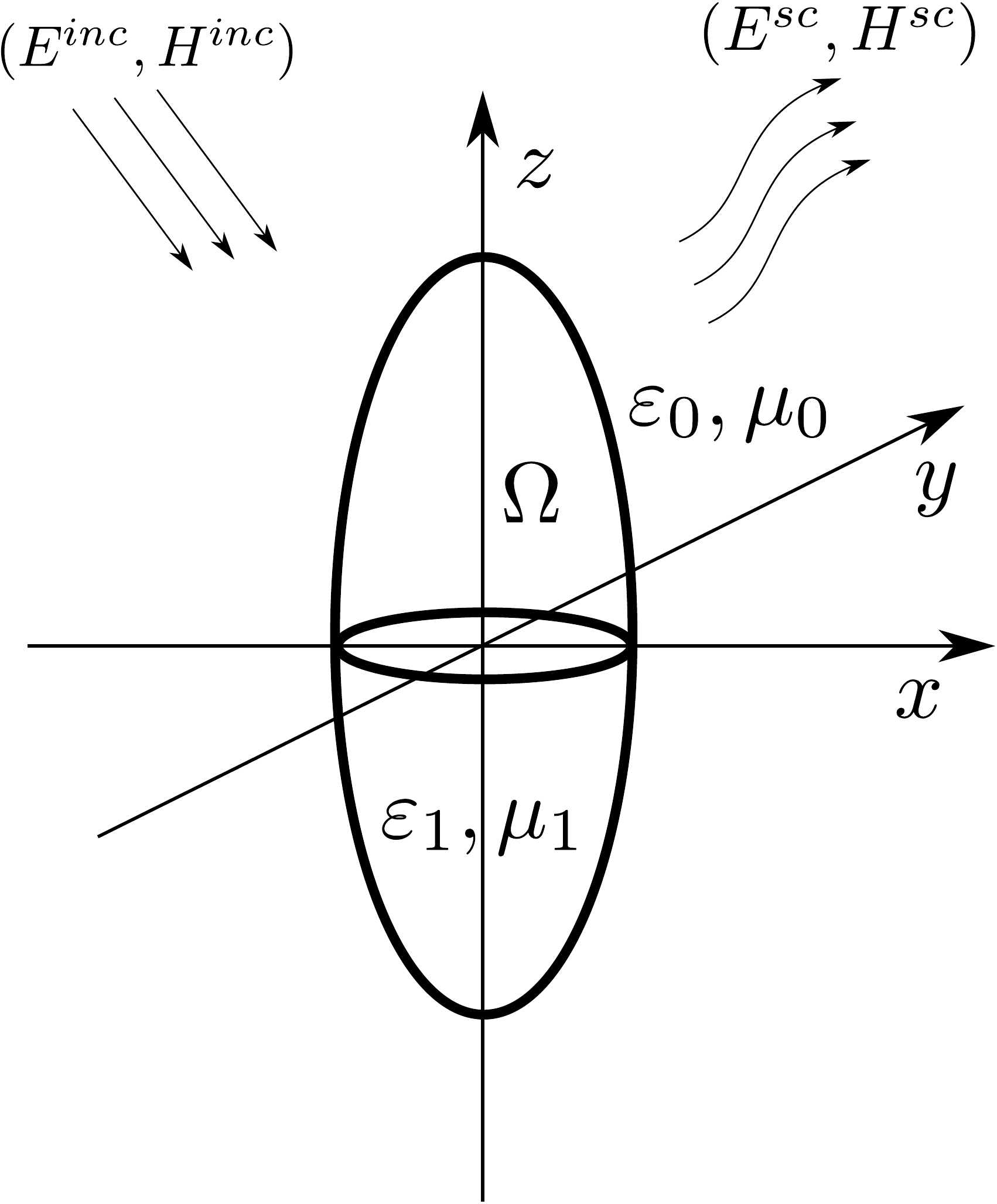}
    \caption{An axisymmetric object $\Omega$ is illuminated by a time
      harmonic incident wave $(\beinc, \bhinc)$. The object is
      penetrable with dielectric coefficients
      $(\varepsilon_1, \mu_1)$. The background material is homogeneous
      with dielectric coefficients
      $(\varepsilon_0, \mu_0)$.}\label{figure1}
  \end{center}
\end{figure} 

Scattering from perfectly conducting bodies of revolution using
integral equation methods appeared as earlier as
1965~\cite{andreasen1965}, and we believe that the transmission
boundary value problem was first directly treated in axisymmetric
domains via an integral equation method (i.e. M\"uller's formulation)
in 1966~\cite{vasilev1966}. Details of this approach, along with
Fortran code, were then collected in a technical report by Mautz and
Harrington in 1977~\cite{mautz1977}.  Since then, many groups in the
electrical engineering community have studied scattering from bodies
of revolution of various materials (e.g. from perfectly conducting
metals to inhomogeneous penetrable materials).  Some works addressed
the issue of generating matrix elements of the decouple integral
equations~\cite{gedney_1990}, and others focused on integral equation
formulations for more complicated
structures~\cite{morgan1979,medgyesi1984}.  Nearly all methods are
based on a Galerkin discretization of the relevant boundary integral
equation, frequently referred to as the boundary element
method~(BEM)~\cite{banerjee1981} in some mathematics communities and
the method of moments~(MoM)~\cite{harrington1968} in most engineering
circles.  There are at most a handful of works in the engineering
literature which have implemented higher-order methods for these
problems. Notably, a 3rd-order hybrid finite-element/boundary integral
code was presented in~\cite{dunn2006} in 2006. This code still had
visible errors in the monostatic radar cross section of a sphere when
compared with the exact Mie-series solution~\cite[Fig. 3]{dunn2006}.

However, over the past decade, integral equation methods for boundary
value problems in axisymmetric geometries have received a lot of
attention in applied and computational mathematics, likely due to the
increase in computational power (thereby enabling high-resolution
experiments to be performed on a desktop or laptop) and the persistent
mathematical and computational difficulties of designing high-order
methods in general, complex geometries (e.g. quadrature design,
geometry generation, etc.). Axisymmetric solvers based on a separation
of variables approach, as in this work, provide a robust test-bed for
the same integral equations which are used in general geometries, not
to mention that several axisymmetric geometries are of real-word
physical importance.  Advances in Nystr\"om discretization and
quadrature were applied to this problem in 2004~\cite{fleming2004},
but the first recent, high-order treatments of scattering
(electrostatic and acoustic) from axisymmetric objects were in
2010~\cite{young2010axi} and 2012~\cite{YHM2012}.  Since then, various
high-order accurate separation of variable methods for integral
equation formulations in axisymmetric geometries have been developed
for boundary value problems for Laplace's equation, the Helmholtz
equation~\cite{helsing2014,helsing2015,Liu2015}, and more recently,
Maxwell's equations~\cite{helsing2015ieee,HK15,epstein_2018,Helsing2016} using a
variety of discretization schemes, solvers, and methods for evaluating
the so-called modal Green's functions. The previous discussion is a
quite complete picture of the state of integral
equation-based solvers for boundary value problems in axisymmetric
domains. The only work relevant to high-order methods for dielectric
bodies of revolution is~\cite{Helsing2016}, but that work is
rather  terse and focused on
computing  eigenvalues/fields.  On a related note, 
however, that it has been shown that high-order
methods based on boundary perturbations and separation of variables
(of the solution to the PDE) can be made quite
efficient~\cite{nicholls2006sisc, fang2007jcp, bruno1998jasa}, and of
course, high-order methods in general geometries is always a work in
progress~\cite{Bremer2012,bruno2009krylov,bruno2003fast}.

It is also worth mentioning some details of related work: recently, a
solver for both PEC and dielectric scattering problem based on the
generalized Debye source formulation~\cite{EG10,EG13} was described
in~\cite{epstein_2018}. This solver, while also of high-order and
stable for all frequencies (including the low-frequency limit
$\omega\to 0$), is currently restricted to globally smooth geometries.
The integral equation method of this paper, based on the classic one
due to M\"uller, is also free from spurious resonances and is of
Fredholm second-kind on smooth geometries (and well-conditioned on
those with edges)  under
  condition~\eqref{eq:materials}. We do not address the situation in
which~$\omega\to 0$. The standard integral representations based on
M\"uller's formulation for penetrable media include terms which
are~$\mathcal O(\omega^{-1})$, and therefore require some care in the
static limit.

To summarize our contribution, the integral equation formulation and
solver of this work have three main features: (1) a novel method for
evaluating higher derivatives of the modal Green's functions based on
kernel splitting and recurrence relations, (2) an adaptive
discretization of the generating curve based on generalized Gaussian
quadratures, and (3) an integral equation formulation and
discretization scheme which is compatible with geometries that contain
edges. We focus our attention on M\"uller's integral equation for the
TBVP, but the discretization methods are applicable to integral
equations arising in other fields with Green's functions exhibiting
similar singularity behavior.

The paper is organized as follows: Section~\ref{sec_form}
introduces the M\"uller integral representation, and its indirect form,
for electromagnetic fields in piecewise constant penetrable media.
Section~\ref{sec_fourier} details the transformation of a surface
integral equation along an axisymmetric surface into a sequence of
decoupled integral equations along a curve in two dimensions using the
Fourier transform in the azimuthal direction.  Details of the fast
kernel evaluation based on a kernel splitting technique and recursion
formulas are given in Section~\ref{kerneleval}. Section~\ref{ggq}
discusses discretization of the sequence of
line integrals using an adaptive mesh and generalized Gaussian
quadratures for the associated weakly singular integral operators.
Numerical examples are given in
Section~\ref{sec_numeri}, including scattering results in both smooth and
non-smooth geometries.
Section~\ref{sec_con} concludes the discussion with drawbacks,
observations, and future research.

\section{Integral equation formulations}
\label{sec_form}

In what follows, we will denote the exterior scattered field by
$(\besce,\bhsce)$ and the interior scattered field by
$(\besci,\bhsci)$. It will also be assumed that the incoming field is
generated from sources exterior to~$\Omega$, i.e. in the background,
and therefore inside~$\Omega$ we simply
have $\besci=\be$ and $\bhsci=\bh$. Therefore the jump
condition~\eqref{bdycond} along~$\Gamma$ can be written as
\begin{equation}
  \label{equvalbdy}
  \begin{aligned}
    \bn\times (\beinc+ \besce) &= \bn\times \besci,  \\
    \bn\times (\bhinc+ \bhsce) &= \bn\times \bhsci.
  \end{aligned}
\end{equation}
We now turn to a derivation of the M\"uller's integral equation, and its
indirect formulation.

\subsection{The M\"uller integral equation}

As discussed in the introduction, the goal of this paper is to design
an efficient numerical solver for the time-harmonic Maxwell's
equations, given in~\eqref{eq:maxwell}, with transmission boundary
conditions~\eqref{equvalbdy} along the surface of axisymmetric
objects.  The same boundary value problem is also referred to as
scattering from piecewise homogeneous penetrable, or dielectric,
media. Using a properly formulated integral equation method
automatically ensures that the scattered field obeys the radiation
condition~\eqref{radbdy} at infinity~\cite{Cot2,chew2009book}.  We
start by defining the single-layer vector potential. Let~$\bJ$ be a
vector field supported along the surface~$\Gamma$. Then the
single-layer potential due to~$\bJ$ is given by
\begin{equation}
  \cS^k \bJ(\bx) = \int_{\Gamma} G^k(\bx,\by) \, \bJ(\by) \, da(\by),
\end{equation}
where it is assumed that $\bx \notin \Gamma$ and~$da$ is the area
element along~$\Gamma$.
Here, the function
$G^k(\bx,\by)$ is the free space Green's function for the three
dimensional Helmholtz equation:
\begin{equation}
  \label{greenfun}
  G^k(\bx,\by) = \frac{e^{ik|\bx-\by|}}{4\pi|\bx-\by|}.
\end{equation}
For~$\bx \in \Gamma$, the integral operator~$\cS^k$ is
weakly-singular and continuous across~$\Gamma$~\cite{Cot2}.
We now derive an integral equation for the TBVP based on what is
referred to as the \emph{direct method}.
The Stratton-Chu formulation~\cite{KH2015} provides a Green's-like
 reproducing formula for the incoming field in~$\Omega$
 using tangential traces of the fields on~$\Gamma$. Let~$k_0$ and~$k_1$
 be the wavenumbers in the interior and exterior of $\Omega$,
 respectively.
 Inside~$\Omega$, the
incident field~$(\beinc,\bhinc)$ satisfies
\begin{equation}
  \label{equinc}
  \begin{aligned}
    i\omega\varepsilon_0 \, \beinc &= \nabla \times \nabla \times \cS^{k_0}
    \lp \bn\times \bhinc \rp - i\omega\varepsilon_0\nabla
    \times \cS^{k_0}
    \lp \bn\times \beinc\rp, \\
    i\omega\mu_0 \, \bhinc &= -\nabla \times \nabla \times \cS^{k_0}
    \lp\bn\times \beinc \rp - i\omega\mu_0\nabla \times \cS^{k_0}
    \lp  \bn\times
   \bhinc\rp.
\end{aligned}
\end{equation}
Furthermore, according to the Extinction Theorem~\cite{MULLER}, in the
\emph{interior}~$\Omega$, the exterior scattered
field~$(\besce,\bhsce)$ vanishes:
\begin{equation}
  \label{equext}
  \begin{aligned}
    0 &= \nabla \times \nabla \times \cS^{k_0} \lp\bn\times \bhsce \rp
    - i\omega\varepsilon_0\nabla \times \cS^{k_0} \lp\bn\times \besce\rp, \\
    0 &= -\nabla \times \nabla \times \cS^{k_0} \lp\bn\times \besce\rp
    - i\omega\mu_0\nabla \times \cS^{k_0} \lp\bn\times \bhsce\rp,
  \end{aligned}
\end{equation}
and likewise, in the \emph{exterior} (the background)~$\Omega_0$, the
interior scattered field $(\besci,\bhsci)$ vanishes:
\begin{equation}
  \label{equint}
  \begin{aligned}
    0 &= -\nabla \times \nabla \times \cS^{k_1} (\bn\times \bhsci)
    + i\omega\varepsilon_1\nabla \times \cS^{k_1} (\bn\times \besci), \\
    0 &= \nabla \times \nabla \times \cS^{k_1} (\bn\times \besci) +
    i\omega\mu_1\nabla \times \cS^{k_1} (\bn\times \bhsci).
  \end{aligned}
\end{equation}
In order to derive an integral equation along the boundary~$\Gamma$,
we now let~$\bx$ approach the boundary~$\Gamma$ from the interior for
equations~\eqref{equinc} and~\eqref{equext}, and from the exterior for
equation~\eqref{equint}.  Taking the limit of the tangential
components of both sides of equations~\eqref{equinc}--\eqref{equint},
where the tangential direction is defined on a parallel surface with
respect to~$\Gamma$~\cite{Ned01}, and using the jump property of
boundary integral operators~\cite{Cot2} and the transmission
condition~\eqref{equvalbdy}, we obtain the following boundary
integral equation:
\begin{equation}
  \label{mullerrep}
  \begin{aligned}
    i\omega\varepsilon_0 \, \bn \times\beinc &=
    \frac{i\omega}{2}(\varepsilon_0+\varepsilon_1)\bM
    + \lp\cK^{k_0} - \cK^{k_1}\rp\bJ - i \omega \lp\varepsilon_0
    \cN^{k_0} -
    \varepsilon_1\cN^{k_1} \rp \bM, \\
    i\omega\mu_0 \, \bn \times \bhinc &= \frac{i\omega}{2}(\mu_0+\mu_1)\bJ
    -\lp \cK^{k_0} - \cK^{k_1} \rp \bM - i\omega \lp \mu_0\cN^{k_0}
    - \mu_1\cN^{k_1} \rp \bJ,
\end{aligned}
\end{equation}
where we have set
\begin{equation}\label{emcurrent}
\bJ = \bn\times \bh, \qquad \bM = \bn\times \be,
\end{equation}
and the boundary-to-boundary
layer potential operators~$\cK^k$ and~$\cN^k$ are
defined as
\begin{align}
\cK^k\bJ &=  \bn\times\nabla \times \nabla \times \cS^{k} \bJ,\label{efieoper}\\
\cN^k\bJ &=  \bn\times \nabla \times \cS^{k} \bJ. \label{mfieoper}
\end{align}
Both of the operators~$\cK^k$ and~$\cN^k$ have singular kernels; the
operator~$\cK^k$ is defined in the Hadamard finite part sense, and the
operator~$\cN^k$ is defined as a Cauchy principal value. The
operator~$\cN^k$ in~\eqref{mfieoper} appears in the classic Magnetic
Field Integral Equation~(MFIE), and is in fact a weakly-singular
integral operator.  Integral equation~\eqref{mullerrep} is the
well-known M\"uller formulation for electromagnetic scattering from
dielectric objects.  Due to the presence of only the~\emph{difference}
of hypersingular operators appearing in~\eqref{mullerrep}, note that
this system of integral equations is Fredholm of the second-kind
when~$\Gamma$ is smooth~\cite{rok2,MULLER}.  By the Fredholm
alternative, the existence of a solution follows from the
uniqueness~\cite{vico2018cpde,MULLER}.

When the boundary~$\Gamma$ is not smooth, but rather contains edges
and corners, the integral operators in~\eqref{mullerrep} are not
compact, but merely bounded in the appropriate Sobolev space.  In this
case the proof of uniqueness is slightly more involved, but the
results still hold. See~\cite{KH2015}, Theorem 5.52 for details. Our
numerical examples will include geometries that are both globally
smooth and merely piecewise smooth (i.e. containing edges). We see
similar high-accuracy results in both cases using the same
discretization scheme (albeit with dyadic mesh-refinement near any
geometric singularities).
 
Once equation~\eqref{mullerrep} has been solved for the surface
currents~$\bJ$ and~$\bM$, the exterior and interior scattered fields
can be evaluated using similar Green's-like identities. For~$\bx \in
\Omega_0$, we have that
\begin{equation}
  \label{exteriorfield}
  \begin{aligned}
    \besce &= -\frac{1}{i\omega\varepsilon_0}\nabla \times \nabla
    \times \cS^{k_0} \bJ + \nabla \times \cS^{k_0} \bM, \\
    \bhsce &= \frac{1}{i\omega\mu_0}\nabla \times \nabla \times \cS^{k_0}
    \bM + \nabla \times \cS^{k_0} \bJ, 
  \end{aligned}
\end{equation}
and for $\bx \in \Omega$, we have that 
\begin{equation}
  \label{interiorfield}
 \begin{aligned}
   \besci &= \frac{1}{i\omega\varepsilon_1}\nabla \times \nabla
   \times \cS^{k_1} \bJ - \nabla \times \cS^{k_1} \bM, \\
   \bhsci &= -\frac{1}{i\omega\mu_1}\nabla \times \nabla
   \times \cS^{k_1} \bM - \nabla \times \cS^{k_1} \bJ.
 \end{aligned}
\end{equation}

\subsection{An indirect formulation}
\label{sec-dual}

Another approach that is commonly used to derive an integral equation
for the TBVP is known as the \emph{indirect method}.  It is based on
the fact that the pair of vector potentials
$\nabla \times \cS^{k} \bJ$ and
$\nabla \times \nabla \times \cS^{k} \bJ/ik$, set to be the
electric field and magnetic field, respectively, automatically satisfy
the time harmonic Maxwell equations with wavenumber~$k$. This is
sometimes referred to as the Rokhlin-M\"uller representation for
transmission problems~\cite{rok2,MULLER}; it is the generalization
from electrostatics and acoustics to electomagnetics of using an
indirect linear combination of a single- and double-layer potential to
represent the solution. The radiation condition is also automatically
satisfied via the use of these layer potentials.  With this in mind,
we keep the same representation for the exterior field as before
in~\eqref{exteriorfield}, and merely replace the appropriate
dielectric constants with their interior counterparts for the interior
field to write:
\begin{equation}
  \label{interior_dual}
  \begin{aligned}
    \besci &= -\frac{1}{i\omega\varepsilon_0}\nabla \times \nabla
    \times
    \cS^{k_1} \bJ + \frac{\mu_1}{\mu_0}\nabla \times \cS^{k_1} \bM, \\
    \bhsci &= \frac{1}{i\omega\mu_0}\nabla \times \nabla \times
    \cS^{k_1} \bM + \frac{\varepsilon_1}{\varepsilon_0}\nabla \times
    \cS^{k_1} \bJ.
  \end{aligned}
\end{equation}
A priori, due to the representation, these interior and exterior fields
automatically satisfy Maxwell's equations. We merely need to solve
for~$\bJ$ and~$\bM$ to satisfy the transmission conditions.
As before, the integral equation along~$\Gamma$ can be obtained by
taking the limit of the tangential components of
representations~\eqref{exteriorfield} and~\eqref{interior_dual}
as~$\bx$ approaches the
boundary and  then forming linear
combinations to enforce the transmission boundary condition.
This procedure results in the integral equation:
\begin{equation}
  \label{dualmullrep}
  \begin{aligned}
    i\omega\mu_0\bn \times\beinc &=
    -\frac{i\omega}{2}\lp\mu_1+\mu_0\rp\bM
    + i\omega\lp \mu_1\cN^{k_1} - \mu_0\cN^{k_0} \rp \bM
    - \frac{\mu_0}{\varepsilon_0} \lp
    \cK^{k_1} -  \cK^{k_0} \rp \bJ, \\
    i\omega\varepsilon_0\bn \times \bhinc &=
    -\frac{i\omega}{2}\lp\varepsilon_1+\varepsilon_0\rp\bJ
    + i\omega\lp \varepsilon_1\cN^{k_1} - \varepsilon_0\cN^{k_0} \rp
    \bJ
    +\frac{\varepsilon_0}{\mu_0} \lp
     \cK^{k_1} -  \cK^{k_0} \rp \bM.
  \end{aligned}
\end{equation}
This system of integral equations is an indirect form of
equation~\eqref{mullerrep}~\cite{GG2013}. 

This indirect  integral equation has similar properties to the
classical M\"uller integral equation;
when the boundary~$\Gamma$ is smooth,
equation~\eqref{dualmullrep} is also Fredholm second-kind 
and admits a unique solution.
In the case of boundaries with edges and corners, since the integral
operators in the indirect formulation are of the same
  order as those in the direct formulation, 
we have the same regularity and uniqueness results.
However, note that the currents~$\bJ$ and~$\bM$ 
in equation~\eqref{dualmullrep} do not represent tangential traces of
the fields anymore. The
advantage of this formulation is that we are able to
easily construct an exact solution to the TBVP in each of the regions
using the integral representations~\eqref{exteriorfield}
and~\eqref{interior_dual}.
This provides true verification of  the accuracy of the numerical
solver. In subsequent numerical experiments, we solve
equation~\eqref{dualmullrep} for accuracy verification and
equation~\eqref{mullerrep} for obtaining tangential traces of the
fields.

In the next section, we discuss the discretization of
equation~\eqref{mullerrep} and~\eqref{dualmullrep} along the
boundaries of axisymmetric objects.

\section{Fourier representation of the
  boundary integral operators}
\label{sec_fourier}

Discretizing integral equations in complex geometries in three
dimensions to high-order is non-trivial and currently rather
computationally expensive. All aspects of this problem are active
areas of research: high-order geometry construction, quadrature,
constructing optimal fast direct solvers, and coupling fast algorithms
with quadrature methods.
  However, there exist many interesting
applications of electromagnetic scattering from axisymmetric objects
(for instance, parabolic reflectors~\cite{Bulygin2013Full}, buried
mines and unexploded ordnances~\cite{Geng1999Wide}, etc); in this
case, variables can be separated in cylindrical coordinates resulting
in a system of decoupled line integrals. The discretization and
solution of integral equations along curves in two dimensions is a
much easier problem, and very efficient schemes
exist~\cite{Hao2015,YHM2012,Liu2015,hao_2014}.  The
resulting Fourier decomposition scheme easily parallelizes and can
address a range of rather complicated axisymmetric geometries.

A concise discussion regarding the discretization of scalar-valued
integral equations along bodies of revolution is contained
in~\cite{YHM2012,helsing2015}; a modern treatment of
the vector-valued case, in particular
integral equation methods for Maxwell's equations, is discussed
in~\cite{helsing2015ieee,HK15,epstein_2018,oneil2018}.
In each case, a choice of discretization, quadrature, and subsequent
linear algebraic solver must be made.
The linear systems resulting
from a separation of variables integral equation approach are
generally small (scaling as $\sqrt{N}$, where $N$ is the number of
degrees of freedom needed on the corresponding rotated surface in
three dimensions)
and can be rapidly solved merely using Gaussian elimination or
GMRES~\cite{GMRES1986} with dense matrix-vector multiplications. However, the
choice of discretization and quadrature varies depending on the
particular geometries of interest and the numerical tools available.
For example, in~\cite{helsing2015ieee,HK15}, a panel-based discretization of
the geometry with exact product integration was used to find
interior eigenfields by solving the Magnetic Field Integral
Equation~(MFIE) and the charge integral equation~(ChIE).
In~\cite{epstein_2018, oneil2018}, the boundaries were assumed to be
smooth and the resulting integral equations based on a generalized
Debye formulation of electromagnetic fields~\cite{EG10,EG13} were
discretized using a trapezoidal rule along with hybrid
Gauss-trapezoidal quadrature rules~\cite{alpert1999}.
Neither work addressed the classic,
and widely used, M\"uller integral equation formulation for
penetrable media, which requires higher derivatives of the modal
Green's functions. Furthermore, the formulation in~\cite{epstein_2018}
(at least in its current state)
is not compatible with geometries containing edges or points;
M\"uller's integral equation is.
In what follows, we emulate the 
style and notation in our previous work~\cite{LAI20171} rather closely.

As in Figure~\ref{figure1}, assume that~$\Omega$ denotes an
axisymmetric object (i.e. body of revolution) with
boundary~$\Gamma$. The boundary will be assumed to be smooth or
contain a small number of edges or points.  Cylindrical coordinates
will be given as $(r,\theta,z)$, and we will denote the corresponding
standard unit vectors by~$(\er,\et,\ez)$.  Furthermore, we assume that
the cross section of~$\Gamma$ in the $\theta=0$ plane, also referred
to as the generating curve~$\gamma$, is parameterized counterclockwise
by~$\vct{\gamma}(s) = (r(s),z(s))$, where~$s$ denotes arclength. This
implies that the unit tangential vector along the generating curve
(and~$\Gamma$ itself) is~$\bt(s) = r'(s) \, \er +z'(s) \, \ez$, with
$r'$ and $z'$ denoting differentiation with respect to arclength,
i.e. $r' = dr/ds$.  The unit exterior normal~$\bn$ is then given
by~$\bn(s) = z'(s) \, \er - r'(s) \, \ez$.  A surface current~$\bJ$
on~$\Gamma$ can be written in terms of these coordinates
as~$\bJ = J^1 \, \bt + J^2 \, \et$. Furthermore,
since~$\Gamma$ is always smooth in
the azimuthal direction, taking the Fourier expansion of~$J^1$
and~$J^2$ with respect to~$\theta$ yields
\begin{equation}
\bJ(r,z) = \sum_m  \lp J_m^1(r,z) \, \bt + J_m^2(r,z) \, \et \rp e^{im\theta}.
\end{equation}
The dependence of the unit vectors on the
variables~$r,\theta,z$ will be omitted unless needed for clarity.
Also, sometimes it will be useful to denote functions in terms of
the arclength variable along~$\gamma$, for example~$J^1_m(s) =
J^1_m\lp r(s), z(s) \rp$.

We begin with the following lemma, which is given
in~\cite{LAI20171}, of which the proof is by direct computation.
\begin{lemma}\label{lemma1}
In cylindrical coordinates, the vector potential $\cS^k\bJ$ 
  has the Fourier expansion 
\begin{align}\label{fourierdecomp}
  \cS^k\bJ(r_t,\theta_t,z_t)
  = \sum_m\lp c_m^1(r_t,z_t) \, \er + c_m^2(r_t,z_t) \, \et
  +c_m^3(r_t,z_t) \, \ez \rp e^{im\theta_t}
\end{align}
where 
\begin{multline}
  c_m^1(r_t,z_t) =  \int_{\gamma} J^1_m(s) \, r(s) \, r'(s) \,
  g^2_m(r_t,z_t,r(s),z(s)) \, ds  \\ 
  -i\int_{\gamma}J^2_m(s) \, r(s) \,
  g_m^3(r_t,z_t,r(s),z(s)) \, ds, \label{equc1}
\end{multline}
\begin{multline}
  c_m^2(r_t,z_t) =  i\int_{\gamma} J^1_m(s) \, r(s) \, r'(s) \,
  g_m^3(r_t,z_t,r(s),z(s)) \, ds \\
  +\int_{\gamma} J^2_m(s) \, r(s) \,
  g_m^2(r_t,z_t,r(s),z(s))
  \, ds, \label{equc2}
\end{multline}
\begin{equation}
  c_m^3(r_t,z_t) = \int_{\gamma}  J^1_m(s) \, r(s) \, z'(s) \,
                    g_m^1(r_t,z_t,r(s),z(s)) \, ds, \label{equc3}
\end{equation}
and the target point is denoted as~$(r_t,\theta_t,z_t) =
(r(t), \theta_t,z(t))$. The kernels above
are defined by
\begin{align}
  g_m^1(r_t,z_t,r,z)& =\int_{0}^{2\pi}\frac{e^{ik\rho}}{4\pi \rho}
                            e^{-im\phi} \, d\phi, \label{green11}\\
  g_m^2(r_t,z_t,r,z)& =\int_0^{2\pi}\frac{e^{ik\rho}}{4\pi \rho}
                            \cos{m\phi} \, \cos{\phi} \, d\phi, \label{green21}\\
  g_m^3(r_t,z_t,r,z)& = \int_{0}^{2\pi}\frac{e^{ik\rho}}{4\pi
                            \rho}
                            \sin{m\phi} \, \sin{\phi} \, d\phi, \label{green31}
\end{align}
where
\begin{equation}\label{eq:rho}
  \rho =
  \sqrt{r_t^2+r^2-2r_tr\cos{\phi}+(z_t-z)^2}
\end{equation}
with $\phi = \theta_t - \theta$ and~$(r,z) = (r(s),z(s))$.
The functions~$g^i_m$ are commonly referred to as \emph{the modal
Green's functions}.
\end{lemma}

Using Lemma~\ref{lemma1}, we can also obtain the azimuthal
Fourier decomposition of the layer potentials~$\cN^k\bJ$
and~$\cK^k\bJ$ along the boundary:
\begin{multline}
 \cN^k\bJ(r_t,\theta_t,z_t)
 =  \sum_{m}
 \lp
 \lp \frac{\partial c^1_m}{\partial z_t}-\frac{\partial c^3_m}{\partial
   r_t} \rp \bt \right. \\
\left. 
   +\lp
 \frac{z_t'}{r_t}
 \lp c^2_m+r_t\frac{\partial c^2_m}{\partial r_t} -imc^1_m \rp
 +r_t' \lp \frac{im}{r_t}c_m^3-\frac{\partial c^2_m}{\partial
  z_t} \rp\rp\et
\rp e^{im\theta_t},
\end{multline}
and by using the fact that~$\cK^k\bJ  = \bn\times
\lp k^2 \cS^k\bJ+\nabla \nabla \cdot \cS^k\bJ  \rp$, where~$\cK^k$ is
given in~\eqref{efieoper},  we have
\begin{equation}
  \resizebox{.9 \textwidth}{!} {
    $
  \begin{aligned}
    \cK^k\bJ(r_t,\theta_t,z_t) &= \sum_m
    \bigg[k^2\bigg(c_m^2\bt+\big(z_t'c_m^3+r'_tc_m^1\big)\et\bigg)
    +\bigg(\frac{im}{r_t^2}c_m^1 + \frac{im}{r_t}\frac{\partial c_m^1}
    {\partial r_t} -\frac{m^2}{r_t^2}c_m^2 +
    \frac{im}{r_t}\frac{\partial c_m^3}
    {\partial z_t}\bigg)\bt  \\
    &\qquad + \bigg( z'_t\bigg(\frac{1}{r_t}\frac{\partial
      c_m^1}{\partial z_t} + \frac{\partial^2c_m^1}{\partial z_t
      \partial r_t} +\frac{im}{r_t}\frac{\partial c_m^2}{\partial z_t}
    + \frac{\partial^2 c_m^3}{\partial z_t^2}\bigg) +
    r'_t\bigg(-\frac{1}{r^2_t}c_m^1+\frac{1}{r_t}\frac{\partial
      c_m^1}{\partial r_t}+\frac{\partial^2c_m^1}{\partial r_t^2} \\
    &\qquad - \frac{im}{r_t^2}c_m^2 +\frac{im}{r_t}\frac{\partial
      c_m^2}{\partial r_t} + \frac{\partial^2c_m^3}{\partial
      r_t\partial z_t}\bigg) \bigg)\et\bigg] e^{im\theta_t}.
\end{aligned}
$ }
\end{equation}
Since this expression is valid for $(r_t,z_t)$ along the generating
curve, the gradient of~$\gamma$ at~$(r_t,z_t)$ is given
by~$\nabla\gamma(t) = (r'_t,z'_t)$.

In order to evaluate the layer potentials~$\cN^k\bJ$ and~$\cK^k\bJ$
rapidly, the values of~$c^1_m$,~$c^2_m$ and~$c^3_m$, as well as their
derivatives, need to be computed efficiently.  The evaluation of the
coefficients~$c^i_m$ can be performed in two steps: (1) Compute
$g^i_m$ and its derivatives, and then (2) integrate~$g_m^i$ and its
derivatives (according to the above formulae for $c^i_m$) along the
generating curve~$\gamma$. There are no known numerically useful
closed-form expressions for~$g_m^i$. The evaluation of these functions
occupies a significant portion of the run-time of the resulting
solver~\cite{LAI20171} (approximately $50\%$, as shown in
Table~\ref{tab:tables}), and therefore an efficient scheme for
computing them is important.  Expansions of these functions in terms
of half-order Hankel functions have, as of yet, proven
to be somewhat expensive to evaluate~\cite{conway_cohl}, and designing
robust contour integration methods for large values of~$m$ is quite
complicated~\cite{gustafsson2010accurate}.  Furthermore, note that
since~$g_m^i$ has a singularity when $z_t=z(s)$, specialized
quadratures must be used when discretizing and integrating
along~$\gamma$. In the next two sections, we give a detailed
discussion on the evaluation and integration of the modal Green's
functions.

\section{Fast modal kernel evaluation}
\label{kerneleval}

Given the important role they play in solving PDEs and integral
equations in axisymmetric domains, the speed of evaluating the kernel
functions~\eqref{green11}-\eqref{green31} is a very important
consideration as it can affect the overall efficiency of the entire
solver.  Due to the existence of singularities in the free-space
Green's functions, specialized routines must be developed for the
evaluation of the modal Green's functions.
In~\cite{LAI20171}, we previously applied adaptive Gaussian quadrature
to evaluate~$g_m^i$. Although high accuracy was achieved, the
algorithm was time consuming as every matrix entry required several
calls to an adaptive integration routine.  In this section, we adopt
an accelerated method based on recurrence relations and kernel
splitting, as discussed
in~\cite{epstein_2018,helsing2014,helsing2015}, and further develop an
efficient evaluation procedure for computing first and second
derivatives of~$g^i_m$ based on novel recurrences.

Let $(r(s),z(s))$ be replaced by~$(r_s,z_s)$ for notational
simplicity. Using the fact that~$\rho$ in equation~\eqref{eq:rho} is
an even function with respect to~$\phi$ on~$[0,2\pi]$, we observe that
\begin{equation}
  \label{equrelation}
    g_m^2 = \frac{g_{m+1}^1  + g_{m-1}^1}{2}, \qquad 
    g_m^3 = \frac{g_{m+1}^1 -      g_{m-1}^1}{2i}
\end{equation}
for any mode $m$.

In general, unless the wave numbers of the dielectric object and
background media are particularly high, only a modest number of
Fourier modes $m$ are needed for high-precision discretizations of the
integral equations.  Once the incoming data has been Fourier
transformed along the azimuthal direction on the boundary, the number
of Fourier modes needed in the discretization can be determined based
on the decay of the coefficients of the data. We will denote this
number, i.e. the bandwidth of the data in the azimuthal direction,
by~$M>0$.  Therefore, our goal is to evaluate all~$g^i_m$ for
$|m|\leq M$.  Furthermore, based on the relations
in~\eqref{equrelation}, we need only to evaluate~$g_m^1$ and its
derivatives efficiently; the other functions can be obtained by linear
combinations.

When the target~$(r_t,z_t)$ is far away from the source~$(r_s,z_s)$,
the integral in~\eqref{green11} can be discretized using the periodic
trapezoidal rule with $2M+1$ points and therefore the fast Fourier
transform (FFT) can be used to evaluate~$g^1_m$ for $m=-M,\ldots,M$.
However, for~$(r_t,z_t)$ near~$(r_s,z_s)$, the integrand is nearly
singular and a prohibitively large number of discretization points
would be needed to obtain sufficient quadrature accuracy.
To overcome this difficulty, we
adopt the kernel splitting technique, which has been successfully
applied in~\cite{helsing2014,helsing2015,epstein_2018,YHM2012}.
The main idea is to explicitly split the
integrand into smooth and singular parts. Fourier coefficients of the
smooths parts can be obtained numerically via the FFT, and  it turns out that
the coefficients of the singular part can be obtained analytically via
recurrence relations. The Fourier coefficients of the original kernel
can then be obtained via discrete convolution. See the previous
references for thorough details, particularly~\cite{epstein_2018},
which provides estimates on the size of the FFT needed and other
important tuning parameters.

In the following subsections, we discuss details of
evaluating~$g_m^1$ and its first and second derivatives for
fixed sources and targets. Previous work on evaluating modal Green's
functions has avoided second derivatives because the formulations did
not involve hypersingular terms. However, the M\"uller formulations
require the evaluation of the difference of hypersingular operators,
which necessitates the following new discussion.

\subsection{Evaluation of the Modal Green's Function}
\label{subsec_val}

We begin by splitting~$g_m^1$ into two parts:
\begin{equation}
  \label{split1}
  \begin{aligned}
  g_m^1(r_t,z_t,r_s,z_s)
  &= \int_{0}^{2\pi}\frac{\cos(k\rho)}{4\pi \rho}e^{-im\phi} \, d\phi
  + \int_{0}^{2\pi}\frac{i\sin(k\rho)}{4\pi \rho}e^{-im\phi} \, d\phi \\
  &= A_m^1 + A_m^2,
\end{aligned}
\end{equation}
where $\rho = \rho(\phi)$, given in equation~\eqref{eq:rho}.
The integrand in~$A_m^2$ is analytic with respect to~$\phi$ since the
singularity is removable for~\mbox{$\rho=0$}.  Therefore, the FFT
can be applied directly to find all~$A_m^2$
for~$-M\le m \le M$ at a cost of~$\mathcal{O}(M\log M)$ flops.
Note that special care must be taken in evaluating the kernel
in~$A^2_m$ for small values of $\rho$; a truncated
Taylor series about $\rho=0$ is an easy solution.

For~$A_m^1$, we first consider the case~$k=0$ and denote it by
$\hat q_m$. As shown in~\cite{YHM2012,cohl_1999}, it holds that
\begin{equation}
  \label{lm}
  \begin{aligned}
    \hat q_m &= \int_{0}^{2\pi}\frac{e^{-im\phi}}{4\pi \rho}d\phi   \\
    &= \frac{1}{4\pi}\int_0^{2\pi} \frac{\cos m \phi }
    {\sqrt{r_t^2+r^2_s-2r_tr_s\cos{\phi}+(z_t-z_s)^2}} d\phi \\
    &= \frac{1}{2\pi\sqrt{r_tr_s}} \int_0^{2\pi}\frac{\cos m \phi}
    {\sqrt{8(\chi-\cos\phi)}}d\phi \\
    &= \frac{1}{2\pi\sqrt{r_tr_s}} \mathcal{Q}_{m-1/2}(\chi),
\end{aligned}
\end{equation}
where
\begin{equation}
  \chi=\frac{r_t^2+r^2_s+(z_t-z_s)^2}{2r_tr_s} \geq 1
\end{equation}
and~$\mathcal{Q}_{m-1/2}$ is the Legendre
function of the second-kind of half-degree.
It can be evaluated by the following
recursion formula~\cite{Hand2010}:
\begin{equation}\label{recurform}
  \mathcal{Q}_{m-1/2}(\chi)
  = 4\frac{m-1}{2m-1}\chi \mathcal{Q}_{m-3/2}(\chi)
  -\frac{2m-3}{2m-1}\mathcal{Q}_{m-5/2}(\chi)
\end{equation}
with
\begin{equation}
  \begin{aligned}
    \mathcal{Q}_{-1/2}(\chi) &= \sqrt{\frac{2}{\chi+1}}
    \, K \lp \sqrt{\frac{2}{\chi+1}} \rp, \\
    \mathcal{Q}_{1/2}(\chi) &= \chi\mathcal{Q}_{-1/2}(\chi) -
    \sqrt{2(\chi+1)} \, E\lp \sqrt{\frac{2}{\chi+1}} \rp,
\end{aligned}
\end{equation}
where~$K$ and~$E$ are the complete elliptic integrals of the first and
second kinds.  While standard from the theory of orthogonal
polynomials, we provide a derivation of the above recurrence formula
in the appendix. This derivation is useful in order to obtain
efficient recurrence relations for higher derivatives. Unfortunately,
the recurrence formula~\eqref{recurform} is unstable for
increasing~$m$, and therefore Miller's algorithm must be
implemented~\cite{gil_2007}. For~\mbox{$\chi \approx 1$}, the forward
recurrence is only mildly unstable, and can be used with
caveats. See~\cite{epstein_2018} for an estimate on the number of
terms that can be evaluated accurately in this regime.

In order to evaluate~$g^1_m$ for $k\ne 0$, we apply the convolution
technique proposed in~\cite{YHM2012} with a slight modification.
Note that~$A^1_m$ is merely the Fourier transform of a
product of functions, which can be computed as a discrete convolution:
\begin{equation}\label{eq:discA}
  \begin{aligned}
    A^1_m &= \int_0^{2\pi} f(\phi) \, q(\phi) \, e^{-im\phi} \, d\phi \\
    &= \frac{1}{2\pi} \sum_{n=-\infty}^\infty \hat f_n \, \hat q_{m-n},
  \end{aligned}
\end{equation}
where
\begin{equation}
  f(\phi) = \cos \lp k \rho(\phi) \rp, \qquad 
  q(\phi) = \frac{1}{4\pi \rho(\phi)}
\end{equation}
and~$\hat f_m$ denotes the $m$th Fourier series coefficient
of the function~$f$:
\begin{equation}
    \hat f_m = \int_0^{2\pi} f(\phi) \, e^{-im\phi} d\phi.
\end{equation}
The representation of~$A^1_m$ in~\eqref{eq:discA}
is a discrete convolution with infinite extent, which is
impractical for numerical purposes. However, since~$f=\cos k\rho$ is an analytic
function of~$\phi$, its Fourier series converges rapidly. Therefore,
the above convolution can be truncated:
\begin{equation}
A^1_m = \frac{1}{2\pi} \sum_{n=-N}^N \hat f_n \, \hat q_{m-n},
\end{equation}
where~$N$ is chosen such that~$|\hat f_n| \leq \epsilon$ for some
user-specified precision~$\epsilon >0$. As is well
known~\cite{briggs_1995}, discrete convolutions can be computed
efficiently using the FFT and properties of the discrete Fourier
transform~(DFT). To begin with, denote by~$\hat{\vct{f}}$ the vector
of Fourier coefficients $\hat f_n$, by $\hat{\vct{q}}$ the vector
of Fourier coefficients $\hat q_n$, and by~$\vct{A}^1$ the vector
with elements $A_{-N}^1, \ldots, A_N^1$. Letting~$\mtx{D}$ denote the
$(2N+1) \times (2N+1)$ DFT matrix, we have:
\begin{equation}
  \begin{aligned}
    \vct{A}^1 &= \mtx{D}^{-1} \mtx{D} \vct{A}^1 \\
    &= \mtx{D}^{-1} \mtx{D} \lp \hat{\vct{f}} \ast \hat{\vct{q}}
    \rp \\
    &= \mtx{D}^{-1} \lp \mtx{D} \hat{\vct{f}} \odot
    \mtx{D} \hat{\vct{q}}   \rp \\
    &= \mtx{D}^{-1} \lp \mtx{D}^2 \vct{f} \odot  
    \mtx{D} \hat{\vct{q}}   \rp \\
    &= \mtx{D}^{-1} \lp  \vct{f} \odot  
    \mtx{D} \hat{\vct{q}}   \rp
  \end{aligned}
\end{equation}
where~$\ast$ denotes the cyclic convolution,~$\odot$ denotes the
pointwise Hadamard product of two vectors, and~$\vct{f}$ denotes
the vector obtained from a $2N+1$-point equispaced sampling of the
functions $f$ on the interval $[0,2\pi)$.
The last identity
follows from the fact that~$f$ is an even function.  Note that a
larger extent of coefficients~$\hat q_n$ will be needed than
for~$\hat f_n$ due to the definition of the convolution; denote this
bandwidth as~$M_q$. Due to the fact that the coefficients~$\hat f_n$
decay rapidly, this sequence can be zero-padded easily (and recall
that~$\hat q_n$ is obtained analytically).
In practice, the bandlimit~$M_q$ can be determined based on the desired number of
coefficients~$A^1_m$ and subsequently all FFTs are of size~$2M_q+1$,
with~$\vct{f}$ being computed from a~$(2M_q+1)$-equispaced sampling of~$f$
on~$[0,2\pi)$.

\subsection{Evaluation of the first derivatives}
\label{subsec_fistd}

Let us now focus on the derivative of~$g^m_1$ with respect to~$r_t$,
as the evaluation of the derivative with respect to~$z_t$ is
similar. The derivative with respect to $r_t$ is given as
\begin{equation}
  \label{firstder}
  \frac{\partial g_m^1}{\partial r_t} =
  \frac{1}{4\pi}\int_0^{2\pi}\frac{(ik\rho e^{ik\rho}-e^{ik\rho}) \,
    r_d}
  {\rho^3} \, e^{-im\phi} \, d\phi ,
\end{equation}
where $r_d=r_t-r_s\cos\phi$. Splitting the kernel on the right hand
side of~\eqref{firstder} into smooth and singular parts leads to
\begin{equation}
  \begin{aligned}
    \frac{\partial g_m^1}{\partial r_t}
    &= \frac{1}{4\pi}\int_0^{2\pi} \frac{i\lp k\rho \cos(k\rho)
      -\sin(k\rho) \rp \, r_d}{\rho^3} \, e^{-im\phi} \, d\phi  \\
    & \qquad-\frac{1}{4\pi}\int_0^{2\pi}
    \frac{-k\sin(k\rho) \, r_d}{\rho} \, e^{-im\phi} \, d\phi
    - \frac{1}{4\pi}\int_0^{2\pi} \frac{\cos(k\rho) \, r_d}{\rho^3}
    \, e^{-im\phi} \, d\phi \\ 
    &= B_m^1+ B_m^2+B_m^3,
  \end{aligned}
\end{equation}
where~$\rho = \rho(\phi)$ as before.  By Taylor expansion, one can
easily see that the integrands in~$B_m^1$ and~$B_m^2$ are analytic
with respect to~$\phi$ due to the removable singularities. Therefore,
once again a modestly-sized FFT is an efficient mean to
evaluate~$B_m^1$ and~$B_m^2$. For~$B_m^3$, we again consider the case
when~$k=0$ and~$r_d=1$, which is denoted by~$\hat p_m$,
\begin{equation}
  \label{pm}
  \begin{aligned}
    \hat p_m &= \frac{1}{4\pi} \int_0^{2\pi}
    \frac{1}{\rho^3} \, e^{-im\phi} \, d\phi \\
    &= \frac{1}{4\pi}\int_0^{2\pi}
    \frac{\cos m \phi
    }{(r_t^2+r_s^2-2r_tr_s\cos{\phi}+(z_t-z_s)^2)^{3/2}}
    \, d\phi  \\
    &= \frac{1}{4\pi (2r_tr_s)^{3/2}} \int_0^{2\pi}
    \frac{\cos m \phi}{(\chi-\cos\phi)^{3/2}} \, d\phi \\
    &= \frac{1}{4\pi (2r_tr_s)^{3/2}} \, \mathcal{S}_m(\chi).
\end{aligned}
\end{equation}
It can be shown the function~$\mathcal{S}_m$ is related to the
functions~$\mathcal{Q}_{m-1/2}$ and~$\mathcal{Q}_{m-3/2}$.
The following relationship  holds for~$m \ge 1$:
\begin{eqnarray}
  \mathcal{S}_m(\chi)  = \sqrt{8} \frac{
  (1-2m)\chi \, 
  \mathcal{Q}_{m-1/2}(\chi)-{(1-2m)} \, \mathcal{Q}_{m-3/2}(\chi)
  }{\chi^2-1}.
\end{eqnarray}
A proof of the above formula is contained in
Appendix~\ref{sec:recurrence}. The sequence~$\hat p_m$ can easily be
obtained once the values~$\mathcal{Q}_{m-1/2}(\chi)$ have been computed.
With~$p_m$ available, the sequence of~$B_m^3$ can be obtained via a
convolution technique similar to the evaluation of~$A_m^1$.  A nearly
identical procedure can be carried out to
compute~$\partial g^1_m/\partial z_t$.

\subsection{Evaluation of the second derivatives}
\label{subsec_secd}

The procedure for evaluating the second partial derivatives of~$g^1_m$
is similar to that of evaluating the first partial derivatives, as
detailed in the previous section.  As before, only the evaluation
of~$\partial^2 g_m^1/\partial r^2_t$ will be discussed since the
evaluation of~$\partial^2 g_m^1/\partial z^2_t$
and~$\partial^2 g_m^1/\partial z_t\partial r_t$ are very similar.
Taking the second derivative of~$g_m^1$ with respect to~$r_t$ we have,
\begin{multline}
  \frac{\partial^2 g_m^1}{\partial r^2_t} = \frac{1}{4\pi}\int_0^{2\pi}
  \lp \lp ik\frac{ik\rho e^{ik\rho}-e^{ik\rho}}{\rho^2}-
  \frac{ik\rho^2e^{ik\rho}-2e^{ik\rho}}{\rho^3} \rp \cdot
  \frac{r_d^2}{\rho^2} \right. \\
+ \left. \frac{ik\rho e^{ik\rho}-e^{ik\rho}}{\rho^2}\cdot
\lp \frac{1}{\rho}-\frac{r_d^2}{\rho^3} \rp \rp \, e^{-im\phi} \, d\phi,
\end{multline}
where, as before, $r_d = r_t-r_s\cos\phi$.
In order to apply the kernel splitting technique, we now decompose the
right hand side in the above formula into the sum of six terms, listed
according to the order of the singularities in their integrands:
\begin{equation}
  \label{eq:cvals}
\frac{\partial^2 g_m^1}{\partial r^2_t} = C_m^1 + C_m^2+ C_m^3+C_m^4+C_m^5+C_m^6
\end{equation}
with 
\begin{equation}
  \begin{aligned}
    C_m^1 &= \frac{1}{4\pi}\int_0^{2\pi} \frac{ik\rho\cos{k\rho}
      -i\sin{k\rho}}{\rho^3} \, e^{-im\phi} \, d\phi,  \\
    C_m^2 &= -\frac{1}{4\pi}\int_0^{2\pi} \frac{k\sin {k\rho}}{\rho}
    \cdot \frac{1}{\rho} \, e^{-im\phi} \, d\phi,  \\
    C_m^3 &= -\frac{1}{4\pi}\int_0^{2\pi}
    \frac{ik^2\sin(k\rho)}{\rho}\cdot\frac{r_d^2}{\rho^2} \,
    e^{-im\phi} \, d\phi,  \\
    C_m^4 &= -\frac{1}{4\pi}\int_0^{2\pi}
    \lp k^2\cos(k\rho)+3 \lp\frac{-k\rho \sin {k\rho}}{\rho^2}
    +\frac{ik\cos{k\rho}-i\sin{k\rho}}{\rho} \rp \rp
    \frac{r_d^2}{\rho^3} \, e^{-im\phi} \, d\phi,  \\
    C_m^5 &= -\frac{1}{4\pi}\int_0^{2\pi}\frac{\cos k\rho}{\rho^3} \,
    e^{-im\phi} \, d\phi,  \\
    C_m^6 &= \frac{1}{4\pi}\int_0^{2\pi} \frac{(\cos
  k\rho) \, r_d^2}{\rho^5} \, e^{-im\phi} \, d\phi.
\end{aligned}
\end{equation}
Examining each term closely, we have that:
\begin{itemize}
  \item By Taylor series expansion
the integrand in~$C_m^1$ is analytic.  It can therefore be evaluated
via the trapezoidal rule.
\item To evaluate~$C_m^2$, since the singularity
in the integrand is~$1/\rho$, we can apply the same method as
evaluating~$A_m^1$ in Section~\ref{subsec_val}.

\item The singular part of the integrand of~$C^m_3$ is~$1/\rho^2$. Consider
the Fourier transform of~$1/\rho^2$ with respect to~$\phi$, which we
denote by~$\hat h_m$:
\begin{equation}
  \label{hm}
  \begin{aligned}
    \hat h_m &= \frac{1}{4\pi}\int_0^{2\pi}\frac{e^{-im\phi}}{\rho^2} \, d\phi \\
    & = \frac{1}{4\pi}\int_0^{2\pi}\frac{\cos m \phi }
    {r_t^2+r_s^2-2r_tr_s\cos{\phi}+(z_t-z_s)^2} \, d\phi \\
    &= \frac{1}{8\pi r_tr_s} \int_0^{2\pi } \frac{\cos m \phi }
    {(\chi-\cos\phi)} \, d\phi \\
    &= \frac{1}{\pi r_tr_s} \mathcal{T}_m(\chi),
  \end{aligned}
\end{equation}
where~$\mathcal T_m$ satisfies the following recursion formula
for $m\ge 1$ (proof given in Appendix~\ref{sec:recurrence}):
\begin{equation}
  \mathcal{T}_{m+1}(\chi) = 2\chi \, \mathcal{T}_{m}(\chi)-\mathcal{T}_{m-1}(\chi).
\end{equation}  
Notably, this is the same recurrence relation as for 
Chebychev polynomials, but with different initial values:
\begin{eqnarray}
\mathcal{T}_0(\chi) & = & \frac{2\pi}{\sqrt{\chi^2-1}}, \\
\mathcal{T}_1(\chi) & = & -2\pi+\chi\mathcal{T}_0(\chi) .
\end{eqnarray}

\item The singularity in the
integrand of~$C_m^4$ is~$1/\rho^3$, and it can be evaluated following
the discussion regarding the computation of~$B_m^3$ in
Section~\ref{subsec_fistd}.

\item For~$C_m^5$ and~$C_m^6$, their singular terms are $1/\rho^3$ and
$1/\rho^5$, respectively. Although there are no essential difficulties
in evaluating them using the same kernel splitting technique, note
that they
will eventually be used when evaluating the second derivative
of~$c_m^i$ in equations~\eqref{equc1}--\eqref{equc3}. This requires
the evaluation of hypersingular kernels. However, this numerical
difficulty can be avoided by observing that only the difference
of~$C_m^5$ and~$C_m^6$ (with different wavenumbers~$k_0$ and~$k_1$)
appear in the integral equations~\eqref{mullerrep}
and~\eqref{dualmullrep}.  The order of the singularity in
evaluating~$C^5_m$ or~$C^6_m$ can therefore be reduced by instead directly
evaluating the difference kernel.  Denote by~$C_m^{5,k_i}$
and~$C_m^{6,k_i}$ the dependence on wavenumber~$k_i$, for $i=0,1$. By
direct computation, we have:
\begin{equation}
  \begin{aligned}
    C_m^{5,k_0}-C_m^{5,k_1} &=
    -\frac{1}{4\pi}\int_0^{2\pi} \lp \frac{\cos k_0\rho
       -\cos k_1\rho}{\rho^2}\rp \frac{1}{\rho} \, e^{-im\phi} \, d\phi \\
     C_m^{6,k_0}-C_m^{6,k_1} & = \frac{1}{4\pi}\int_0^{2\pi}
     \lp \frac{\cos k_0\rho-\cos k_1 \rho}{\rho^2} \rp
     \frac{r_d^2}{\rho^3} \, e^{-im\phi} \, d\phi.
\end{aligned}
\end{equation}
Therefore, the singularities in the integrand of the differences
of~$C_m^{5,k_i}$ or~$C_m^{6,k_i}$ are~$1/\rho$ and~$1/\rho^3$,
respectively.  Their evaluation follows the same procedure as
evaluating~$A_m^1$ and~$B_m^3$. Once all the values of~$C_m^i$ have
been obtained, the evaluation of~$\partial^2 g_m^1 /\partial r_1^2$ is
obtained via the summation in~\eqref{eq:cvals}.

\end{itemize}

Lastly, it is worth pointing out that for a fixed target~$r_t,z_t$ and
source~$r_s,z_s$ the evaluation of~$\hat q_m$,~$\hat p_m$
and~$\hat h_m$ in equations~\eqref{lm},~\eqref{pm} and~\eqref{hm}
dominate the cost of kernel evaluation. However, these values can be
reused during the computation of the value of~$g_m^1$ and its
derivatives. This offers a significant savings in the cost of kernel
evaluation.

\section{Generalized Gaussian Quadrature}
\label{ggq}

Once a scheme is in place to evaluate the modal Green's functions and
their derivatives, the next step is to discretize each decoupled modal
integral equation along the generating curve~$\gamma$.  We use a
Nystr\"om-like method for discretizing the integral equations.  Since
the modal Green's functions have logarithmic
singularities~\cite{conway_cohl,cohl_1999}, any efficient
Nystr\"om-like scheme will require a quadrature that accurately
evaluates weakly-singular integral operators.  For high-accuracy
integration, one can construct the quadrature based on the kernel
splitting technique as in~\cite{Kress2010}. However, this will become
tedious given the many formulas for the singularities in the derivatives
of~$g_m^i$. For our numerical simulations, we implemented a
panel-based discretization scheme using generalized Gaussian
quadratures to evaluate the layer potentials. See~\cite{BGR2010,
  hao_2014} for an in-depth discussion of generalized Gaussian
quadrature schemes. The panel-based discretization scheme of this
paper, as opposed to that based on hybrid-Gauss trapezoidal
rules~\cite{alpert1999}, as presented
in~\cite{oneil2018,epstein_2018}, allow for adaptive discretizations,
in particular, axisymmetric surfaces in three dimensions with edges
and points.

To this end, we describe the procedure in general for any
weakly-singular integral operator with 
logarithmically-singular kernel~$g$. For a smooth
function~$\psi$, the goal is to evaluate, with high-order accuracy,
the integral
\begin{equation}\label{singint}
\mathcal S \psi(\bx) = \int_{\gamma} g(\bx,\by) \, \psi(\by) \, ds(\by),
\end{equation}
where~$\bx$ and $\by$ denote \emph{targets} and \emph{sources} on the
generating curve~$\gamma \subset \mathbb R^2$.
As before, assume that the generating curve~$\gamma$ is parameterized
as~$\gamma(s)$, where $s$ is arclength. The total arclength will be
denoted by $L$. The parameter domain~$[0,L]$ is first divided into~$N$
panels. This division
can be either uniform or nonuniform, depending on the particular geometry.
Each of the corresponding~$N$ image panels~$\gamma_i \subset \gamma$,
and therefore any function supported on it,
is discretized using $p$~scaled Gauss-Legendre nodes.
In a true Nystr\"om discretization scheme, the integral in~\eqref{singint}
would be approximated as
\begin{equation}
  \mathcal S \psi(\bx_{ln}) \approx
  \sum_{i = 1}^{N} \sum_{j=1}^{p} w_{lnij} \, g(\bx_{ln},\bx_{ij}) \,
  \psi(\vectorsym{x}_{ij}),
\end{equation}
where~$\bx_{ln}$ is the~$n$th Gauss-Legendre node on
panel~$l$,~$\bx_{ij}$ is the~$j$th Gauss-Legendre node on panel~$i$,
and~$w_{lnij}$ is the Nystr\"om quadrature weight for this term. However, it is
generally numerically difficult to derive efficient, high-accuracy
Nystr\"om schemes in which the quadrature nodes are the same
as the discretization nodes (i.e. the points at which $\psi$ are
sampled, the $p$ Gauss-Legendre nodes on each panel). Often it can be
very beneficial to use additional (or at least different) quadrature
\emph{support nodes} for approximating the integral.
%These support
%nodes and corresponding weights in our case were computed using the
%scheme of~\cite{BGR2010}.
In general, we instead approximate the layer
potential~$\mathcal S \psi$ in~\eqref{singint} by
\begin{equation}\label{disequ}
  \mathcal S \psi(\bx_{ln}) \approx
  \sum_{i = 1}^{N} \sum_{j=1}^{p} Q (\bx_{ln},\bx_{ij}) \,
  \psi(\vct{x}_{ij}),
\end{equation}
where $Q$ is referred to as the \emph{quadrature kernel}.
For non-adjacent panels
(i.e. when the source and target are well separated), we simply set
$Q(\bx_{ln},\bx_{ij})= w_{ij} \, g(\bx_{ln},\bx_{ij})$,
where~$w_{ij}$ is the standard $j$th scaled Gauss-Legendre weight on
panel~$i$.  Therefore, for non-adjacent panels, the order of
convergence
is expected to be~$2p-1$, although rigorous analysis and estimates on
the number of digits obtained requires knowledge of the
regularity of the density function~$\psi$.

On the other hand, along adjacent panels and
self-interaction panels, a pre-computed generalized Gaussian
quadrature is applied~\cite{BGR2010}.
More specifically, to compute the integral over
panel~$\gamma_i$ at a target~$\bx_{ln}$ on the same (or adjacent) panel,
we approximate~$\mathcal S \psi(\bx_{ln})$ as
\begin{equation}\label{eq_nyslike}
  \begin{aligned}
    \mathcal S \psi(\bx_{ln}) &= \int_{\gamma_i}
    g(\bx_{ln},\by) \, \psi(\by) \, ds(\by) \\
    &\approx \sum_{j=1}^{p} c_{ij} \int_{\gamma_i} g(\bx_{ln},\by)
    \, P^i_j(\by) \, ds(\by) \\
    &= \sum_{j=1}^{p}c_{ij} \, I_j(\bx_{ln}),
\end{aligned}
\end{equation}
where~$P^i_j$ is the scaled Legendre polynomial of degree~$j-1$ on
panel~$\gamma_i$, the numbers~$c_{ij}$ are the Legendre expansion coefficients
of the degree~$j-1$ interpolating polynomial for~$\psi$,
and~$I_j(\bx_{ln})$ is merely a number, given by:
\begin{equation}
I_j(\bx_{ln}) =  \int_{\gamma_i} g(\bx_{ln},\by) \, P^i_j(\by) \, ds(\by).
\end{equation}
For a
fixed~$\bx_{ln}$, each $I_j(\bx_{ln})$ contains no unknowns and can be
evaluated via a pre-computed high-accuracy generalized Gaussian
quadrature.  The number of nodes required in these quadratures may
vary, and details regarding the construction of these quadratures via
a nonlinear optimization procedure are discussed
in~\cite{BGR2010}. An analogous Nystr\"om-like discretization
scheme along surfaces in three dimensions is discussed
in~\cite{Bremer2012}. Fortran code for computing these generalized
Gaussian quadrature rules is  available
at~\mbox{\url{github.com/JamesCBremerJr/GGQ}}. In our numerical examples, we
use $16$th-order generalized Gaussian rules which contain~$16$ support
nodes and weights for self-interacting panels (which vary according to
the location of the target) and~$48$ support nodes and weights on
adjacent panels (which are target independent).

\begin{remark}
We would like to point out that this type of quadrature, generalized
Gaussian quadratures, often do not exhibit convergence in the classical sense;
similar to classical Gaussian quadratures for polynomials,
they are accurate to machine precision when applied to a fixed
set of functions. When discretizing to high-order, often
machine precision (or maximum precision, up to conditioning) is
obtained before any convergence study can be carried out (unless
extended precision calculations are used).
In our case, for example, the $k$th-order quadratures for
self-interacting panels are exact for integrals of the form
\[
I(x) =   \int_{-1}^1 \left( p(y) \log|x-y| + q(y) \right) \, dy,
\]
where~$x \in (-1,1)$ is a root of the $k$th-degree Legendre polynomial
and $p$, $q$ are any polynomials of degree less
than~$k$. See~\cite{BGR2010,yarvin1998}
for a discussion regarding the construction of these quadratures.
\end{remark}

Continuing, note that each $c_{ij}$ can be obtained via application
of a~$p\times p$ transform matrix acting on values of~$\psi$. Denote
this transform matrix as~$\matrixsym{U}^{i}$, and its entries
as~$U^{i}_{jk}$. Inserting this into~\eqref{eq_nyslike}, we have
\begin{equation}\label{eq_nyslike2}
  \begin{aligned}
    \mathcal S \psi(\bx_{ln}) &\approx   \sum_{j=1}^{p} c_{ij} \,
    I_j(\bx_{ln}) \\
    &=  \sum_{j=1}^{p}   \sum_{k=1}^p U^{i}_{jk} \, \psi(\bx_{ik})
    \, I_j(\bx_{ln}) \\
    &=    \sum_{k=1}^p  \lp \sum_{j=1}^{p} U^{i}_{jk} \, I_j(\bx_{ln})
    \rp \psi(\bx_{ik}) \\
    &=    \sum_{k=1}^p  Q(\bx_{ln}, \bx_{ik} ) \, \psi(\bx_{ik}).
\end{aligned}
\end{equation}
The above formula provides the expression for the quadrature
kernel~$Q$ in this case.

In the case that~$\gamma$ is only piecewise smooth, a graded mesh near
the corners is used to maintain high accuracy. After uniform
discretization of each smooth component of~$\gamma$, we perform a
dyadic refinement on panels that impinge on each corner
point. Unless \emph{very} specialized quadrature and discretization
schemes are used~\cite{serkh2016}, adaptive refinement is needed in geometries
with edges and corners in order to resolve both the numerical
evaluation of the integral operator as well as to resolve the solution
to the integral equation. Integral operators with logarithmic
singularities cease to be compact on Lipschitz domains, but are still
bounded operators on $L^2 \to L^2$~\cite{Brem2012,costabel1988}.
Dyadic refinement is therefore an
appropriate discretization scheme to approximate these integrals and
functions. We then
apply the $p$th-order generalized Gaussian quadrature on each of the
refined panels. This procedure, along with proper quadrature
weighting, has been shown to obtain very high
accuracy results~\cite{Brem2012}.

\section{Numerical Examples}
\label{sec_numeri}

In this section, we apply the discretization and quadrature technique
of the previous section to the separation of variables formulation of
M\"uller's integral equation to compute electromagnetic
scattering from various penetrable axisymmetric objects.
In order to verify the
accuracy of the solver, we choose to test the extinction theorem by
solving the indirect M\"uller formulation~\eqref{dualmullrep} with an
artificial solution. Specifically, we define the field in~$\Omega$
to be generated by a current loop located in~$\Omega_0$:
\begin{equation}\label{eq:loop}
  \begin{aligned}
    \besci &= \nabla\times \int_\ell G^{k_1}(\cdot,\by) \,
        d\vct{l}(\by), \\
    \bhsci &= \frac{1}{i\omega\mu_1}\nabla\times\nabla\times
    \int_\ell G^{k_1}(\cdot,\by) \,
        d\vct{l}(\by),
  \end{aligned}
\end{equation}
where~$\ell$ is a small loop centered at~$(x_0,y_0,z_0)$ and located
in the plane~$z=z_0$. Throughout all the numerical examples, we let
the center of the loop~$\ell$ be $(0.4,\, 0.5,\, 5.0)$ with
radius~$0.42$. The field in~$\Omega_0$ is simply
set to be zero, i.e.  $\besce = 0$, $\bhsce = 0$. These known 
fields are used to generate the boundary data for the TBVP.
Since both fields satisfy Maxwell equations in~$\Omega$
and~$\Omega_0$ with wavenumber $k_1$ and $k_0$, respectively, by
uniqueness,  solving
equation~\eqref{dualmullrep} will result in functions~$\bJ$ and~$\bM$
that can be used to reconstruct the known fields.

To compute tangential field traces from an actual scattering problem,
we solve M\"uller's formulation~\eqref{mullerrep} with an incident
plane wave:
\begin{equation}\label{eq:pw}
  \begin{aligned}
    \beinc(\bx) &= (\vct{d} \times \vct{p} )\times \vct{d}
    \exp(ik_0\vct{d}\cdot \bx), \\
    \bhinc(\bx) &= \vct{d} \times \vct{p} \exp(ik_0\vct{d}\cdot \bx),
\end{aligned}
\end{equation}
where~$\vct{d}$ is the propagation direction and $\vct{d}\times \vct{p}$ is an
orthogonal polarization vector; both are unit vectors~\cite{jackson1999}.
Unless specifically given, throughout all the examples, we let
\begin{equation}\label{incidentdirection}
  \begin{aligned}
  \vct{d} &= \lp \cos\theta_1\sin\phi_1,\, \sin\theta_1\sin\phi_1,
  \, \cos\phi_1 \rp,\\
  \vct{p} &= \lp \cos\theta_2\sin\phi_2, \,  \sin\theta_2\sin\phi_2,
  \, \cos\phi_2 \rp, 
\end{aligned}
\end{equation}
with $\theta_1 = \pi/3$, $\phi_1 = 2\pi/3$, $\theta_2=\pi/2$, and
$\phi_2 = \pi/3$. The far field pattern of the scattered wave can be
found by letting~$|\bx|\rightarrow\infty$ in~\eqref{exteriorfield} and
using the asymptotic form of the Green's function~\eqref{greenfun}:
\begin{multline}
  \label{farfieldform}
  |\besc_{\infty}(\theta,\phi)| = \frac{1}{4\pi} \left|
    \frac{1}{i\omega\varepsilon_0}
    \lp k_0^2\int_\Gamma e^{-ik_0\,
      {\vectorsym{x}} \cdot \by} \, \bJ(\by) \, ds(\by)
  \right. \right. \\
\left. \left. - \int_\Gamma \nabla_{\by}e^{-ik_0 \, 
      {\vectorsym{x}} \cdot \by }
    \, \nabla_\Gamma \cdot\bJ(\by) \, ds(\by) \rp 
             + \int_\Gamma  \nabla_{\by}e^{-ik_0
              {\vct{x}} \cdot \by }\times \bJ(\by) \, ds(\by) \right|,  
% &=  |\bh_{\infty}(\theta, \phi)|
%\end{aligned}
\end{multline}
where~$\theta\in[0,2\pi]$ is the azimuthal angle,~$\phi\in[0,\pi]$
is the angle with respect to the positive~$z$ axis, and
${\vectorsym{x}}=(\cos\theta\sin\phi,\sin\theta\sin\phi,\cos\phi)$
is a point on the unit sphere. 
The gradient operator with respect to $\by$ is denoted~$\nabla_{\by}$,
and $\nabla_\Gamma \cdot $ is the surface divergence operator
along~$\Gamma$. The norm of~$\bhsc_\infty$ is the same as that of~$\besc_\infty$.

The accuracy that controls the kernel evaluation (i.e. where to
truncate Fourier coefficients in the discrete convolutions) and the
number of Fourier modes in the decomposition of the incident wave is
set to be~$10^{-12}$. We apply the 16th-order Nystr\"om-like
discretization described earlier in Section~\ref{ggq} to each modal
integral equation along the generating curve. All experiments were
implemented in~\textsc{Fortran 90} and carried out on an HP
workstation with twenty 2.7Ghz Intel cores and 128Gb of RAM. We made
use of~\textsc{OpenMP} for parallelism across decoupled Fourier modes;
linear systems were solved via $\mtx{LU}$-factorization using a
standard~\textsc{LAPACK} library and the code was compiled using
the~\textsc{GCC}~\textsc{Fortran} compiler. Various fast direct
solvers such as \cite{GYM2012,HG2012,Liu2015,JL2014} could be applied
if larger problems were involved, but our examples did not warrant
such methods.

\begin{remark}
  Most of the following numerical experiments show accuracies of
  approximately $10^{-10}$ in relative precision. This is to be
  expected: the modal Green's functions are computed to 12 digits of
  relative accuracy (as mentioned above), and the data is resolved to
  an $L^2$ norm of $10^{-12}$. The remaining loss of precision arise
  from the inherent condition number of the problem, which remains
  very small at modestly sized frequencies because of the second-kind
  integral equation formulation. In addition, we apply~$L^2$
  weighting, as described in~\cite{Brem2012}, to the matrix elements
  in order to handle any non-physical ill-conditioning due to dyadic
  refinement along the generating curve. When~$k_0=10$ and~$k_1=5$,
  the approximate condition number of the 0th-mode system matrices for
  examples with smooth geometries in Section~\ref{sec:ex1}
  and~\ref{sec:ex2} is 4.93E3.  The approximate condition number of
  the 0th-mode system matrix in Section~\ref{sec:ex3} is 1.35E3 and in
  Section~\ref{sec:ex4} is 1.43E4 (obtained by the
  function~\texttt{ZGESVX} in~\textsc{LAPACK}).
\end{remark}

\begin{remark} It is generally difficult to analytically parameterize the
generating curve with respect to arclength. In the following examples,
instead of resampling the curve to obtain an arclength discretization,
the generating curve is sampled at Legendre nodes in~\emph{some}
parameter~$t$, not necessarily arclength. Any previous formulae based
on an arclength parameterization
can be easily adjusted with factors of~$ds/dt$ to
account for the change of variables.
\end{remark}

\begin{remark}
  The examples we tested here are all for the scattering of a single
  object. The algorithm, however, can be extended to the scattering of
  multiple axisymmetric objects by transforming the incident field
  into the local coordinate of each object\cite{GG2013,hao_2014}. This
  will be considered  in future work.
\end{remark}

We make use of the following notation in the subsequent tables that
present data from our scattering experiments:
\begin{itemize}
\item $k_0$: the exterior wavenumber,
\item $k_1$: the interior wavenumber,
\item $N_f$: the number of Fourier modes in the azimuthal direction
  used to resolve the solution. In other words, the Fourier modes are
  $-N_f,-N_f+1, \dots, N_f-1,N_f$.
\item $N_{pts}$: the total number of points used to
  discretize~$\gamma$,
\item $T_{kernel}$: the time (seconds) to evaluate all the relevant modal Green's functions,
\item $T_{matgen}$: the time (seconds) to construct the relevant
  matrix entries for all integral equations,
\item $T_{solve}$: the time (seconds) to solve the linear system by 
  $\mtx{LU}$-factorization for all modes,
\item $T_{add}$: the time (seconds) to solve with an additional
  right-hand side once the matrix is factorized,
\item $E_{error}$: the relative $\ell^2$ error of the electric and magnetic
  fields measured at a few points randomly placed inside~$\Omega$.
\end{itemize}

\subsection{Example 1: Scattering from a Torus}
\label{sec:ex1}
 
Consider a torus with the generating curve given by
\begin{equation}
  \begin{aligned}
    r(t) &= 2+\cos(t), \\
    z(t) &= 0.5\sin(t),
  \end{aligned}
\end{equation}
for~$t\in [0,2\pi)$. The accuracy of the integral equation solver was
tested in this geometry using the extinction theorem (described above
in Section~\ref{sec_numeri}) with known boundary data and fields given
in equation~\eqref{eq:loop}. The accuracy of the solver was verified
at several wavenumbers, and results are shown in Table~\ref{table_1}.
Using around 25 points per wavelength, approximately $8$ digits of
relative accuracy in the fields were obtained for most of the cases.
Note that the CPU time is dominated by the formation of the system
matrix, which (as expected) roughly scales quadratically with the
number of unknowns. For the same number of unknowns, the computational
time depends linearly on the number of Fourier modes. Although the
linear systems are decoupled across modes (this makes parallelization
straightforward), we present the computational time~$T_{solve}$ in
Table~\ref{table_1} as the total matrix inversion time via sequential
solve.  Despite its~$\mathcal{O}(N_fN_{pts}^3)$ complexity, it is
still much smaller than~$T_{matgen}$.  Once the matrix is factorized,
additional solves for new right hand sides are very fast.

For an incident plane wave, given in equation~\eqref{eq:pw},
Figure~\ref{figure_torus} shows the real~part of the $x$-component of
the electric current~$\bJ$, which is $\Re(\bn\times \bh)_x$ by~\eqref{emcurrent},
and the far field pattern at $\phi = \frac{\pi}{2}$ and
$\theta=0$, respectively. The exterior wavenumber is set to
be~$k_0 = 10.0$ and the interior is set to be~$k_1=5.0$.  When
$\phi = \pi/2$, we see a maximum at $\theta = \pi/3$ and a minimum
near~$\theta = 4\pi/3$. This is due to specular reflection.

Since the exact solution to the true scattering problem is not known,
in order to verify the accuracy of our solver
we perform a self-consistent convergence study
on the far field pattern. The error is obtained by comparing with the
result obtained using~$20$ panels with~$k_0=10$ and~$k_1=5$, and is measured
in the~$L^2$ norm. Figure~\ref{fig:conv1} shows that the far
field converges exponentially fast when the number of panels increases
by~2; we scanned through discretizations with 4 panels to 18 panels.

\subsection{Example 2: Scattering from a Rotated Starfish}
\label{sec:ex2}

For the second example, we consider an axisymmetric object with
generating curve
\begin{equation}
  \begin{aligned}
    r(t) &= [2+0.5\cos(5\pi (t-1))]\, \cos(\pi(t-0.5)),  \\
    z(t) &= [2+0.5\cos(5\pi (t-1))]\, \sin(\pi(t-0.5)),
  \end{aligned} 
\end{equation}
for $t\in [0,1]$. We refer to this object as the rotated
\emph{starfish}, as shown in Figure~\ref{figure_wigg}.
As before, 
the accuracy of the integral equation solver was tested in this
geometry using the extinction theorem (described in
Section~\ref{sec_numeri}) with known boundary data and fields given in
equation~\eqref{eq:loop}. The accuracy of the solver was verified at
several wavenumbers, and results are shown in Table~\ref{table_2}.

\afterpage{
 \begin{figure}[h]
   \centering
   \begin{subfigure}[b]{.38\linewidth}
     \centering
     \includegraphics[width=.95\linewidth]{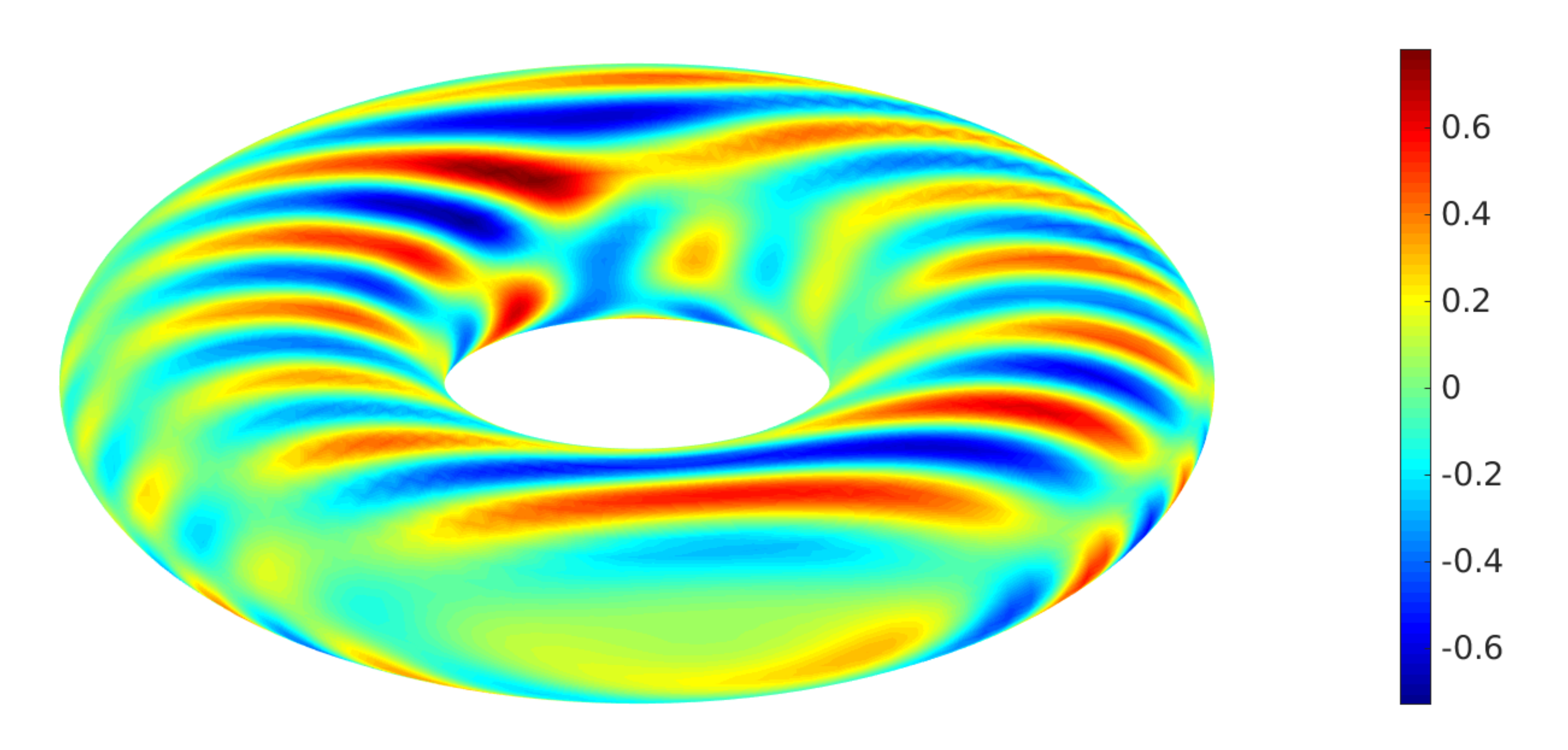}
     \caption{$\Re J_x$.}
   \end{subfigure}
   \hfill
   \begin{subfigure}[b]{.3\linewidth}
     \centering
     \includegraphics[width=.95\linewidth]{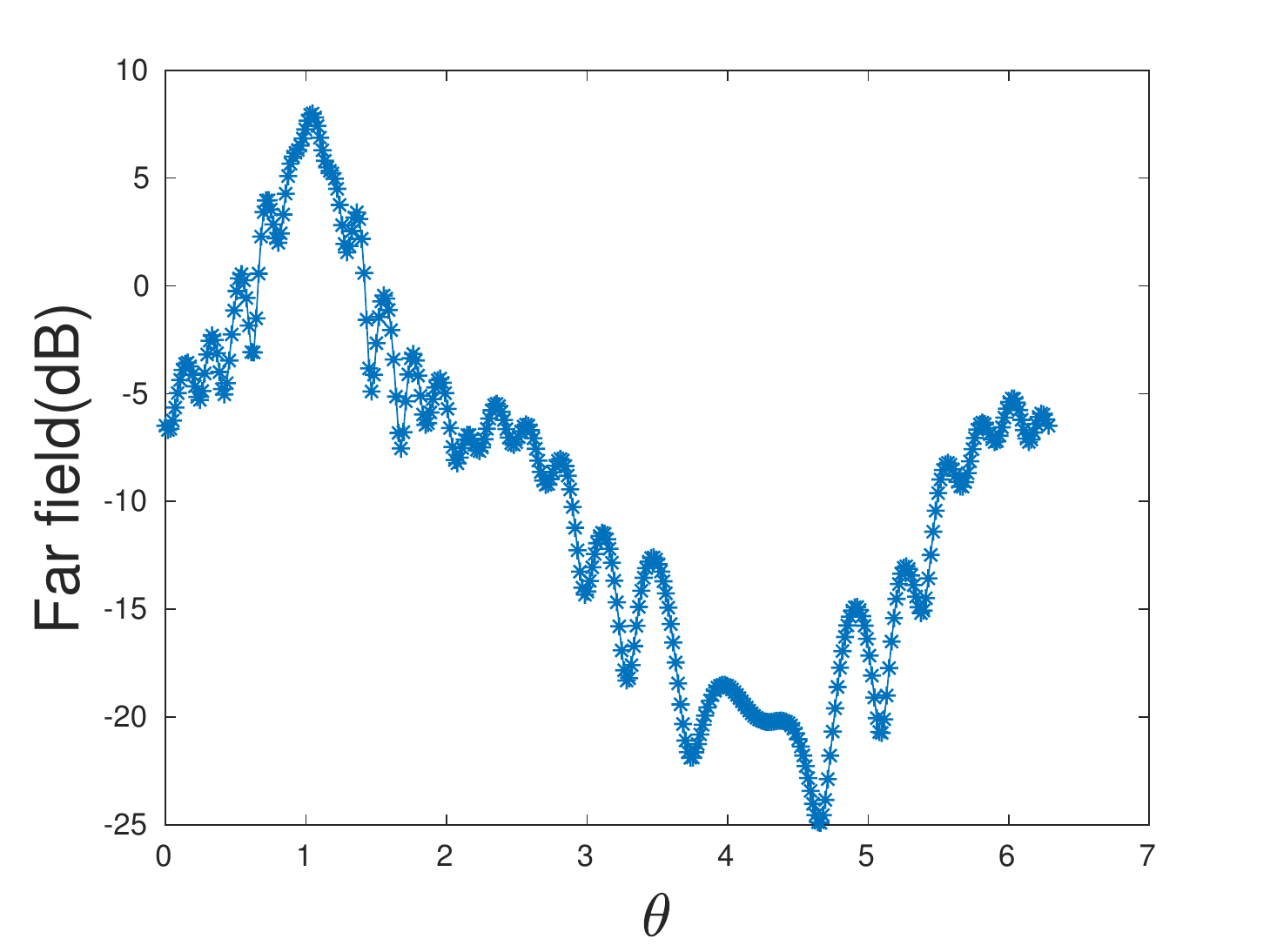}
     \caption{$|\besc_\infty(\cdot,\pi/2)|$.}
   \end{subfigure}
   \hfill
   \begin{subfigure}[b]{.3\linewidth}
     \centering
     \includegraphics[width=.95\linewidth]{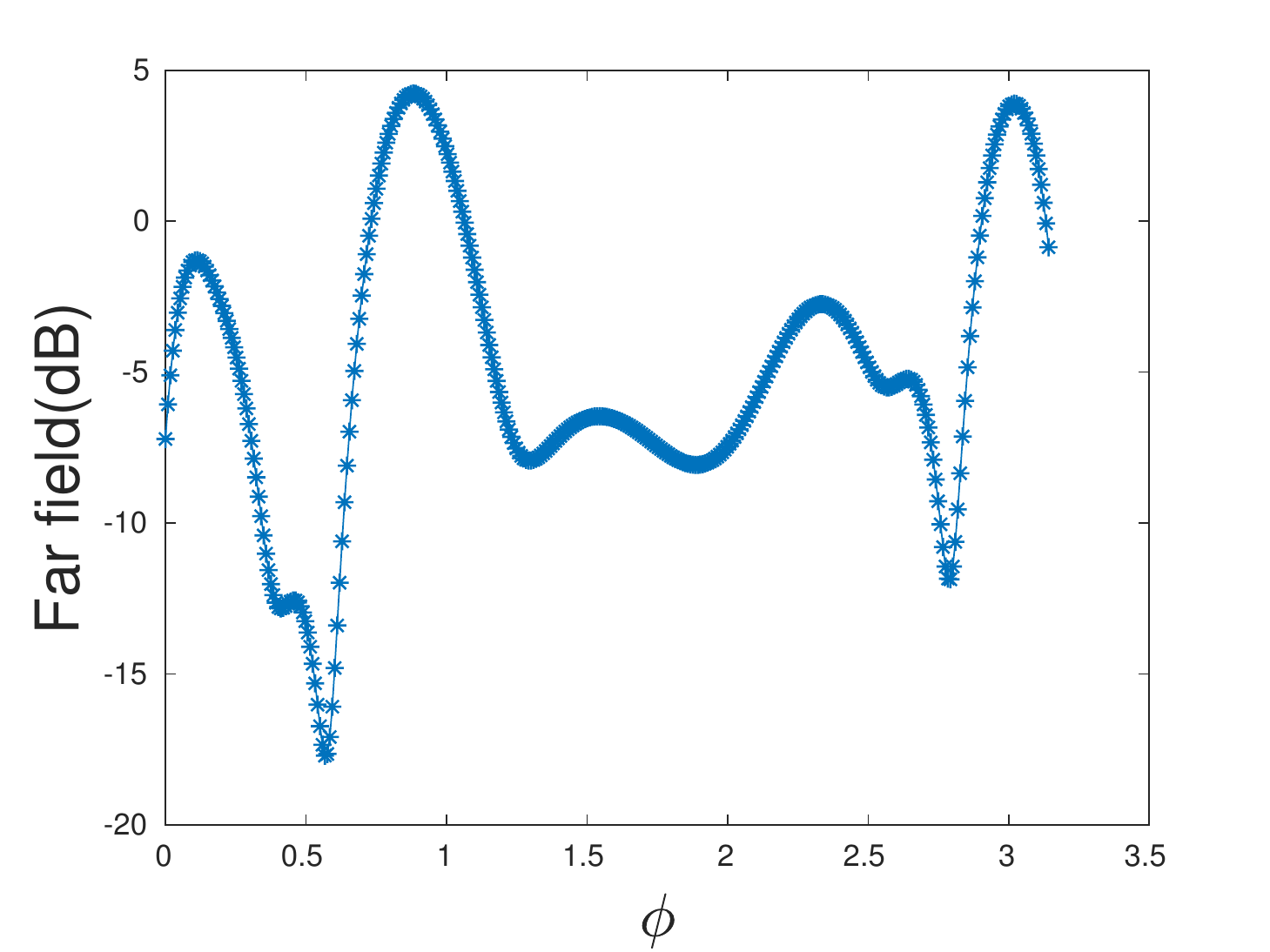}
     \caption{$|\besc_\infty(0,\cdot)|$.}
   \end{subfigure}
   \caption{A penetrable torus with interior~$k_1=5.0$ and
     background~$k_0=10.0$.}
   \label{figure_torus}
 \end{figure}

\begin{figure}[h]
  \centering
  \begin{subfigure}[b]{.38\linewidth}
    \centering
    \includegraphics[width=.95\linewidth]{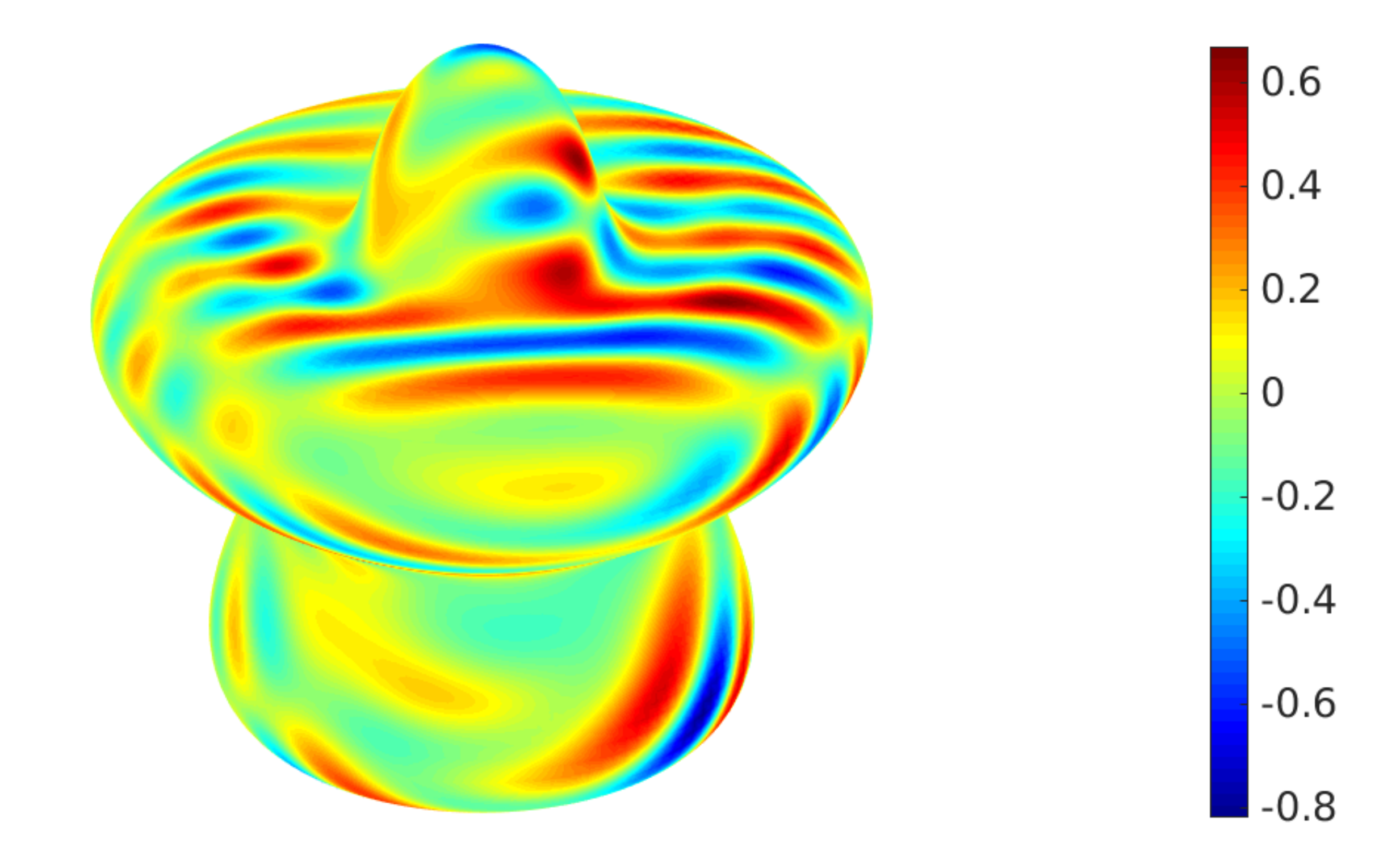}
    \caption{$\Re J_x$.}
  \end{subfigure}
  \hfill
  \begin{subfigure}[b]{.3\linewidth}
    \centering
    \includegraphics[width=.95\linewidth]{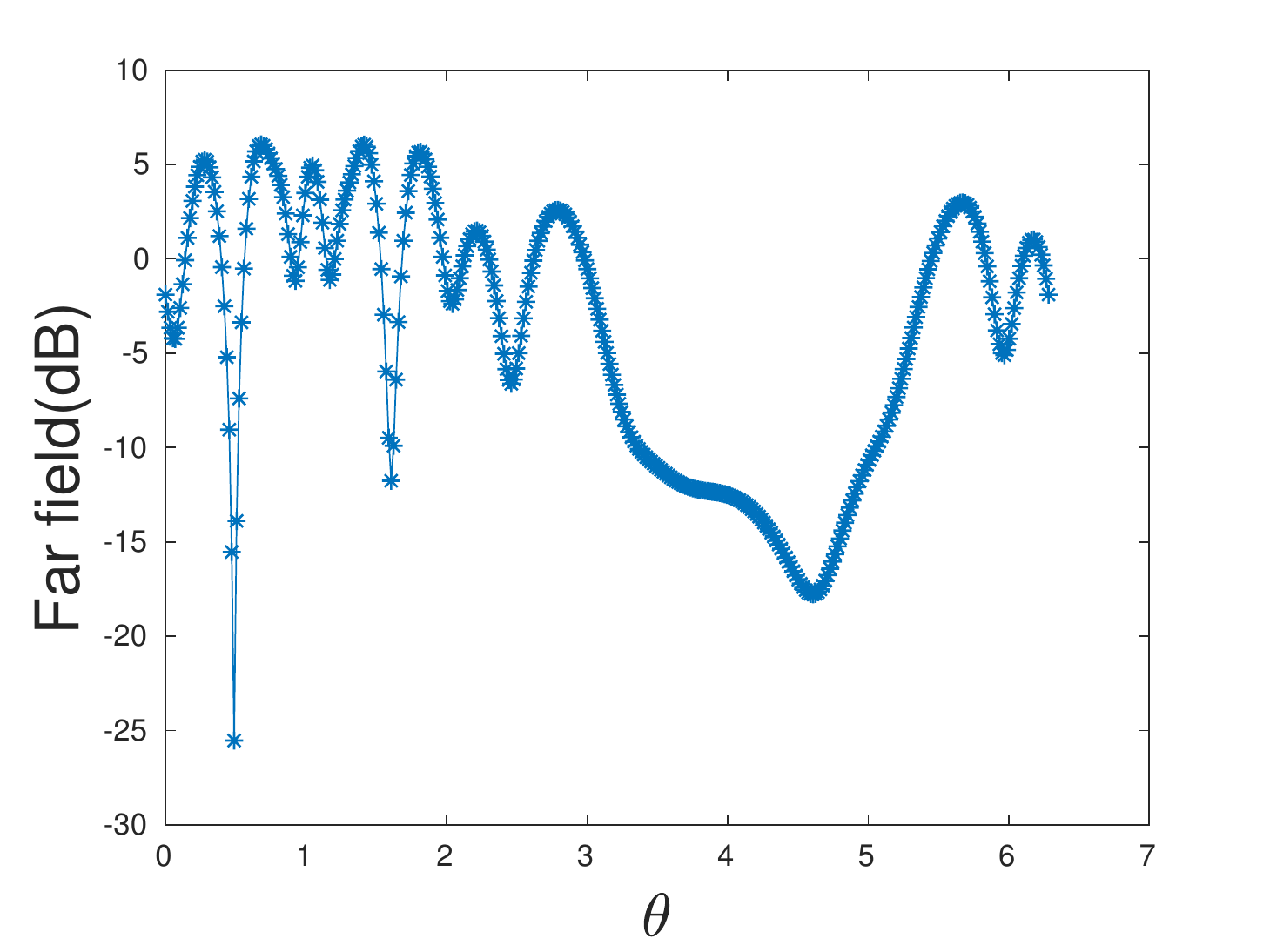}
    \caption{$|\besc_\infty(\cdot,\pi/2)|$.}
  \end{subfigure}
  \hfill
  \begin{subfigure}[b]{.3\linewidth}
    \centering
    \includegraphics[width=.95\linewidth]{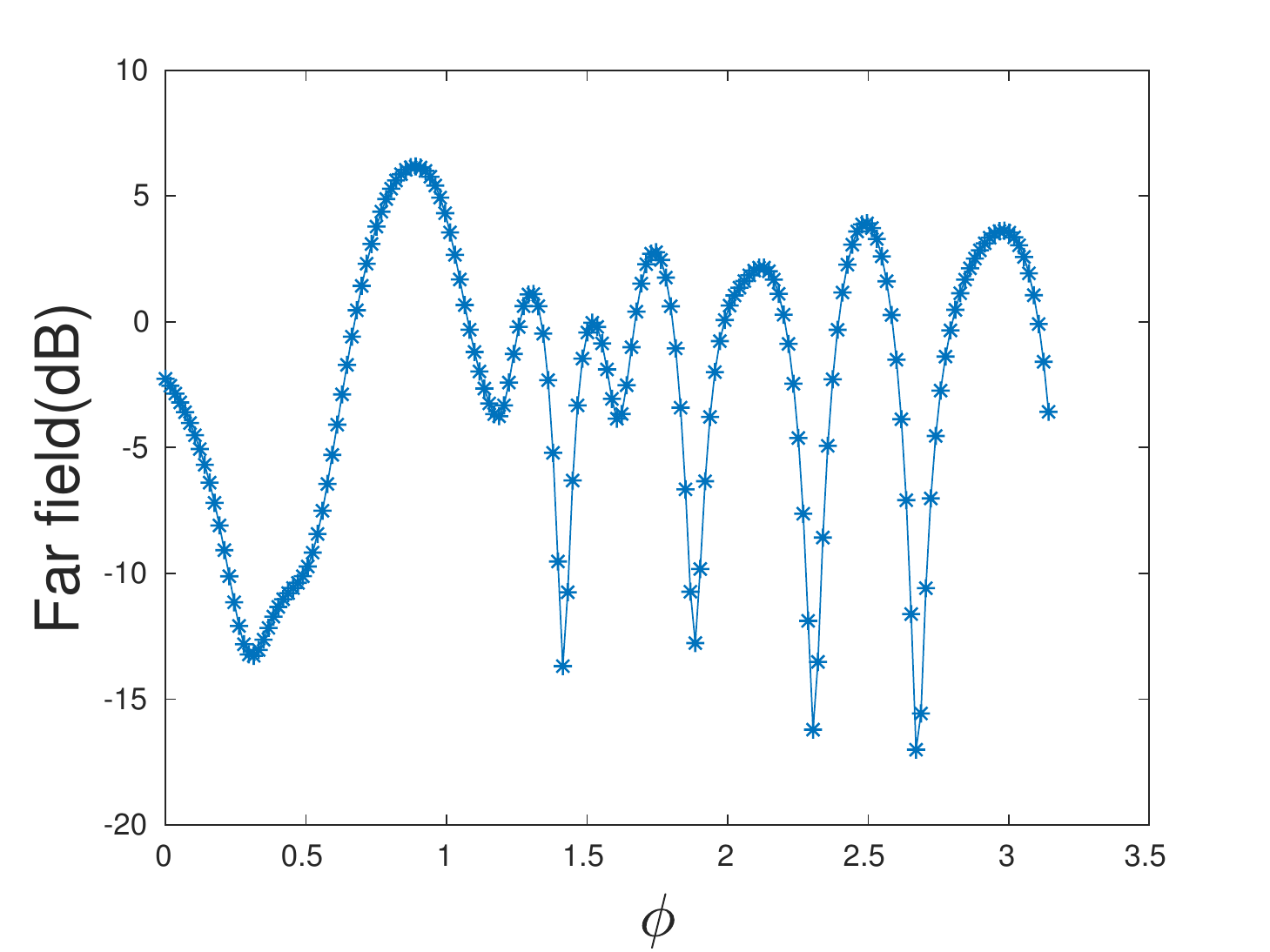}
    \caption{$|\besc_\infty(0,\cdot)|$.}
  \end{subfigure}
  \caption{A penetrable star-fish~$\Omega$  with interior~$k_1=5.0$
    and background~$k_0=10.0$.}
  \label{figure_wigg}
\end{figure}

\begin{figure}[h]
  \centering
  \begin{subfigure}[b]{.38\linewidth}
    \centering
    \includegraphics[width=.95\linewidth]{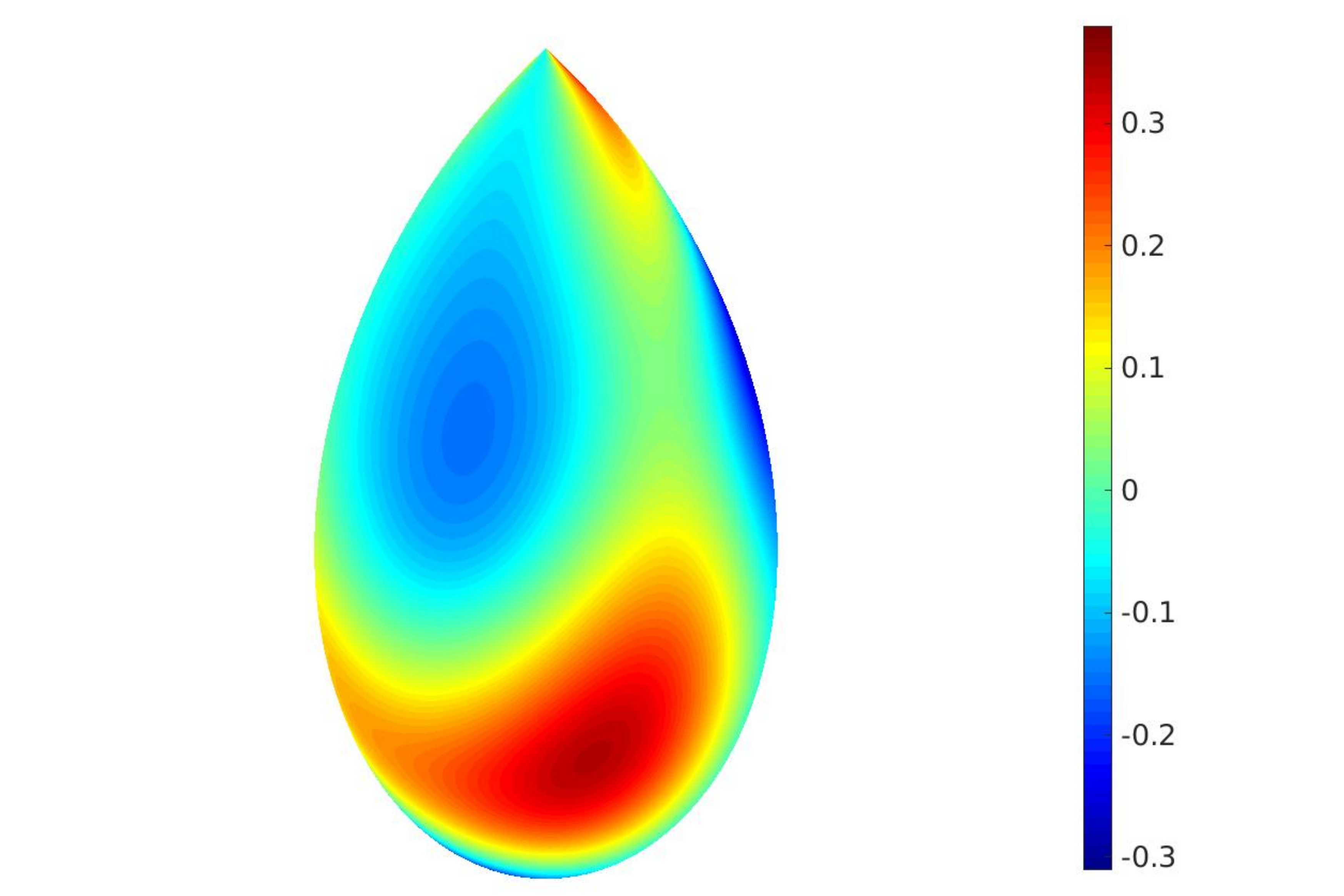}
    \caption{$\Re J_x$.}
  \end{subfigure}
  \hfill
  \begin{subfigure}[b]{.3\linewidth}
    \centering
    \includegraphics[width=.95\linewidth]{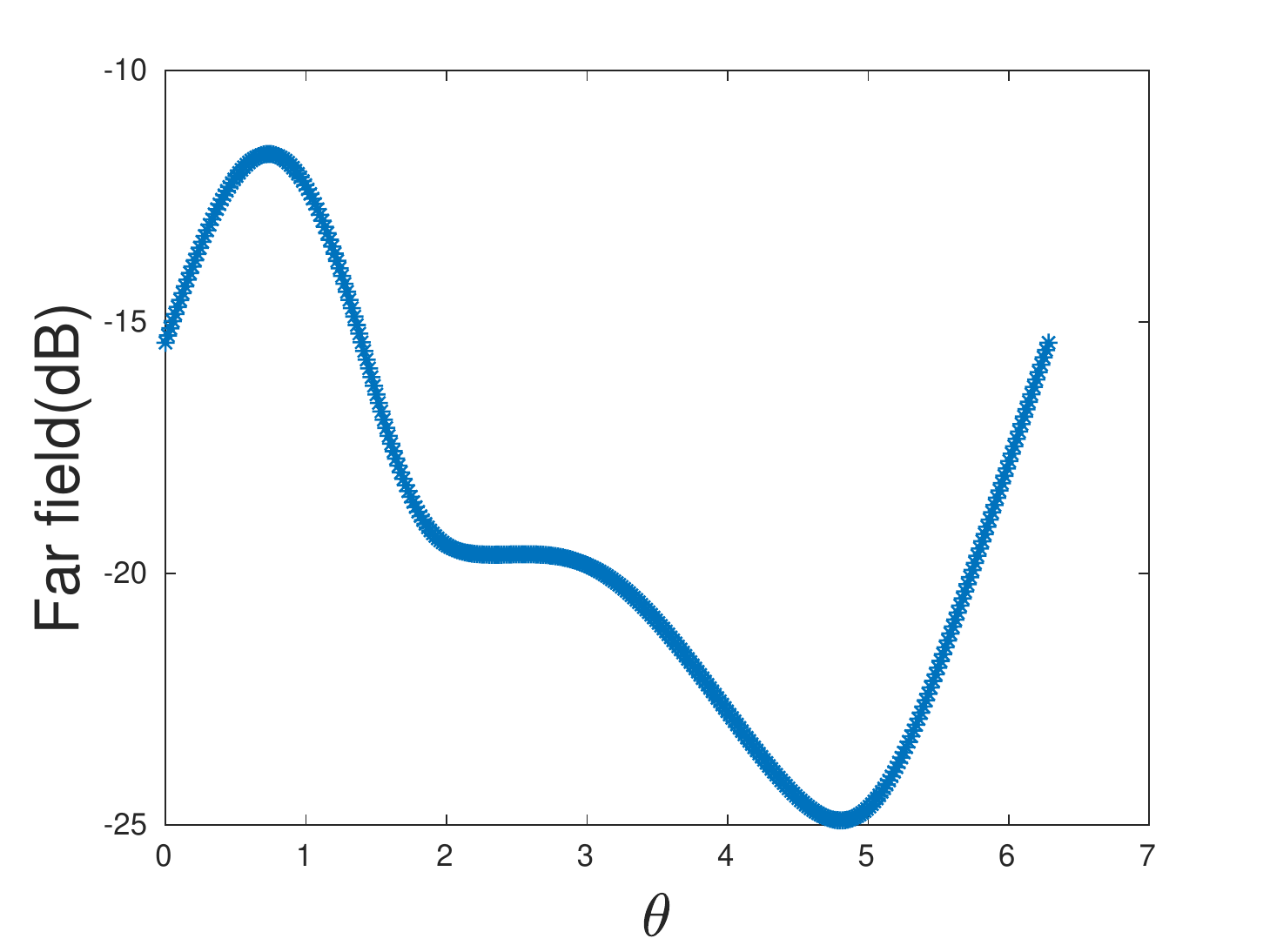}
    \caption{$|\besc_\infty(\cdot,\pi/2)|$.}
  \end{subfigure}
  \hfill
  \begin{subfigure}[b]{.3\linewidth}
    \centering
    \includegraphics[width=.95\linewidth]{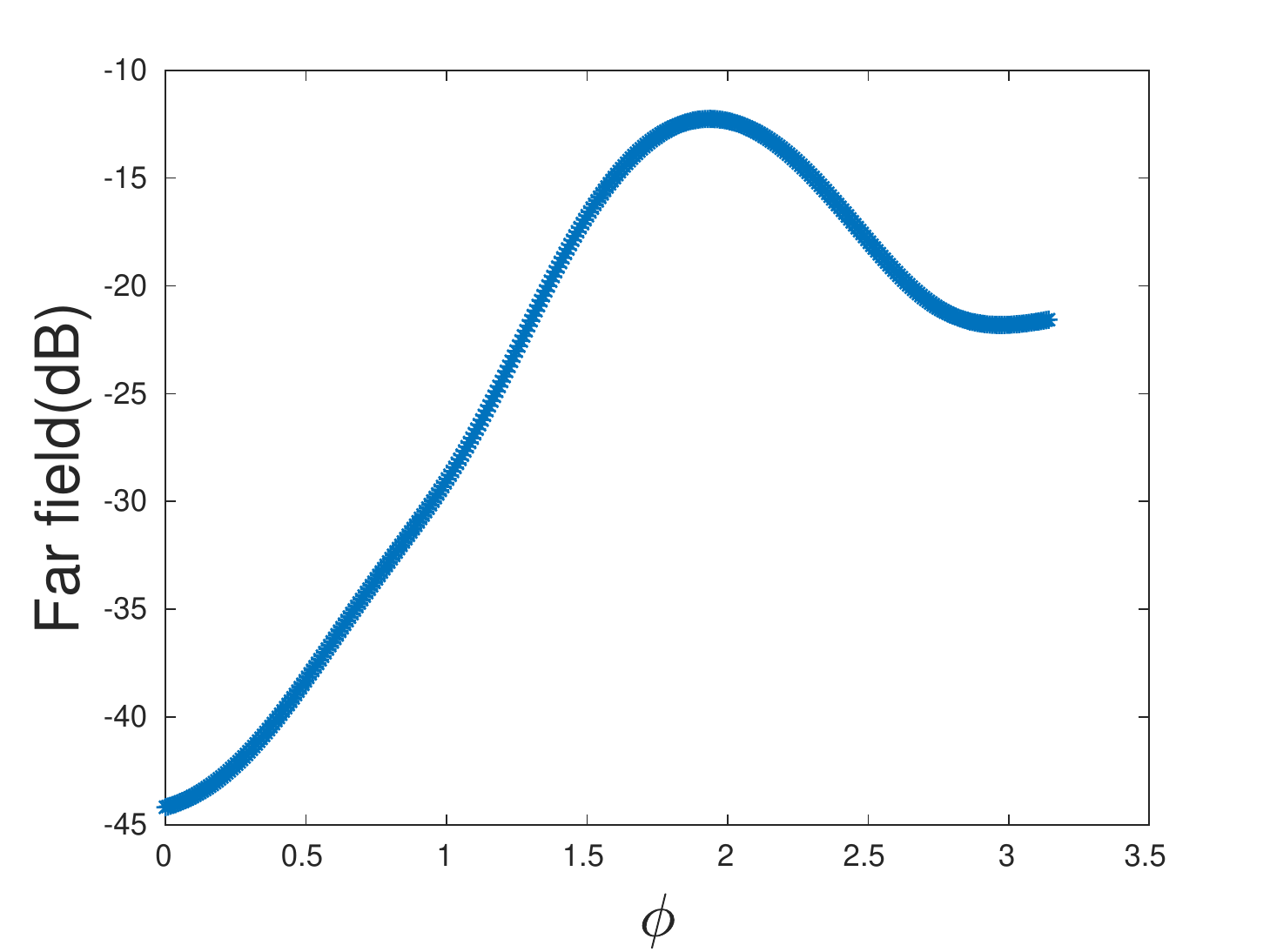}
    \caption{$|\besc_\infty(0,\cdot)|$.}
  \end{subfigure}
  \caption{A penetrable droplet~$\Omega$  with interior~$k_1=5.0$
    and background~$k_0=10.0$.}
  \label{figure_drop}
\end{figure}
\clearpage}

Although the generating curve~$\gamma$ is open (when viewed
as~$\gamma \subset \mathbb R^2$), the object is globally smooth. We
therefore apply a uniform panel discretization in the parameter
space~$[0,1]$. Table~\ref{table_2} provides the accuracy results at
various wavenumbers and discretization refinements.  We easily
obtain~$8$ to~$9$ digits of accuracy by using a sufficient number of
discretization points per wavelength. The computational time is again
dominated by the matrix generation. Once the matrix is generated and
factored, the time for additional solves is negligible.

Figure~\ref{figure_wigg} shows the scattering behavior for an incident
plane wave with~$k_1=5$ and~$k_0 = 10$. The far field pattern
oscillates around~$\theta=\pi/3$. This effect is a combination of
specular reflection and the changing convexity of the geometry. The
far field is notably small near~\mbox{$\theta=4\pi/3$}; this location
is in the shadow region with respect to the direction of the incident
wave.  Again, to verify the accuracy of our solver,
Figure~\ref{fig:conv2} provides the results of a self-consistent
convergence test of the far field pattern.  We obtain exponential
convergence when the number of panels is increased by 2, as before.

\begin{figure}[t]
  \centering
  \begin{subfigure}[b]{.32\linewidth}
    \centering
    \includegraphics[width=.95\linewidth]{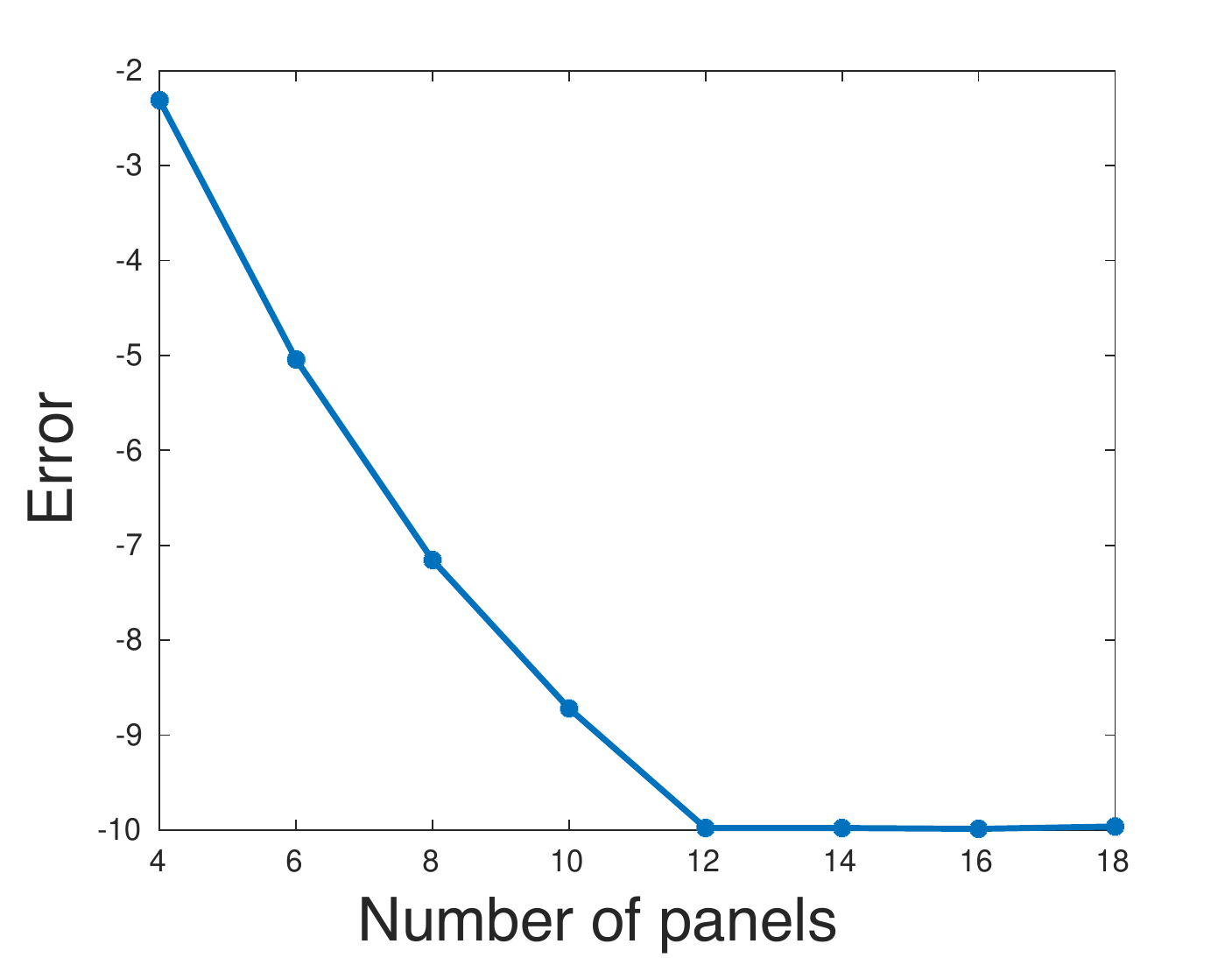}
    \caption{Ex. 1: The torus.}
    \label{fig:conv1}
  \end{subfigure}
  \hfill
  \begin{subfigure}[b]{.32\linewidth}
    \centering
    \includegraphics[width=.95\linewidth]{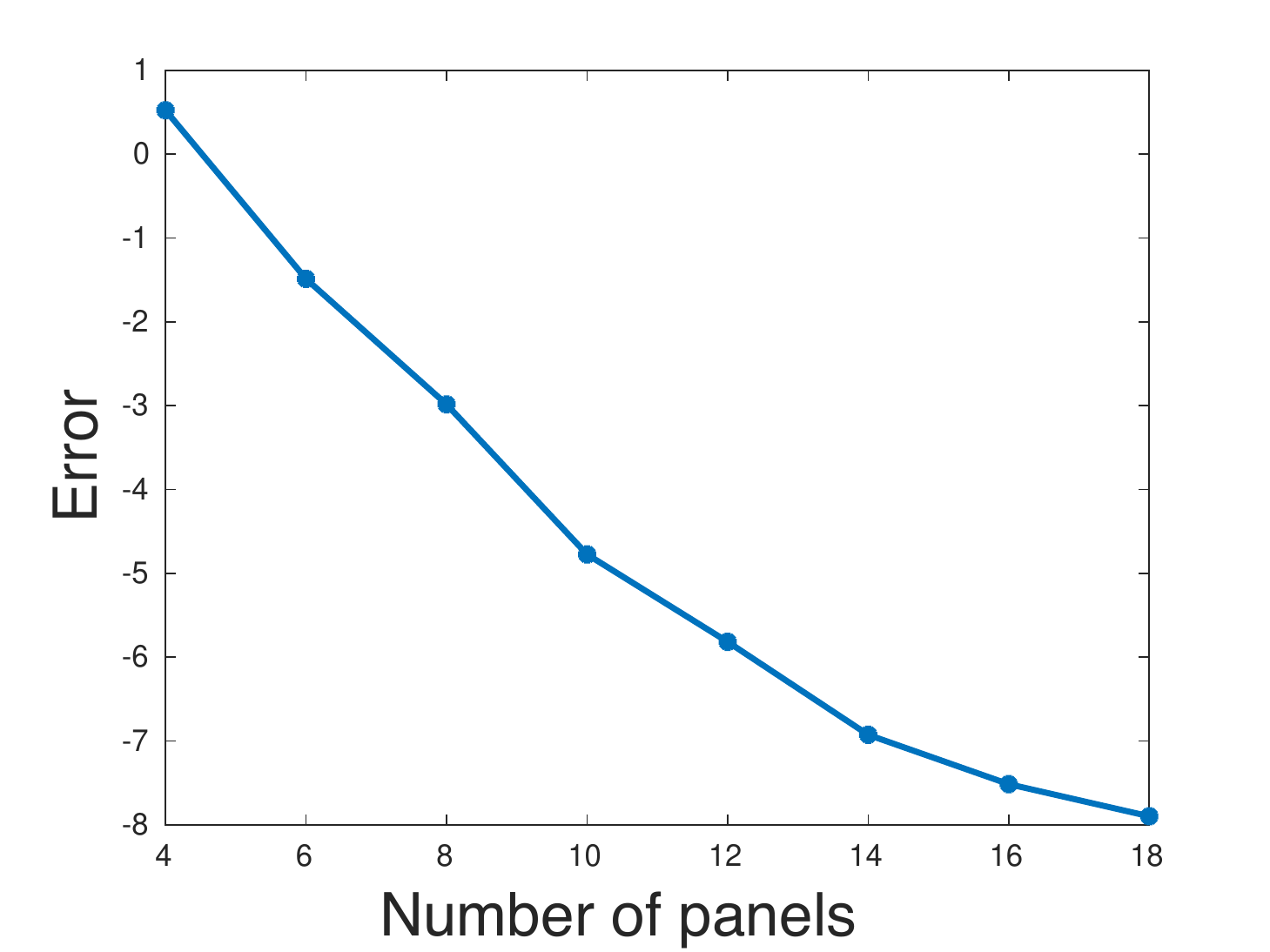}
    \caption{Ex. 2: The rotated starfish.}
    \label{fig:conv2}
  \end{subfigure}
  \hfill
  \begin{subfigure}[b]{.32\linewidth}
    \centering
    \includegraphics[width=.95\linewidth]{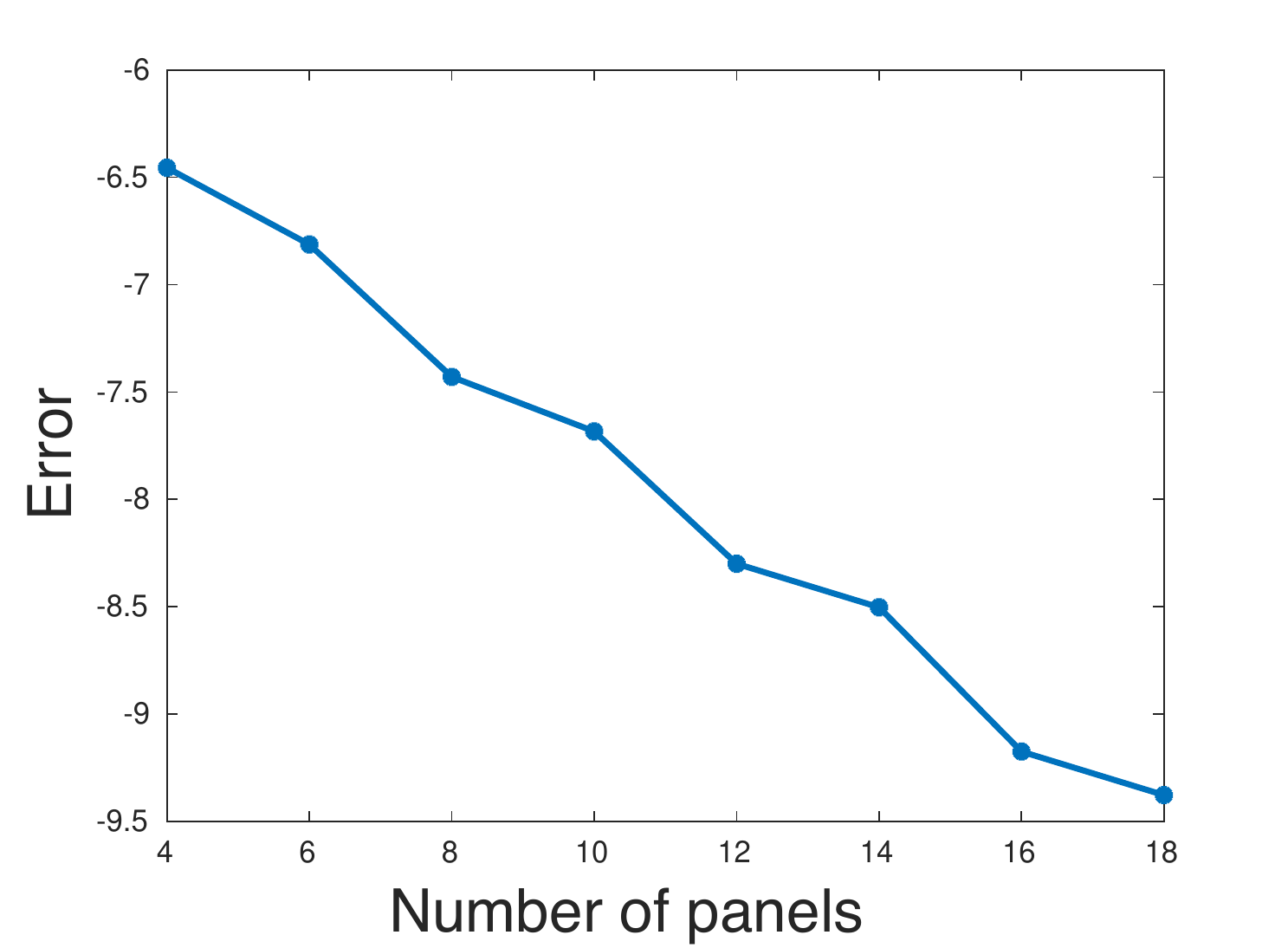}
    \caption{Ex. 3: The droplet.}
    \label{fig:conv3}
  \end{subfigure}
  \caption{A self-consistent convergence study on the far field
    pattern of three geometries with interior $k_1=5.0$ and
    background~$k_0=10.0$. The error of the far field is given on a
    $\log_{10}$-scale.}
  \label{figure_convergence}
\end{figure}

\subsection{Example 3: Scattering from a Droplet}
\label{sec:ex3}

In this example, we consider scattering from a droplet whose
generating curve is parameterized as:
\begin{equation}
  \begin{aligned}
    r(t) &= \sin(\pi t)\cos[0.5\pi(t-1.5)],  \\
    z(t) &= \sin(\pi t)\sin[0.5\pi(t-1.5)]+0.5,
  \end{aligned}
\end{equation}
for $t\in[0.5,1]$. As is clear in~Figure \ref{figure_drop}, there is a
point singularity at $t=1$ on the $z$-axis.
To resolve this singularity, we first compute
a uniform length panel discretization in the parameter space
$[0.5,1]$. Five dyadic refinements are then performed along the panel
adjacent to the end point, yielding a graded mesh.
As illustrated in Figure~\ref{fig:dyadic_droplet}, the last panel near the
end point is of size 8.18E-3 when~$k_1=5$ and~$k_0 = 10$. 

Once again, the accuracy of the integral equation solver was tested in this
geometry using the extinction theorem (described in
Section~\ref{sec_numeri}) with known boundary data and fields given in
equation~\eqref{eq:loop}. The accuracy of the solver was verified at
several wavenumbers, and results are shown in Table~\ref{table_3}.
More than 8 digits of accuracy was obtained in the tests
at various wavenumbers, which implies that the solution to the integral
equation near the point singularity of the droplet is also resolved to this
accuracy. 

\begin{figure}[t]
  \centering
  \begin{subfigure}[b]{.3\linewidth}
    \centering
    \includegraphics[width=.95\linewidth]{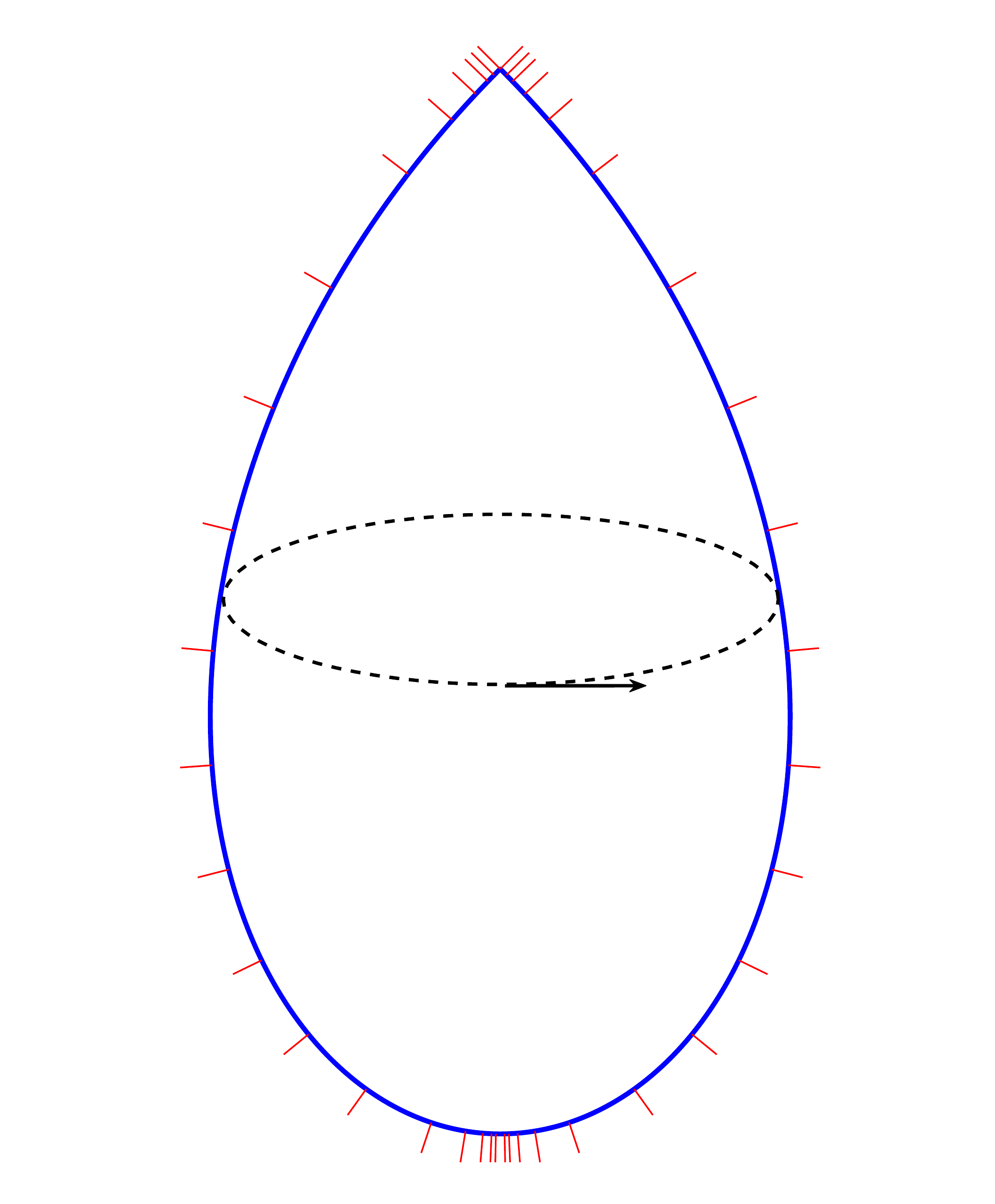}
    \caption{The droplet.}
    \label{fig:dyadic_droplet}
  \end{subfigure}
  \hspace{.5in}
  \begin{subfigure}[b]{.4\linewidth}
    \centering
    \includegraphics[width=.95\linewidth]{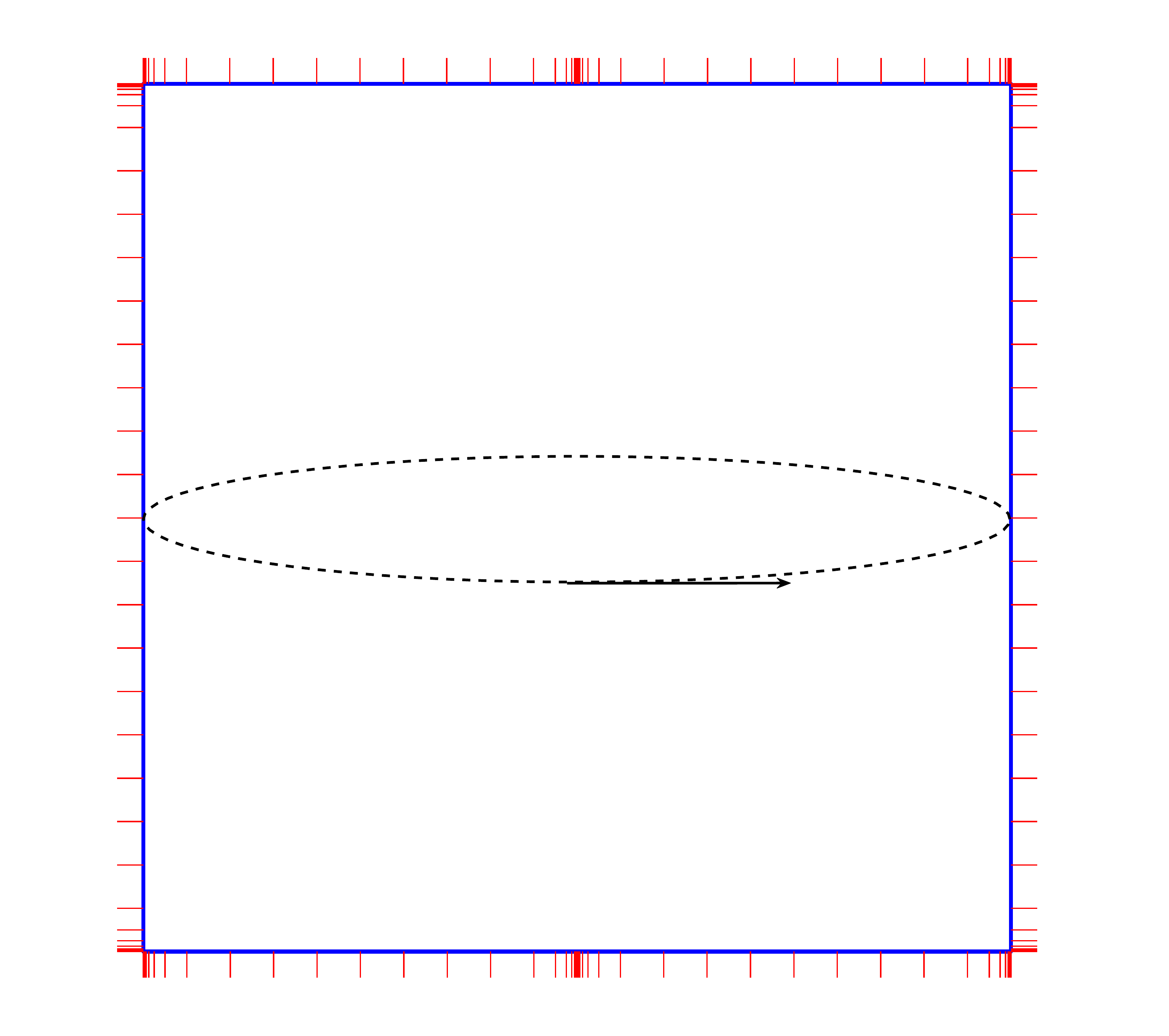}
    \caption{The cylinder.}
    \label{fig:dyadic_cylinder}
  \end{subfigure}
  \hfill
  \caption{The dyadic refinement on the generating curves for
    non-smooth geometries.}
  \label{figure_dyadic}
\end{figure}

Figure~\ref{figure_drop} then illustrates scattering of an incident
plane wave for~$k_1=5$ and~$k_0 = 10$. Unlike the previous two
examples, the far field pattern at $\phi = \pi/ 2$ and $\theta=0$ are
quite smooth due to the particular geometry and observation
angle. However, the far field pattern reaches a the minimum near
$\theta=4\pi/3$ due to the specular reflection.
Similarly, as before, self-consistent convergence results on the far field pattern
are reported in Figure~\ref{fig:conv3}.
The plot demonstrates exponential convergence when the
number of panels is increased by 2, starting at 4 panels.

\subsection{Example 4: Scattering from a Cylinder}
\label{sec:ex4}

Finally, we consider scattering from a cylindrical geometry whose
generating curve has vertices given by
\begin{equation}
  V = \left\{
    (0,-1), \, (1,-1), \, (1,1), \, (0,1)
  \right\},
\end{equation}  
see Figure~\ref{figure_cylinder}. Once again, the generating curve is
open but gives rise to a smooth surface when crossing
the~$z$-axis. However, the cylinder contains an edge at the top and
bottom which necessitates dyadic refinement to ensure accuracy.
 In particular, on the two
panels that are adjacent to the corner, dyadic refinement is performed
such that the size of the last panel is on the order of 1E-5,
as illustrated in Figure~\ref{fig:dyadic_cylinder}. This is
of course deep sub-wavelength, and is small enough to resolve the
solution to the integral equation to sufficiently high accuracy.

The accuracy of the integral equation solver was again tested in this
geometry using the extinction theorem (described in
Section~\ref{sec_numeri}) with known boundary data and fields given in
equation~\eqref{eq:loop}. The accuracy of the solver was verified at
several wavenumbers, and results are shown in Table~\ref{table_4}.
At small wavenumbers, we obtain approximately~8 digits of
accuracy; this accuracy slowly deteriorates as the wavenumber
increases.  This is due to a stronger singularity near the corner at
higher wavenumbers (since the characteristic length-scale of the
singularity is on the order of the wavelength).  More digits can be
obtained if additional refinement were implemented.

For an incident plane wave with~$k_0=10.0$ and~$k_1=5.0$,
Figure~\ref{figure_cylinder} plots the real-part of the
  $x$-component of $\bn\times \bh$  and the far field
pattern.  At~$\phi=0$ in
the far field pattern, we again see strong and weak scattering
near~$\theta=\pi/3$ and~$\theta = 4\pi/3$, respectively. This effect
is due to similar reasons as in the previous three examples.

\afterpage{
\begin{table}[t]
  \centering
  \caption{Accuracy results for scattering from objects in
    Figures~\ref{figure_torus}-\ref{figure_cylinder}.}
  \label{tab:tables}
  \begin{subtable}[t]{.48\textwidth}
    \centering
    \caption{The torus.}
    \label{table_1}
    \resizebox{\textwidth}{!} {
      \begin{tabular}{|c|c|c|c|c|c|c|c|c|}
        \hline
        $k_0$ &$k_1$ & $N_f$ & $N_{pts}$ &$T_{kernel}$& $T_{matgen}$ & $T_{solve}$ &$T_{add}$ & $E_{error}$  \\
        \hline
        $1$ & $2$ & 13 & 64 &2.76E+0 & 2.98E+0 & 6.48E-2 & 1.08E-2 & 2.32E-9\\
        $1$ &$5$ & 16 &160 &6.09E+0 & 7.95E+0 & 3.29E-1 & 2.64E-2 & 4.61E-9 \\
        $1$ &$10$ & 21 &320 &1.28E+1 & 2.51E+1 & 3.37E+0  & 8.59E-2 & 3.11E-8 \\
        \hline
        $5$ &$2$ & 13 &160 &6.13E+0 & 7.65E+0 & 3.56E-1 & 3.24E-2 & 1.01E-9 \\
        $5$ &$10$ & 21 & 320 &1.27E+1 & 2.49E+1 & 2.35E+0 & 8.59E-2 & 3.84E-9 \\
        $5$ &$20$ & 30 & 640 &6.12E+1 & 1.38E+2 & 2.53E+1 & 5.85E-1 & 4.63E-9 \\
        \hline
        $10$ &$5$ & 17 & 320 &1.32E+1 & 2.36E+1 & 1.97E+0 & 6.99E-2 & 1.61E-9 \\
        $10$ &$20$ & 30 & 640 &6.14E+1 & 1.37E+2 & 2.49E+1 & 5.36E-1 & 2.08E-9 \\
        $10$ &$40$ & 45 & 640 &9.14E+1 & 2.06E+2 & 3.49E+1 & 8.37E-1 &2.66E-8\\
        \hline
        $20$ &$5$ & 17 &320 &1.76E+1 & 2.77E+1 & 2.95E+0 & 9.8E-2 & 4.78E-9 \\
        $20$ &$10$ & 22 & 640 &6.09E+1 & 1.17E+2 & 1.83E+1 & 4.14E-1 & 3.37E-9 \\
        $20$ &$40$ & 45 & 640 &9.64E+1 & 2.11E+2 & 3.30E+1 & 8.37E-1 & 1.10E-8 \\
        \hline
      \end{tabular}
    }
  \end{subtable}
  \hfill
  \begin{subtable}[t]{.48\textwidth}
    \centering
    \caption{The rotated starfish.}
    \label{table_2}
    \resizebox{\textwidth}{!} {
      \begin{tabular}{|c|c|c|c|c|c|c|c|c|}
        \hline
        $k_0$ &$k_1$ & $N_f$ & $N_{pts}$ & $T_{kernel}$& $T_{matgen}$ & $T_{solve}$ &$T_{add}$ & $E_{error}$  \\
        \hline
        $5$ & $2$ & 10 & 224 &6.47E+0 & 9.34E+0 & 4.61E-1 & 3.35E-2 & 1.24E-9\\
        $5$ &$10$ & 16 &304 &9.49E+0 & 1.80E+1 & 1.57E+0 & 6.59E-2 & 3.44E-10 \\
        $5$ &$20$ & 21 &464 &2.96E+1 & 6.19E+1 & 7.01E+0  & 2.06E-1 & 1.51E-9 \\
        \hline
        $10$ &$5$ & 13 &304 &9.82E+0 & 1.69E+1 & 1.78E+0 & 6.48E-2 & 2.04E-10 \\
        $10$ &$20$ & 21 & 464 &3.01E+1 & 5.94E+1 & 7.21E+0 & 2.59E-1 & 9.74E-10 \\
        $10$ &$40$ & 30 & 784 &7.73E+1 & 2.09E+2 & 3.98E+1 & 7.31E-1 & 1.45E-8 \\
        \hline
        $20$ &$5$ & 13 & 464 &2.90E+1 & 4.93E+1 & 1.97E+1 & 6.99E-1 & 1.05E-10 \\
        $20$ &$10$ & 16 & 464 &2.98E+1 & 5.50E+1 & 7.78E+1 & 1.84E-1 & 4.41E-10 \\
        $20$ &$40$ & 30 & 784 &7.85E+1 & 2.17E+2 & 4.02E+1 & 8.04E-1 & 4.60E-9 \\
        \hline
        $40$ &$5$ & 13 &784 &7.41E+1 & 1.39E+2 & 2.02E+1 & 3.34E-1 & 3.39E-10 \\
        $40$ &$10$ & 16 & 784 &7.56E+1 & 1.55E+2 & 2.58E+1 & 4.46E-1 & 3.58E-9 \\
        $40$ &$20$ & 21 & 784 &7.75E+1 & 1.73E+2 & 3.32E+1 & 5.01E-1 & 3.01E-9 \\
        \hline
      \end{tabular}
    }    
  \end{subtable}
  \\ \vspace{\baselineskip}
  \begin{subtable}[t]{.48\textwidth}
    \centering
    \caption{The droplet.}
    \label{table_3}
    \resizebox{\textwidth}{!} {
      \begin{tabular}{|c|c|c|c|c|c|c|c|c|}
        \hline
        $k_0$ &$k_1$ & $N_f$ & $N_{pts}$ & $T_{kernel}$ & $T_{matgen}$
        &$T_{solve}$ &$T_{add}$ & $E_{error}$  \\
        \hline
        $5$ &$2$ & 7 &224 &1.67E+1 & 1.89E+1 & 3.24E-1 & 6.01E-2 & 2.71E-10 \\
        $5$ &$10$ & 10 & 320 &3.05E+1 & 3.72E+1 & 7.14E-1 & 8.39E-2 & 1.86E-9 \\
        $5$ &$20$ & 11 & 480 &7.52E+1 & 9.43E+1 & 2.39E+0 & 1.84E-1 & 5.11E-9 \\
        \hline
        $10$ &$5$ & 9 & 320 &3.05E+1 & 3.67E+1 & 5.47E-1 & 7.59E-2 & 1.24E-10 \\
        $10$ &$20$ & 11 & 480 &7.80E+1 & 9.69E+1 & 2.45E+0 & 1.84E-1 & 2.59E-9 \\
        $10$ &$40$ & 14 & 800 &2.02E+2 & 2.79E+2 & 1.13E+1 & 9.27E-1 & 6.76E-9 \\
        \hline
        $20$ &$5$ & 9 &480 &7.27E+1 & 8.86E+1 & 1.97E+0 & 2.28E-1 & 5.81E-10 \\
        $20$ &$10$ & 10 & 480 &7.56E+1 & 9.34E+1 & 1.99E+0 & 2.52E-1 & 1.31E-10 \\
        $20$ &$40$ & 14 & 800 &2.04E+2 & 2.77E+2 & 1.12E+1 & 9.27E-1 & 3.43E-9 \\
        \hline
        $40$ & $5$ & 9 & 800 &1.89E+2 & 2.39E+2 & 7.77E+0 & 6.07E-1 &  3.48E-10\\
        $40$ &$10$ & 10 & 800 &1.92E+2 & 2.53E+2 & 8.19E+0 & 5.88E-1 & 6.12E-10 \\
        $40$ &$20$ & 12 & 800 &2.05E+2 & 2.66E+2 & 1.01E+1  & 8.01E-1 & 1.45E-10 \\
        \hline
      \end{tabular}
    }
  \end{subtable}
  \hfill
  \begin{subtable}[t]{.48\textwidth}
    \centering
    \caption{The cylinder.}
    \label{table_4}
    \resizebox{\textwidth}{!} {
      \begin{tabular}{|c|c|c|c|c|c|c|c|c|}
        \hline
        $k_0$ &$k_1$ & $N_f$ & $N_{pts}$ & $T_{kernel}$ & $T_{matgen}$ & $T_{solve}$ &$T_{add}$ & $E_{error}$  \\
        \hline
        $2$ & $1$ & 9 & 1312 &1.06E+2 & 2.66E+2 & 4.46E+1 & 5.31E-1 &  4.11E-9\\
        $2$ &$5$ & 11 & 1440 &1.56E+2 & 3.80E+2 & 7.13E+1 & 7.26E-1 & 2.56E-8 \\
        $2$ &$10$ & 14 &1696 &2.18E+2 & 7.97E+2 & 1.35E+2  & 1.31E0 & 9.32E-7 \\
        \hline
        $5$ &$2$ & 10 &1440 &1.41E+2 & 3.58E+2 & 6.56E+1 & 6.63E-1 & 1.54E-8 \\
        $5$ &$10$ & 14 & 1696 &2.21E+2 & 7.17E+2 & 1.34E+2 & 1.42E0 & 5.92E-8 \\
        $5$ &$20$ & 17 & 2208 &3.44E+2 & 3.14E+3 & 3.61E+2 & 2.61E0 & 9.74E-7 \\
        \hline
        $10$ &$2$ & 10 & 1696 &1.94E+2 & 5.64E+2 & 9.92E+1 & 9.15E-1 & 1.91E-8 \\
        $10$ &$5$ & 11 & 1696 &2.21E+2 & 5.99E+2 & 1.07E+2 & 1.03E0 & 3.95E-8 \\
        $10$ &$20$ & 17 & 2208 &3.49E+2 & 2.27E+3 & 3.62E+2 & 2.54E0 & 2.24E-6 \\
        \hline
        $20$ &$2$ & 10 &2208 &3.29E+2 & 1.15E+3 & 2.19E+2 & 1.51E0 & 3.54E-8 \\
        $20$ &$5$ & 11 & 2208 &3.41E+2 & 1.23E+3 & 2.38E+2 & 1.61E0 & 1.39E-7 \\
        $20$ &$10$ & 14 & 2208 &3.43E+2 & 1.49E+3 & 2.96E+2 & 2.08E0 & 1.32E-7 \\
        \hline
      \end{tabular}
    }
  \end{subtable}
\end{table}

\vspace{\baselineskip}
\begin{figure}[h]
  \centering
  \begin{subfigure}[b]{.32\linewidth}
    \centering
    \includegraphics[width=.95\linewidth]{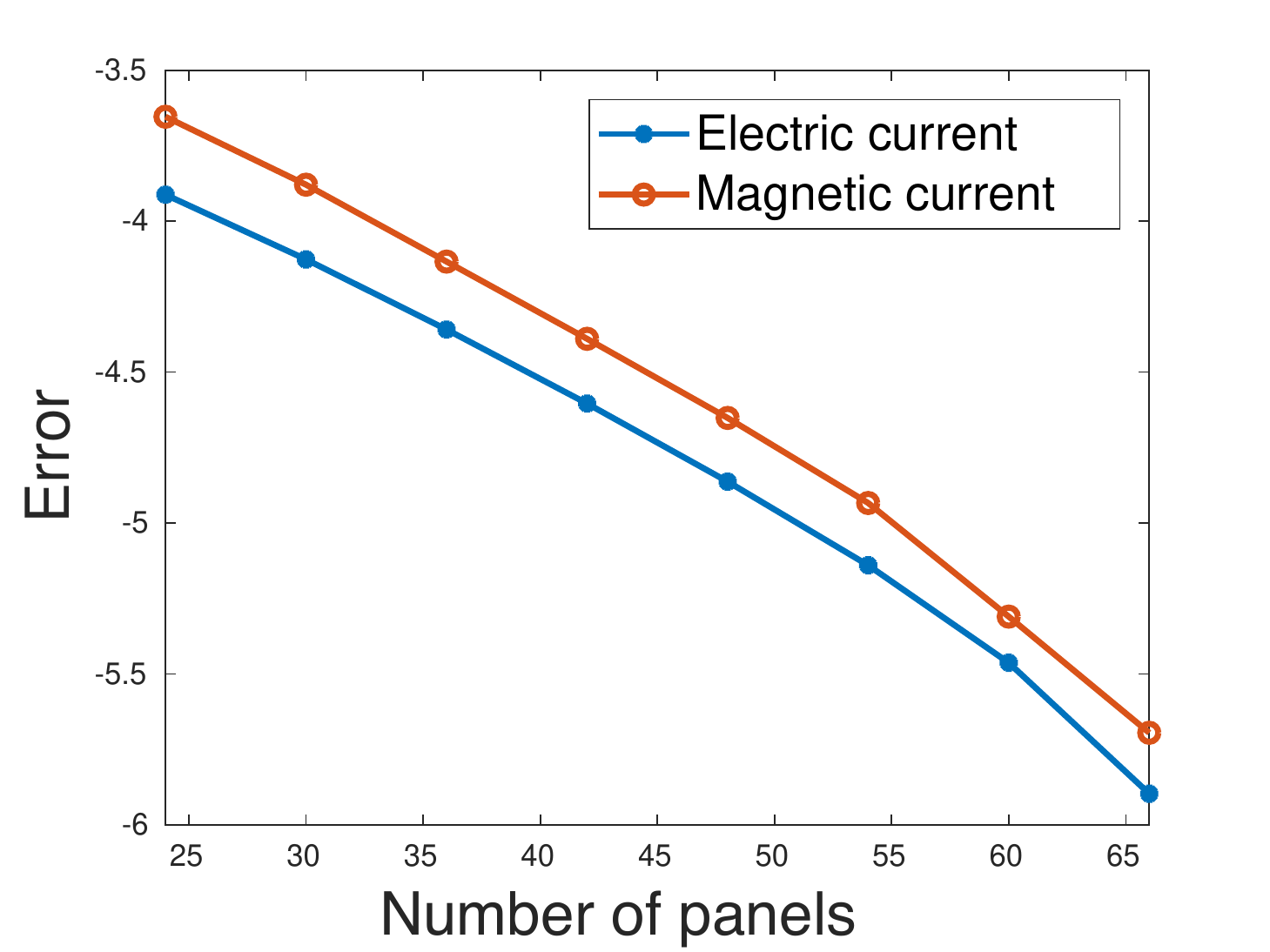}
    \caption{Current error}
    \label{fig:feko1}
  \end{subfigure}
  \hfill
  \begin{subfigure}[b]{.32\linewidth}
    \centering
    \includegraphics[width=.95\linewidth]{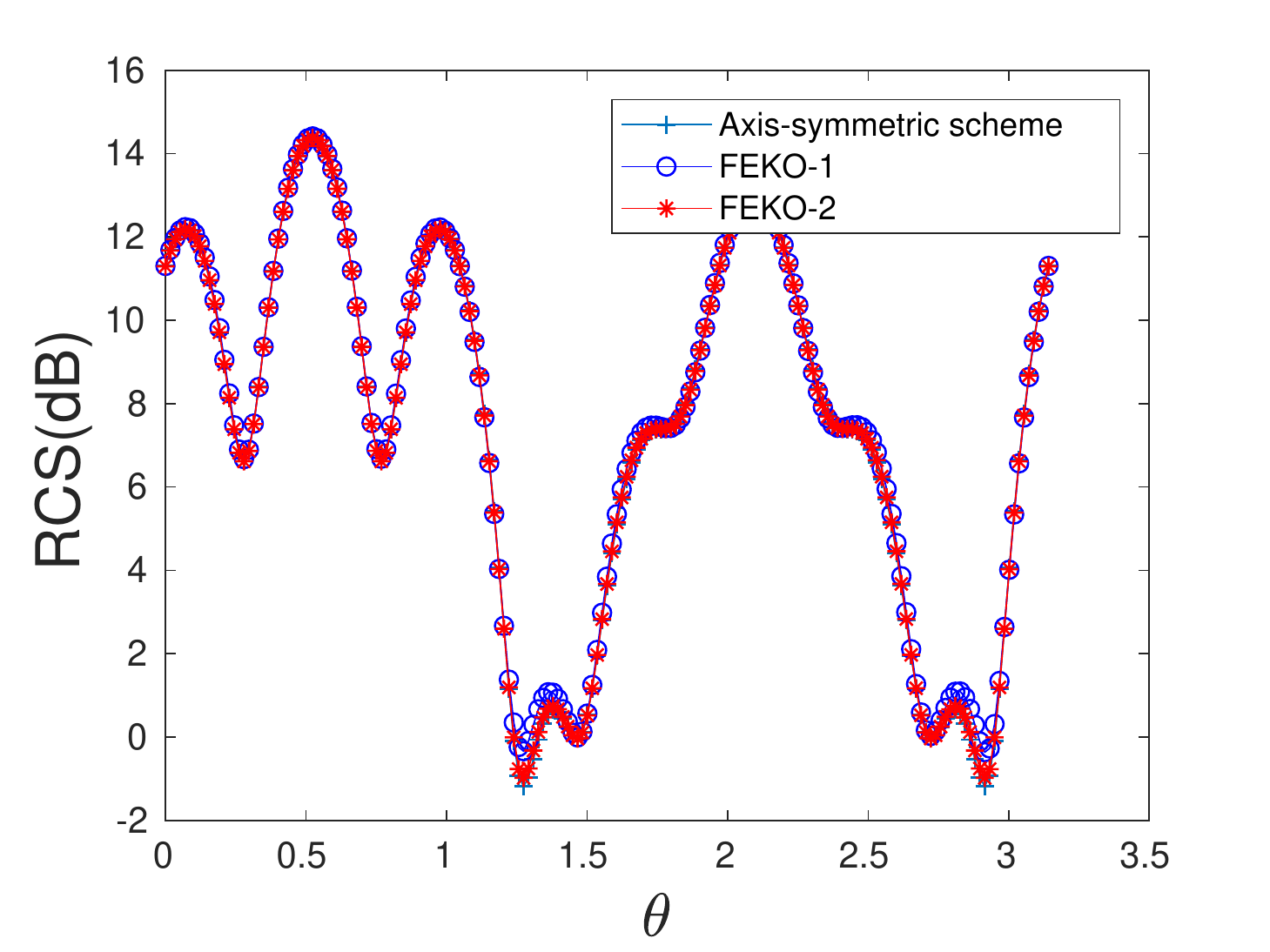}
    \caption{$|\besc_\infty(\cdot,\pi/2)|^2$}
    \label{fig:feko2}
  \end{subfigure}
  \hfill
  \begin{subfigure}[b]{.32\linewidth}
    \centering
    \includegraphics[width=.95\linewidth]{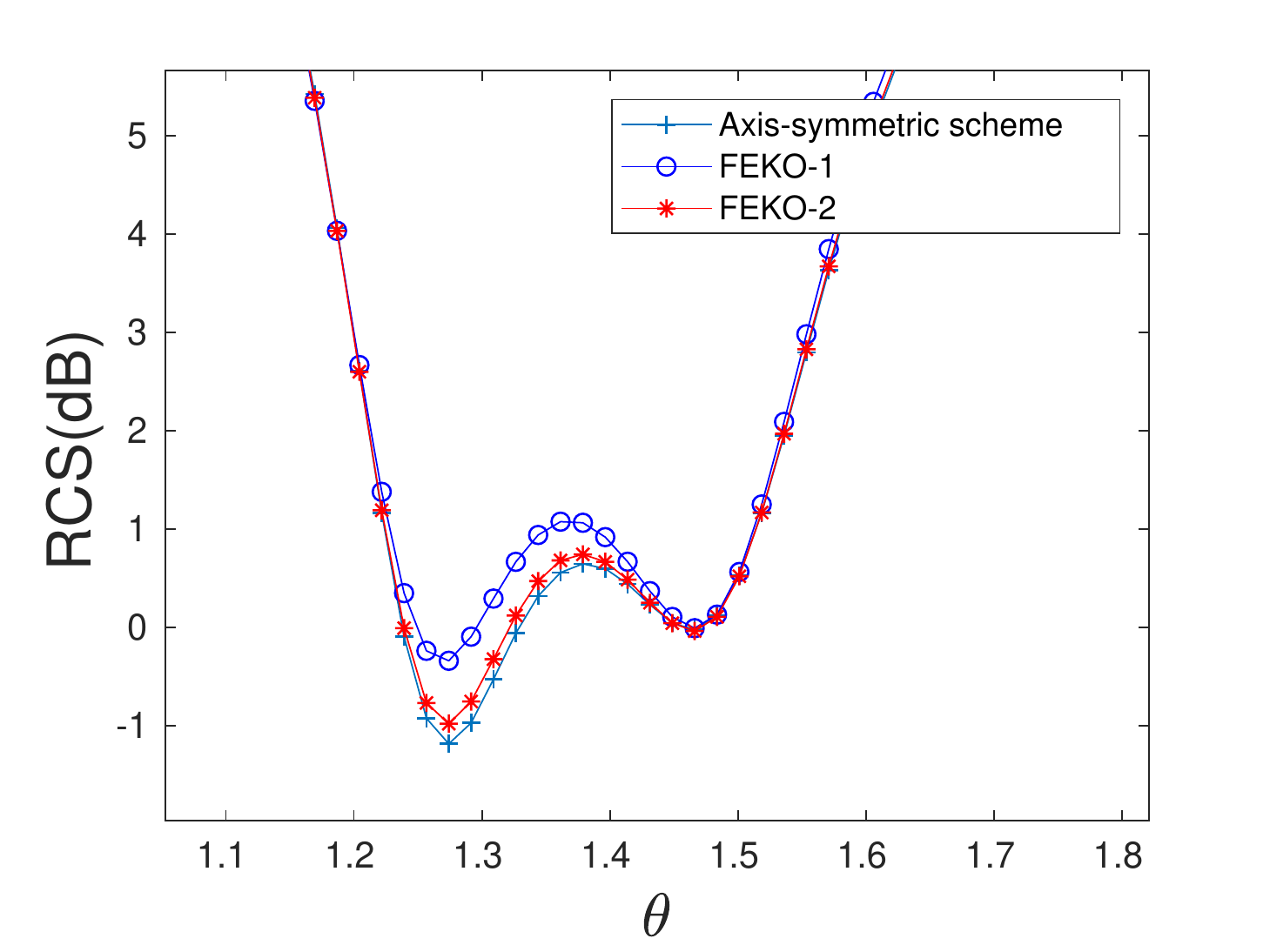}
    \caption{A zoom in of $|\besc_\infty(\cdot,\pi/2)|^2$ }
    \label{fig:feko3}
  \end{subfigure}
  \hfill
  \caption{Validating convergence results for the cylinder.
    (a) Self-consistent convergence of the electric and magnetic
      current from the scattering of a penetrable cylinder when the
      number of discretization panels is increased. The error is given
      in the logarithmic scale. (b) A comparison of the radar cross section
      computed using our scheme and that obtained by~\textit{FEKO}.
      Lines labeled \emph{FEKO-1} were obtained from a coarse mesh
      (edge length 0.2) and lines labeled \emph{FEKO-2} were obtained
      using a fine mesh (edge length 0.1). (c) A blow-up of part of
      results in (b).}
  \label{figure_feko}
\end{figure}

\clearpage}

To demonstrate the accuracy of our solver for scattering from the
cylinder, Figure~\ref{fig:feko1} shows the convergence of the electric
and magnetic currents when the number of panels is increased by 6,
starting with 24 panels.  In this case, we have set $k_0=\pi$,
$k_1=2\pi$, and the incident plane wave has parameters
$\theta_1 = \pi/3$, $\phi_1 = 2\pi/3$, $\theta_2=\pi/3$ and
$\phi_2 = \pi/2$ (as in~\eqref{incidentdirection}).  The error is
measured in the $L^2$ norm and is obtained by comparing the results
with those obtained using a 72 panel discretization.  To further
validate the numerical results of our solver we compare our far-field
patterns with those computed using the commercial software
package~\textit{FEKO}, which also uses a boundary integral method
(i.e. method of moments).  Figure~\ref{fig:feko2} and~\ref{fig:feko3}
show a comparison of the radar cross section (RCS) of the cylinder
when $\phi=\pi/2$ and $\theta\in[0,2\pi]$ (the RCS is the square of
the modulus of the far-field pattern). Two different meshes are used
to compute the RCS in~\textit{FEKO}. One is on a coarse mesh with edge
length 0.2 and another one is on a refined mesh with edge length
0.1. It can been seen from Figure~\ref{figure_feko} that the results
from~\textit{FEKO} converge to the result obtained by our scheme. It
should be also noticed that our solvers used to obtain the results in
Figures~\ref{fig:feko2} and~\ref{fig:feko3} only requires 18 seconds
for 37 modes and 24 panels (i.e. 384 discretization points on the
generating curve); since~\textit{FEKO} does not take advantage of
axisymmetric geometry, and does not discretize to high-order, more
than 200 seconds were required to obtain the solution on the refined
mesh (edge length 0.1).

\begin{figure}[t]
  \centering
  \begin{subfigure}[b]{.38\linewidth}
    \centering
    \includegraphics[width=.95\linewidth]{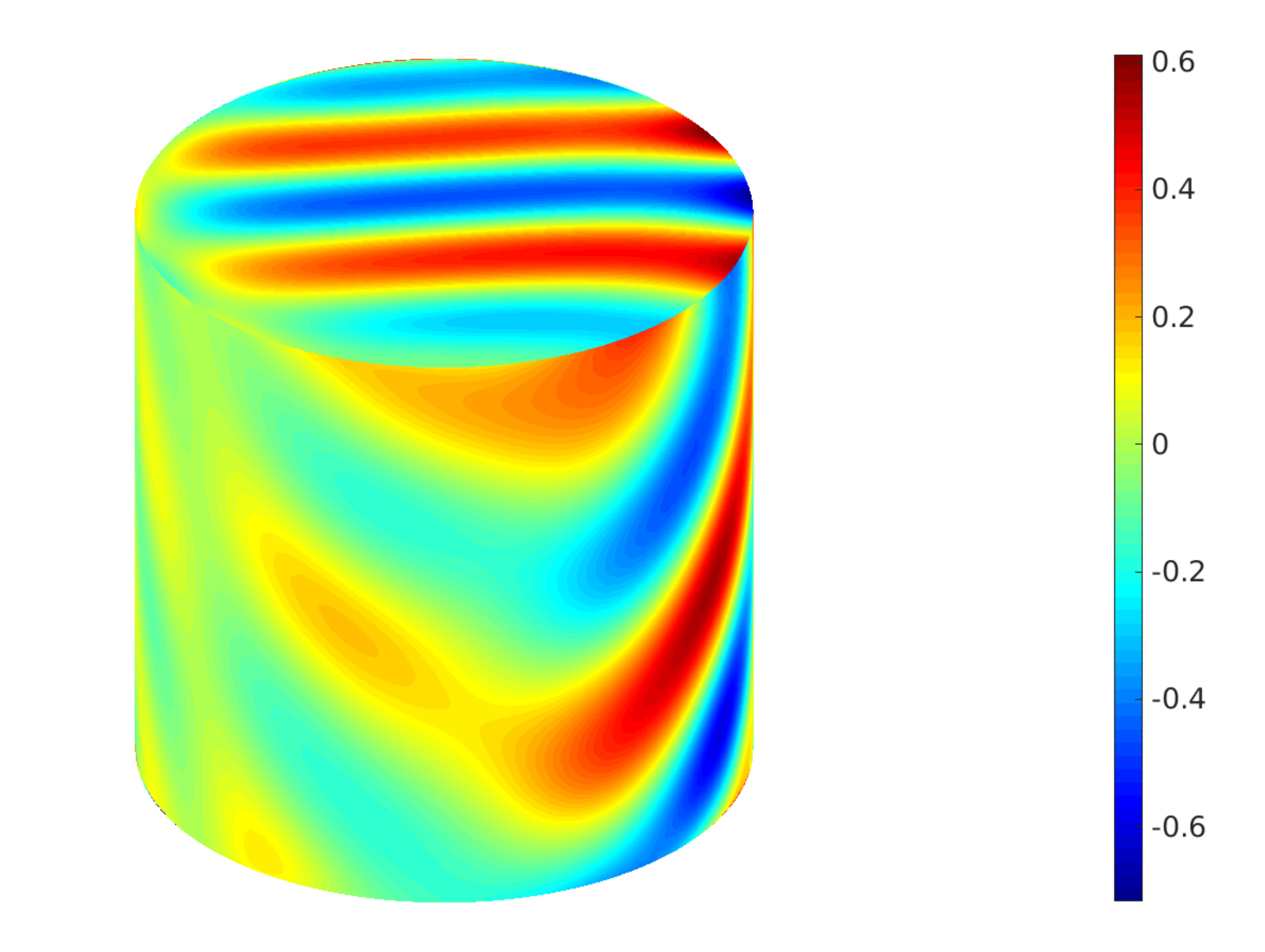}
    \caption{$\Re J_x$.}
  \end{subfigure}
  \hfill
  \begin{subfigure}[b]{.3\linewidth}
    \centering
    \includegraphics[width=.95\linewidth]{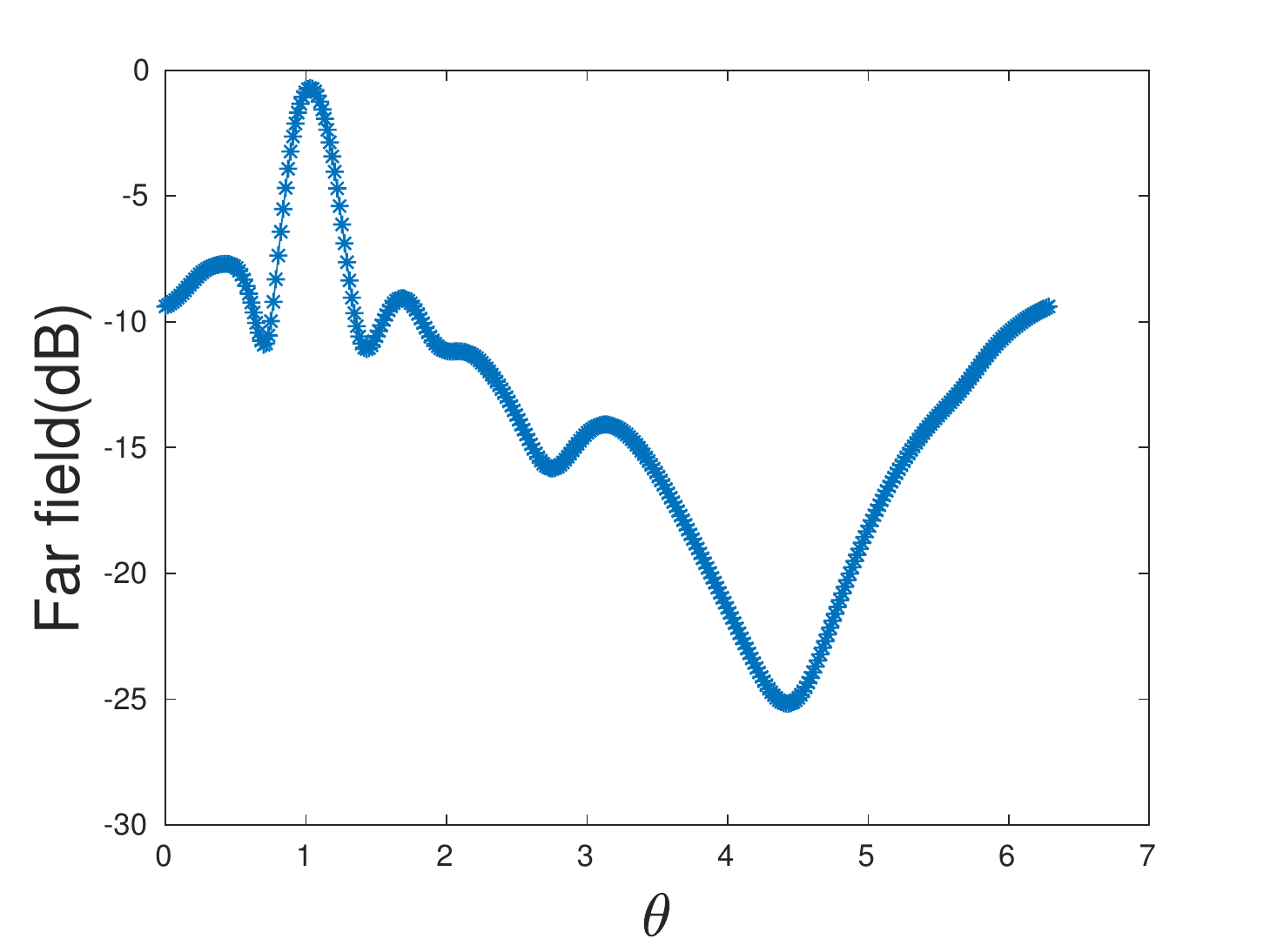}
    \caption{$|\besc_\infty(\cdot,\pi/2)|$.}
  \end{subfigure}
  \hfill
  \begin{subfigure}[b]{.3\linewidth}
    \centering
    \includegraphics[width=.95\linewidth]{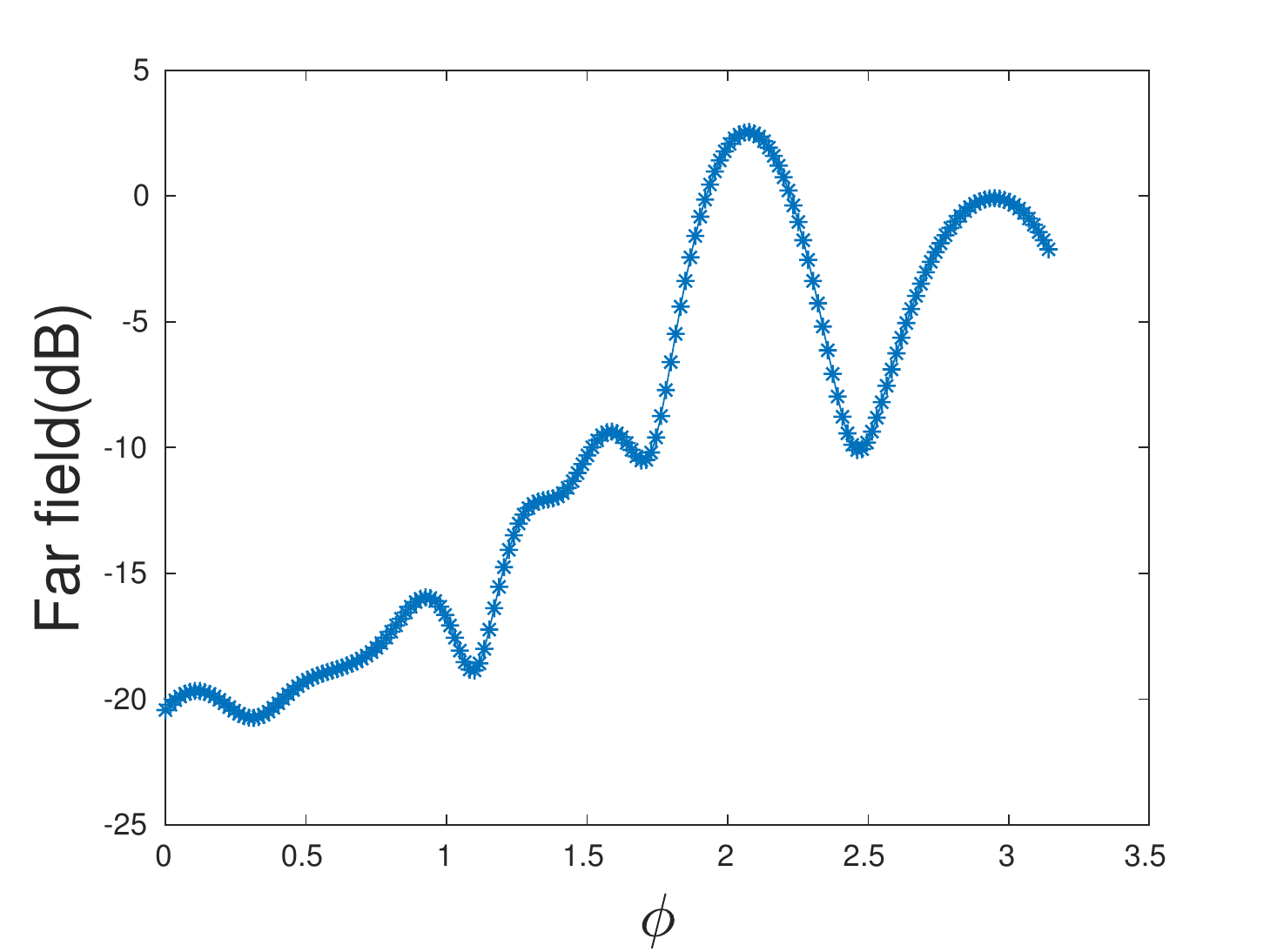}
    \caption{$|\besc_\infty(0,\cdot)|$.}
  \end{subfigure}
  \caption{A penetrable cylinder~$\Omega$  with interior 
    $k_1=5.0$ and background~$k_0=10.0$.}
  \label{figure_cylinder}
\end{figure}

\section{Conclusion}
\label{sec_con}
In this paper, we provided a derivation of M\"uller's integral
equation, and its indirect formulation, for electromagnetic scattering
from piecewise constant penetrable media.  The resulting integral
equations are second-kind when the boundary of the inclusion is
smooth, and remain relatively well-conditioned on $L^2$ when the boundary has a
modest number of edges or geometric singularities. In either case, the
integral equation representation admits a unique solution for all
ranges of interior and exterior material properties (as defined in the
introduction). Our numerical solver strongly takes advantage of the
axisymmetric geometry by using a Fourier-based separation of variables
in the azimuthal angle to obtain a sequence of decoupled integral
equations along a cross section of the geometry. Using FFTs, discrete
convolution, kernel splitting, and novel recurrence relations we are
able to efficiently evaluate the modal Green's functions and their
derivatives.  High-order accurate convergence is observed when
discretizing the integral equations using generalized Gaussian
quadratures and an adaptive Nystr\"om-like method.  Numerical examples
show that the algorithm can efficiently and accurately solve the
scattering problem from various axisymmetric objects, even in the
presence of geometric edge and point singularities.

Throughout all the numerical examples, we assume the
frequency~$\omega$ is in the resonant regime. In particular, it is not
close to zero. In the low-frequency regime, when $\omega$ approaches
zero, the operator $\cK^k/(i\omega)$ becomes numerically unstable,
which leads to the well-known low-frequency breakdown issue
\cite{EG10,EG13,epstein_2018}. In that case, M\"uller's formulation
(as well as its indirect form) has to be modified to overcome this
instability. Some helpful discussions regarding this can be found
in~\cite{Helsing2016}. However, when~$\omega$ is not close to zero,
the integral equation formulations of this paper extend directly and
remain well-conditioned in more complicated geometries. Extensions to
arbitrary geometries in three dimensions are underway, but will of
course require a completely new set of tools (fast algorithms,
quadrature, etc.). We will report on this in the future.

\appendix
\section{Proof of the recursion formulas}
\label{sec:recurrence}

For simplicity, we consider
\begin{equation}
\mathcal{P}_m = \int_0^{2\pi}\frac{\cos m \phi}{\sqrt{\chi-\cos\phi}}
\, d\phi 
\end{equation}
instead of~$\mathcal{Q}_{m-1/2}(\chi)$, since $\mathcal{Q}_{m-1/2}(\chi)=\mathcal{P}_m/\sqrt{8}$, and
\begin{align}
  \mathcal{S}_m(\chi) &= \int_0^{2\pi}\frac{\cos m
    \phi}{(\chi-\cos\phi)^{3/2}} \, d\phi, \\
  \mathcal{T}_m(\chi) &=  \int_0^{2\pi } \frac{\cos m \phi
  }{\chi- \cos\phi} \, d\phi.
\end{align}
We first derive the recursion formula for~$\mathcal{P}_{m}$.
\begin{lemma}\label{Alemma1} If $m\ge 1$, then
  \begin{equation}
    \mathcal{P}_{m+1}(\chi) = \frac{4m}{2m+1}\chi
    \mathcal{P}_{m}(\chi)-\frac{2m-1}{2m+1}\mathcal{P}_{m-1}(\chi).
  \end{equation}
\end{lemma}
\begin{proof} Starting from  $\mathcal{P}_{m+1}$
  \begin{equation}
    \begin{aligned}
      \mathcal{P}_{m+1}(\chi) &= \int_{0}^{2\pi}
      \frac{\cos{(m+1)\phi}}{\sqrt{\chi-\cos 
          \phi}} \, d\phi \\
      &= \int_{0}^{2\pi}\frac{\cos{(m-1)\phi}}{\sqrt{\chi-\cos \phi}}
      \, d\phi 
      -2\int_{0}^{2\pi}\frac{\sin\phi \, \sin{m\phi}}{\sqrt{\chi-\cos
          \phi}} \, d\phi  \\
      &=\int_{0}^{2\pi}\frac{\cos{(m-1)\phi}}{\sqrt{\chi-\cos \phi}}
      \, d\phi 
      - 4\int_{0}^{2\pi}{\sin{m\phi}} \, d\lp \sqrt{\chi-\cos \phi} \rp \\
      &=\int_{0}^{2\pi}\frac{\cos{(m-1)\phi}}{\sqrt{\chi-\cos \phi}}
      \, d\phi 
      + 4m\int_{0}^{2\pi}{\cos{m\phi}} \, \sqrt{\chi-\cos \phi} \,
      d\phi \\
      &=\mathcal{P}_{m-1}(\chi) +4m \, \mathcal{O}_m(\chi),
    \end{aligned}
  \end{equation}
  where
  \begin{equation}
    \mathcal{O}_m(\chi)=\int_{0}^{2\pi} {\cos{m\phi}} \,
    \sqrt{\chi-\cos\phi} \, d\phi,
  \end{equation}
  and also by noting that~$\mathcal P_m$ can be expanded as:
  \begin{equation}
    \begin{aligned}
      \mathcal{P}_m &=\int_{0}^{2\pi}\frac{\cos{m\phi}}{\sqrt{\chi-\cos
          \phi}}d\phi \\
      &= \frac{1}{\chi}\int_{0}^{2\pi}\frac{(\chi-\cos \phi) \, 
        \cos{m\phi}}
      {\sqrt{\chi-\cos \phi}} \, d\phi+\frac{1}{\chi}
      \int_{0}^{2\pi}\frac{\cos\phi \, \cos{m\phi}}
      {\sqrt{\chi-\cos \phi}} \, d\phi \\
      &= \frac{1}{\chi}\mathcal{O}_m(\chi) + \frac{1}{\chi} \lp
      \int_{0}^{2\pi}
      \frac{\cos{(m-1)\phi}}{\sqrt{\chi-\cos \phi}} \, d\phi -
      \int_{0}^{2\pi}\frac{\sin \phi \, \sin{m\phi}}
      {\sqrt{\chi-\cos\phi}} \, d\phi \rp \\
      &= \frac{1}{\chi}\mathcal{O}_m(\chi)
      +\frac{1}{\chi} \lp \mathcal{P}_{m-1}(\chi) +2m \,
      \mathcal{O}_m(\chi)
      \rp.
    \end{aligned}
  \end{equation}
  The recursion formula follows by combining the two expressions
  together. $\square$
\end{proof}

\begin{lemma} If $m\ge 1$, then
  \begin{equation}
    \mathcal{S}_m(\chi)  = \frac{(1-2m)\chi
      \, \mathcal{P}_{m}(\chi) - {(1-2m)} \, \mathcal{P}_{m-1}(\chi) }
    {\chi^2-1}.
  \end{equation}
\end{lemma}
\begin{proof}
  Performing a similar calculation as in the previous lemma, we have:
  \begin{equation}
    \label{equ_Pn}
    \begin{aligned}
      \mathcal{S}_m(\chi) &=
      \frac{1}{\chi}\int_{0}^{2\pi}\frac{(\chi-\cos
        \phi)\cos{m\phi}}{(\chi-\cos \phi)^{3/2}} \, d\phi
      +\frac{1}{\chi}\int_{0}^{2\pi}\frac{\cos
        \phi \, \cos{m\phi}}{(\chi-\cos \phi)^{3/2}} \, d\phi \\
      &= \frac{1}{\chi} \mathcal{P}_m(\chi) +
      \frac{1}{\chi}\mathcal{U}_m(\chi),
    \end{aligned}
  \end{equation}
  where
  \begin{equation}
    \mathcal{U}_m(\chi) =
    \int_{0}^{2\pi} \frac{\cos \phi \, \cos{m\phi}} {(\chi-\cos
      \phi)^{3/2}}
    \, d\phi.
  \end{equation}
  It then holds that
  \begin{equation}\label{equ_Tn}
  \resizebox{.9 \textwidth}{!} {
    $
    \begin{aligned}
      \mathcal{U}_m(\chi)
      &=\frac{1}{\chi}\int_{0}^{2\pi}\frac{(\chi\cos \phi-\cos^2 \phi)
        \, \cos{m\phi}}{(\chi-\cos \phi)^{3/2}} \, d\phi +
      \frac{1}{\chi}\int_{0}^{2\pi}\frac{\cos^2 \,
        \phi\cos{m\phi}}{(\chi-\cos \phi)^{3/2}} \, d\phi  \\
      &= \frac{1}{\chi}\int_{0}^{2\pi}\frac{\cos\phi \, \cos{m\phi}}
      {(\chi-\cos \phi)^{1/2}} \, d\phi +
      \frac{1}{\chi}\int_{0}^{2\pi}
      \frac{\cos^2 \phi \, \cos{m\phi}} {(\chi-\cos \phi)^{3/2}} \, d\phi \\
      &= \frac{1}{2}\frac{1}{\chi} \lp
      \mathcal{P}_{m+1}(\chi)+\mathcal{P}_{m-1}(\chi) \rp
      +  \frac{1}{\chi} \lp \int_{0}^{2\pi}\frac{\cos{m\phi}}
      {(\chi-\cos \phi)^{3/2}} \, d\phi- \int_{0}^{2\pi}\frac{\sin^2
        \phi \, \cos{m\phi}} {(\chi-\cos \phi)^{3/2}} \, d\phi \rp  \\
      &= \frac{1}{2}\frac{1}{\chi} \lp \mathcal{P}_{m+1}(\chi) +
      \mathcal{P}_{m-1}(\chi) \rp
      + \frac{1}{\chi}\mathcal{S}_m(\chi)
      + \frac{2}{\chi}\int_{0}^{2\pi}\sin \phi\cos{m\phi} \, d\lp
      \frac{1}{\sqrt{\chi-\cos \phi}} \rp \\
      &=
      \frac{1}{2}\frac{1}{\chi}  \lp \mathcal{P}_{m+1}(\chi)  +
      \mathcal{P}_{m-1}(\chi) \rp +
      \frac{1}{\chi}\mathcal{S}_m(\chi)  -
      \frac{2}{\chi}\int_{0}^{2\pi}\frac{\cos
        \phi \, \cos{m\phi}}{\sqrt{\chi-\cos \phi}} \, d\phi +
      \frac{2m}{\chi}\int_{0}^{2\pi}\frac{\sin \phi \,
        \sin{m\phi}}{\sqrt{\chi-\cos \phi}} \, d\phi  \\
      &=
      \frac{1}{2}\frac{1}{\chi}(\mathcal{P}_{m+1}(\chi) + \mathcal{P}_{m-1}(\chi)) +
      \frac{1}{\chi}\mathcal{S}_m(\chi) -
      \frac{1}{\chi}(\mathcal{P}_{m+1}(\chi) + \mathcal{P}_{m-1}(\chi)
      )-\frac{m}{\chi}(\mathcal{P}_{m+1}(\chi) - \mathcal{P}_{m-1}(\chi)).
    \end{aligned}
    $
    }
  \end{equation}
  Combining equations~\eqref{equ_Pn} and~\eqref{equ_Tn}, and using
  Lemma~\ref{Alemma1}, we obtain the final recurrence relation
  \begin{equation}
    \begin{aligned}
      \mathcal{S}_m(\chi) &= \frac{1}{\chi^2-1} \lp  \chi \, 
      \mathcal{P}_m(\chi) + \frac{2m-1}{2} \mathcal{P}_{m-1}(\chi)
      -\frac{2m+1}{2} \mathcal{P}_{m+1}(\chi) \rp \\
      &=\frac{1}{\chi^2-1} \lp (1-2m) \chi \, 
      \mathcal{P}_m(\chi) - {(1-2m)} \mathcal{P}_{m-1}\chi \rp).
      \end{aligned}
    \end{equation}
    This ends the proof.  $\square$
  \end{proof}

Finally, we have the following lemma and its proof.
      
\begin{lemma} If~$m\ge 1$, then
  $\mathcal{T}_{m+1}(\chi) = 2\chi \mathcal{T}_{m}(\chi)-\mathcal{T}_{m-1}(\chi)$.
\end{lemma}

\begin{proof}
  \begin{equation}
    \begin{aligned}
    \mathcal{T}_m(\chi) &= \frac{1}{\chi} \int_0^{2\pi } \frac{\chi \,
      \cos m \phi }{\chi- \cos\phi} \, d\phi \\
    &= \frac{1}{\chi}\int_0^{2\pi } \frac{\cos\phi \, \cos m \phi }
    {\chi- \cos\phi} \, d\phi \\
    &= \frac{1}{2\chi} \lp \mathcal{T}_{m+1}(\chi) + \mathcal{T}_{m-1}
    (\chi) \rp.
  \end{aligned}
\end{equation}
\end{proof}
        
%\printbibliography

\bibliographystyle{abbrv}
\bibliography{reference}

\begin{thebibliography}{10}

\bibitem{alpert1999}
B.~Alpert.
\newblock Hybrid {G}auss-trapezoidal quadrature rules.
\newblock {\em {SIAM} {J}. {S}ci. {C}omput.}, 20(5):1551--1584, 1999.

\bibitem{andreasen1965}
M.~Andreasen.
\newblock Scattering from bodies of revolution.
\newblock {\em IEEE Trans. Antennas Propag.}, 13(2):303--310, 1965.

\bibitem{banerjee1981}
P.~K. Banerjee and R.~Butterfield.
\newblock {\em Boundary element methods in engineering science}.
\newblock McGraw-Hill, London, UK, 1981.

\bibitem{Brem2012}
J.~Bremer.
\newblock On the {N}ystr\"om discretization of integral equations on planar
  curves with corners.
\newblock {\em Appl. Comput. Harm. Anal.}, 32:45--64, 2012.

\bibitem{Bremer2012}
J.~Bremer and Z.~Gimbutas.
\newblock {A Nystr{\"{o}}m method for weakly singular integral operators on
  surfaces}.
\newblock {\em J. Comput. Phys.}, 231(14):4885--4903, 2012.

\bibitem{BGR2010}
J.~Bremer, Z.~Gimbutas, and V.~Rokhlin.
\newblock A nonlinear optimization procedure for generalized {G}aussian
  quadratures.
\newblock {\em SIAM J. Sci. Comput.}, 32(4):1761--1788, 2010.

\bibitem{briggs_1995}
W.~L. Briggs and V.~E. Henson.
\newblock {\em {The DFT: An Owner's Manual for the Discrete Fourier
  Transform}}.
\newblock SIAM, Philadelphia, PA, 1995.

\bibitem{bruno2003fast}
O.~P. Bruno.
\newblock Fast, high-order, high-frequency integral methods for computational
  acoustics and electromagnetics.
\newblock In {\em Topics in computational wave propagation}, pages 43--82.
  Springer, 2003.

\bibitem{bruno2009krylov}
O.~P. Bruno, T.~Elling, R.~Paffenroth, and C.~Turc.
\newblock Electromagnetic integral equations requiring small numbers of
  {K}rylov-subspace iterations.
\newblock {\em J. Comput. Phys.}, 228:6169--6183, 2009.

\bibitem{bruno1998jasa}
O.~P. Bruno and F.~Reitich.
\newblock Boundary-variation solutions for bounded-obstacle scattering problems
  in three dimensions.
\newblock {\em J. Acoust. Soc. Am.}, 104(5):2579--2583, 1998.

\bibitem{Bulygin2013Full}
V.~S. Bulygin, T.~M. Benson, Y.~V. Gandel, and A.~I. Nosich.
\newblock {Full-Wave Analysis and Optimization of a TARA-Like Shield-Assisted
  Paraboloidal Reflector Antenna Using a Nystrom-Type Method}.
\newblock {\em IEEE Trans. Antennas Propag.}, 61(10):4981--4989, 2013.

\bibitem{chew2009book}
W.~C. Chew, M.~S. Tong, and B.~Hu.
\newblock {\em {Integral Equation Methods for Electromagnetic and Elastic
  Waves}}.
\newblock Morgan \& Claypool, Williston, VT, USA, 2009.

\bibitem{cohl_1999}
H.~S. Cohl and J.~E. Tohline.
\newblock A compact cylindrical {G}reen's function expansion for the solution
  of potential problems.
\newblock {\em Astrophys. J.}, 527(1):86--101, 1999.

\bibitem{Cot2}
D.~Colton and R.~Kress.
\newblock {\em Integral Equation Method in Scattering Theory}.
\newblock Wiley-Interscience, New York, 1983.

\bibitem{conway_cohl}
J.~T. Conway and H.~S. Cohl.
\newblock Exact {F}ourier expansion in cylindrical coordinates for the
  three-dimensional {H}elmholtz {G}reen function.
\newblock {\em Z. {A}ngew. {M}ath. {P}hys.}, 61:425--442, 2010.

\bibitem{costabel1988}
M.~Costabel.
\newblock {B}oundary integral operators on lipschitz domains: {E}lementary
  results.
\newblock {\em SIAM J. Math. Anal.}, 19:613--626, 1988.

\bibitem{dunn2006}
E.~A. Dunn, J.-K. Byun, E.~D. Branch, and J.-M. Jin.
\newblock {Numerical Simulation of BOR Scattering and Radiation Using a Higher
  Order FEM}.
\newblock {\em IEEE Trans. Antennas Propag.}, 54(3):945--952, 2006.

\bibitem{EG10}
C.~L. Epstein and L.~Greengard.
\newblock {Debye Sources and the Numerical Solution of the Time Harmonic
  Maxwell Equations}.
\newblock {\em Comm. Pure Appl. Math.}, 63(4):413--463, 2010.

\bibitem{EG13}
C.~L. Epstein, L.~Greengard, and M.~O'Neil.
\newblock {Debye Sources and the Numerical Solution of the Time Harmonic
  Maxwell Equations II}.
\newblock {\em Comm. Pure Appl. Math.}, 66(5):753--789, 2013.

\bibitem{epstein_2018}
C.~L. Epstein, L.~Greengard, and M.~O'Neil.
\newblock A high-order wideband direct solver for electromagnetic scattering
  from bodies of revolution.
\newblock {\em J. Comput. Phys.}, 2019.
\newblock To appear.

\bibitem{fang2007jcp}
Q.~Fang, D.~P. Nicholls, and J.~Shen.
\newblock A stable, high‐order method for three‐dimensional,
  bounded‐obstacle, acoustic scattering.
\newblock {\em J. Comput. Phys.}, 224(2):1145--1169, 2007.

\bibitem{fleming2004}
J.~L. Fleming, A.~W. Wood, and W.~D.~W. Jr.
\newblock {Locally corrected Nystr\"om method for EM scattering by bodies of
  revolution}.
\newblock {\em J. Comput. Phys.}, 196:41--52, 2004.

\bibitem{gedney_1990}
S.~D. Gedney and R.~Mittra.
\newblock The use of the {FFT} for the efficient solution of the problem of
  electromagnetic scattering by a body of revolution.
\newblock {\em IEEE Trans. Antennas Propag.}, 38:313--322, 1990.

\bibitem{Geng1999Wide}
N.~Geng and L.~Carin.
\newblock {Wide-band electromagnetic scattering from a dielectric BOR buried in
  a layered lossy dispersive medium}.
\newblock {\em IEEE Trans. Antennas Propag.}, 47(4):610--619, 1999.

\bibitem{gil_2007}
A.~Gil, J.~Segura, and N.~M. Temme.
\newblock {\em {Numerical Methods for Special Functions}}.
\newblock SIAM, Philadelphia, PA, 2007.

\bibitem{GYM2012}
A.~Gillman, P.~M. Young, and P.-G. Martinsson.
\newblock A direct solver with {$O(N)$} complexity for integral equations on
  one-dimensional domains.
\newblock {\em Front. Math. China}, 7(2):217--247, 2012.

\bibitem{GG2013}
Z.~Gimbutas and L.~Greengard.
\newblock Fast multi-particle scattering: A hybrid solver for the {Maxwell}
  equations in microstructured materials.
\newblock {\em J. Comput. Phys.}, 232:22--32, 2013.

\bibitem{gustafsson2010accurate}
M.~Gustafsson.
\newblock Accurate and efficient evaluation of modal {G}reen's functions.
\newblock {\em J. Electromagnet. Wave.}, 24(10):1291--1301, 2010.

\bibitem{hao_2014}
S.~Hao, A.~H. Barnett, P.~G. Martinsson, and P.~Young.
\newblock High-order accurate {N}ystr\"om discretization of integral equations
  with weakly singular kernels on smooth curves in the plane.
\newblock {\em Adv. Comput. Math.}, 40:245--272, 2014.

\bibitem{Hao2015}
S.~Hao, P.-G. Martinsson, and P.~Young.
\newblock {An efficient and highly accurate solver for multi-body acoustic
  scattering problems involving rotationally symmetric scatterers}.
\newblock {\em Comput. Math. Appl.}, 69:304--318, 2015.

\bibitem{harrington1968}
R.~F. Harrington.
\newblock {\em {Field Computation by Moment Methods}}.
\newblock Macmillan, Co., New York, NY, 1968.

\bibitem{helsing2015}
J.~Helsing and A.~Holst.
\newblock Variants of an explicit kernel-split panel-based {N}ystr\"om
  discretization scheme for {H}elmholtz boundary value problems.
\newblock {\em Adv. Comput. Math.}, 41:691--708, 2015.

\bibitem{helsing2014}
J.~Helsing and A.~Karlsson.
\newblock An explicit kernel-split panel-based {N}ystr\"om scheme for integral
  equations on axially symmetric surfaces.
\newblock {\em J. Comput. Phys.}, 272:686--703, 2014.

\bibitem{helsing2015ieee}
J.~Helsing and A.~Karlsson.
\newblock {Determination of Normalized Magnetic Eigenfields in Microwave
  Cavities}.
\newblock {\em IEEE Trans. Microw. Theory Tech.}, 63(5):1457--1467, 2015.

\bibitem{HK15}
J.~Helsing and A.~Karlsson.
\newblock Determination of normalized electric eigenfields in microwave
  cavities with sharp edges.
\newblock {\em J. Comput. Phys.}, 304:465 -- 486, 2016.

\bibitem{Helsing2016}
J.~Helsing and A.~Karlsson.
\newblock {Resonances in axially symmetric dielectric objects}.
\newblock {\em IEEE Trans. Microw. Theory Tech.}, 65(7):2214--2227, 2017.

\bibitem{HG2012}
K.~L. Ho and L.~Greengard.
\newblock A fast direct solver for structured linear systems by recursive
  skeletonization.
\newblock {\em SIAM J. Sci. Comput.}, 34(5):2507--2532, 2012.

\bibitem{imbert2018}
L.-M. Imbert-Gerard, F.~Vico, L.~Greengard, and M.~Ferrando.
\newblock Integral equation methods for electrostatics, acoustics and
  electromagnetics in smoothly varying, anisotropic media.
\newblock 2018.
\newblock arXiv:1805.04791 [math.NA].

\bibitem{jackson1999}
J.~D. Jackson.
\newblock {\em {Classical Electrodynamics}}.
\newblock Wiley, Hoboken, NJ, 3rd edition, 1999.

\bibitem{KH2015}
A.~Kirsch and F.~Hettlich.
\newblock {\em The Mathematical Theory of Time-Harmonic Maxwell's Equations}.
\newblock Springer Verlag, Cham, Switzerland, 2015.

\bibitem{Kress2010}
R.~Kress.
\newblock {\em Linear Integral Equations}.
\newblock Springer, New York, 1999.

\bibitem{JL2014}
J.~Lai, S.~Ambikasaran, and L.~F. Greengard.
\newblock {A Fast Direct Solver for High Frequency Scattering from a Large
  Cavity in Two Dimensions}.
\newblock {\em SIAM J. Sci. Comput.}, 36(6):B887--B903, 2014.

\bibitem{LAI20171}
J.~Lai, L.~Greengard, and M.~O'Neil.
\newblock Robust integral formulations for electromagnetic scattering from
  three-dimensional cavities.
\newblock {\em Journal of Computational Physics}, 345:1 -- 16, 2017.

\bibitem{Liu2015}
Y.~Liu and A.~H. Barnett.
\newblock Efficient numerical solution of acoustic scattering from
  doubly-periodic arrays of axisymmetric objects.
\newblock {\em J. Comput. Phys.}, 324:226 -- 245, 2016.

\bibitem{mautz1969bor}
J.~R. Mautz and R.~F. Harrington.
\newblock Radiation and scattering from bodies of revolution.
\newblock {\em Appl. Sci. Res.}, 20(1):405--435, 1969.

\bibitem{mautz1977}
J.~R. Mautz and R.~F. Harrington.
\newblock Electromagnetic scattering from a homogeneous body of revolution.
\newblock Technical report, Department of Electrical and Computer Engineering,
  Syracuse, NY, November 1977.
\newblock TS-77-20.

\bibitem{MH1978}
J.~R. Mautz and R.~F. Harrington.
\newblock {H}-field, {E}-field and combined-field solutions for conducting
  bodies of revolution.
\newblock {\em Arch. Elec. Ubertragung.}, 32:159--164, 1978.

\bibitem{medgyesi1984}
L.~N. Medgyesi-Mitschang and J.~M. Putnam.
\newblock {Electromagnetic Scattering from Axially Inhomogeneous Bodies of
  Revolution}.
\newblock {\em IEEE Trans. Antennas Propag.}, 32(8):797--806, 1984.

\bibitem{morgan1979}
M.~A. Morgan and K.~K. Mei.
\newblock {Finite-Element Computation of Scattering by Inhomogeneous Penetrable
  Bodies of Revolution}.
\newblock {\em IEEE Trans. Antennas Propag.}, 27(2):202--214, 1979.

\bibitem{MULLER}
C.~M\"{u}ller.
\newblock {\em Foundations of the Mathematical Theory of Electromagnetic
  Waves}.
\newblock Springer Verlag, 1969.

\bibitem{Ned01}
J.-C. Nedelec.
\newblock {\em Acoustic and Electromagnetic Equations}.
\newblock Springer-Verlag New York, 2001.

\bibitem{nicholls2006sisc}
D.~P. Nicholls and J.~Shen.
\newblock {A Stable High‐Order Method for Two‐Dimensional
  Bounded‐Obstacle Scattering}.
\newblock {\em SIAM J. Sci. Comput.}, 28(4):1398--1419, 2006.

\bibitem{Hand2010}
F.~W.~J. Olver, D.~W. Lozier, R.~F. Boisvert, and C.~W. Clark.
\newblock {\em NIST Handbook of Mathematical Functions}.
\newblock Cambridge University Press, New York, 2010.

\bibitem{oneil2018}
M.~O'Neil and A.~J. Cerfon.
\newblock An integral equation-based numerical solver for {T}aylor states in
  toroidal geometries.
\newblock {\em J. Comput. Phys.}, 359:263--282, 2018.

\bibitem{rok2}
V.~Rokhlin.
\newblock Solution of acoustic scattering problems by means of second kind
  integral equations.
\newblock {\em Wave Motion}, 5:257--272, 1983.

\bibitem{GMRES1986}
Y.~Saad and M.~Schultz.
\newblock {GMRES}: {A} generalized minimal residual algorithm for solving
  nonsymmetric linear-systems.
\newblock {\em SIAM J. Sci. Stat. Comput.}, 7:856--869, 1986.

\bibitem{serkh2016}
K.~Serkh and V.~Rokhlin.
\newblock On the solution of elliptic partial differential equations on regions
  with corners.
\newblock {\em J. Comput. Phys.}, 305:150--171, 2016.

\bibitem{vasilev1966}
E.~N. Vasil'ev and L.~B. materikova.
\newblock {Excitation of Dielectric Bodies of Revolution}.
\newblock {\em Soviet Physics - Technical Physics}, 10(10):1401--1406, 1966.

\bibitem{vico2018cpde}
F.~Vico, L.~Greengard, and M.~Ferrando.
\newblock Decoupled field integral equations for electromagnetic scattering
  from homogeneous penetrable obstacles.
\newblock {\em Commun. Part. Diff. Eq.}, 43(2):159--184, 2018.

\bibitem{yarvin1998}
N.~Yarvin and V.~Rokhlin.
\newblock {Generalized Gaussian Quadratures and Singular Value Decompositions
  of Integral Operators}.
\newblock {\em SIAM J. Sci. Comput.}, 20(2):699--718, 1998.

\bibitem{YHM2012}
P.~Young, S.~Hao, and P.~G. Martinsson.
\newblock A high-order {N}ystr\"om discretization scheme for boundary integral
  equations defined on rotationally symmetric surfaces.
\newblock {\em J. Comput. Phys.}, 231(11):4142--4159, 2012.

\bibitem{young2010axi}
P.~M. Young and P.-G. Martinsson.
\newblock {A Direct Solver for the Rapid Solution of Boundary Integral
  Equations on Axisymmetric Surfaces in Three Dimensions}.
\newblock 2010.
\newblock arXiv:1002.2001 [math.NA].

\bibitem{Yu2008Closed}
W.~M. Yu, D.~G. Fang, and T.~J. Cui.
\newblock {Closed Form Modal Green's Functions for Accelerated Computation of
  Bodies of Revolution}.
\newblock {\em IEEE Trans. Antennas Propag.}, 56(11):3452--3461, 2008.

\end{thebibliography}

\end{document}